\documentclass[11 pt, reqno]{amsart}

\numberwithin{equation}{section} \theoremstyle{plain}
\usepackage{amssymb,amsmath,amsfonts,mathrsfs}
\usepackage{appendix}
\usepackage{graphicx}
\usepackage{color}
\usepackage[colorlinks,citecolor=blue,filecolor=black,linkcolor=blue,urlcolor=black]{hyperref}

\newtheorem{theorem}{Theorem}[section]
\newtheorem{lemma}[theorem]{Lemma}
\newtheorem{cor}[theorem]{Corollary}
\newtheorem{prop}[theorem]{Proposition}

\theoremstyle{definition}
\newtheorem{definition}[theorem]{Definition}

\theoremstyle{remark}
\newtheorem{remark}[theorem]{Remark}

\numberwithin{equation}{section}

\newcommand{\loc}{\operatorname{loc}}
\newcommand{\comp}{\operatorname{comp}}

\newcommand{\lloc}{L_{\loc}}

\newcommand{\dom}{\operatorname{Dom}}
\newcommand{\End}{\operatorname{End}}

\newcommand{\supp}{\operatorname{supp}}

\newcommand{\IM}{\operatorname{Im}}
\newcommand{\RE}{\operatorname{Re}}

\newcommand\CC{\mathbb{C}}
\newcommand\RR{\mathbb{R}}
\newcommand\ZZ{\mathbb{Z}}

\newcommand{\mcomp}{C_{c}^{\infty}(M)}

\newcommand{\ric}{\operatorname{Ric}}

\newcommand{\Del}{\Delta}

\renewcommand{\div}{\operatorname{div}}
\newcommand{\grad}{\operatorname{grad}}

\def\pa{\partial}
\def\R{\mathbb R}

\newcommand{\vbn}{\mathcal{E}}
\newcommand{\vcomp}{C_{c}^{\infty}(\mathcal{E})}

\newcommand{\hvmaxp}{H_{p,\max}}
\newcommand{\hvminp}{H_{p, \min}}
\newcommand{\rank}{\operatorname{rank}}
\newcommand{\hess}{\operatorname{Hess}}

\newcommand{\ricc}{\operatorname{Ric}}
\newcommand{\vep}{\varepsilon}

\begin{document}

\title[Ornstein--Uhlenbeck Semigroups in $L^p$-spaces on Manifolds]{Generalized Ornstein--Uhlenbeck Semigroups in weighted $L^p$-spaces on Riemannian Manifolds}
\author{Ognjen Milatovic, Hemanth Saratchandran}
\address{Department of Mathematics and Statistics\\
         University of North Florida   \\
       Jacksonville, FL 32224 \\
        USA
           }
\email{omilatov@unf.edu}
\address{School of Mathematical Sciences \\
University of Adelaide \\
Adelaide \\
South Australia 5005 \\
Australia}
\email{hemanth.saratchandran@adelaide.edu.au}



\begin{abstract} Let $\mathcal{E}$ be a Hermitian vector bundle over a Riemannian manifold $M$ with metric $g$, let $\nabla$ be a metric covariant derivative on $\mathcal{E}$. We study the generalized Ornstein--Uhlenbeck differential expression $P^{\nabla}=\nabla^{\dagger}\nabla u+\nabla_{(d\phi)^{\sharp}}u-\nabla_{X}u+Vu$, where  $\nabla^{\dagger}$ is the formal adjoint of $\nabla$, $(d\phi)^{\sharp}$ is the vector field corresponding to $d\phi$ via $g$, $X$ is a smooth real vector field on $M$, and $V$ is a self-adjoint locally integrable section of the bundle $\textrm{End }\mathcal{E}$. We show that (the negative of) the maximal realization $-H_{p,\max}$ of $P^{\nabla}$ generates an analytic quasi-contractive semigroup in $L^p_{\mu}(\mathcal{E})$, $1<p<\infty$, where $d\mu=e^{-\phi}d\nu_{g}$, with $\nu_{g}$ being the volume measure.
Additionally, we describe a Feynman--Kac representation for the semigroup generated by 
$-H_{p,\max}$. For the Ornstein--Uhlenbeck differential expression acting on functions, that is, $P^{d}=\Delta u+(d\phi)^{\sharp}u-Xu+Vu$, where $\Delta$ is the (non-negative) scalar Laplacian on $M$ and $V$ is a locally integrable real-valued function, we consider another way of realizing $P^{d}$ as an operator in $L^p_{\mu}(M)$ and, by imposing certain geometric conditions on $M$, we prove another semigroup generation result. 
\end{abstract}

\maketitle

\section{Introduction}\label{S:intro} Owing to its importance in stochastic analysis, the \emph{generalized Ornstein--Uhlenbeck operator (with a potential $V$)},
\begin{equation}\label{E:intro-1}
P=-\Delta u+\grad\phi\cdot \grad u-X\cdot \grad u+Vu,
\end{equation}
where $\Delta$ and $\grad$ are the usual Laplacian and gradient operators on $\mathbb{R}^n$, $X\in C^{\infty}(\RR^n; \RR^n)$ is a vector field, and $V\in C^{\infty}(\RR^n)$, $\phi\in C^{\infty}(\RR^n)$ are real-valued functions, has been the object of investigation by many researchers. We point to the recent paper~\cite{bog-18} and numerous references therein for the information about the studies of various questions regarding the operator~(\ref{E:intro-1}) and, in particular, applications to stochastic analysis. One aspect of the study of~(\ref{E:intro-1}), sparking quite a bit of interest in the last two decades, involves formulating the conditions under which (a suitable realization of) the expression~(\ref{E:intro-1}) serves as the negative generator of an analytic $C_0$-semigroup in the weighted space $L^p_{\mu}:=L^p(\RR^n, d\mu)$, $1<p<\infty$, with $d\mu=e^{-\phi}\,dx$, where $dx$ is the usual Lebesgue measure. A prominent place in the progression of thought on this problem occupies the paper~\cite{MPR-05}, in which the authors established the mentioned generation property for the realization $R_1$ of~(\ref{E:intro-1}) with $V\equiv 0$, with the corresponding domain
\begin{equation*}
\dom(R_1)= W^{2,p}_{\mu}:=\{u\in W^{2,p}_{\loc}(\RR^n)\colon D^{\alpha}u\in L^p_{\mu}, |\alpha|\leq 2\},
\end{equation*}
where $W^{2,p}_{\loc}(\RR^n)$ is the usual local Sobolev space on $\RR^n$ with $2$ indicating the highest differential order, and $D^{\alpha}$ is the partial derivative of order $|\alpha|:=\alpha_1+\alpha_2+\dots+\alpha_n$ corresponding to the multi-index $\alpha=(\alpha_1,\alpha_2,\dots \alpha_n)$. Subsequently, the authors of~\cite{Ko-Yo-10} established the semigroup generation property for a realization $R_2$ of~(\ref{E:intro-1}) with the domain $\dom(R_2)=\{u\in W^{2,p}_{\mu}\colon Vu\in L^p_{\mu}\}$. A  milestone in the semigroup generation problem for~(\ref{E:intro-1}), resulting in a considerable improvement of the results of~\cite{Ko-Yo-10,MPR-05}, is the paper~\cite{So-Yo-13}, in which the authors used a direct approach based on a refinement of the integration by parts technique featured in~\cite{So-12} in the special case of~(\ref{E:intro-1}) with $X\equiv 0$. This approach enabled the authors of~\cite{So-Yo-13} to establish their semigroup generation result without using the condition $|\div X|\leq \varepsilon |\grad \phi|^2+ C_{\varepsilon}$, present in~\cite{Ko-Yo-10,MPR-05} as a way of guaranteeing the separation property  of the operator $S$ in $L^p(\RR^n)$ defined via $Su=e^{-\phi/p}Pe^{\phi/p}u$. (Here, the ``separation property" essentially means that $\dom(S)\subseteq W^{2,p}(\RR^n)$).

In this article we study the generation of analytic $C_0$-semigroups problem for the following geometric analogue of~(\ref{E:intro-1}):
\begin{equation}\label{E:intro-1-a}
P^{\nabla}u:=\nabla^{\dagger}\nabla u+\nabla_{(d\phi)^{\sharp}}u-\nabla_{X}u+Vu.
\end{equation}
Here,  $\nabla$ is a (smooth) metric covariant derivative on a Hermitian vector bundle $\vbn$ over a Riemannian manifold $M$, $\nabla^{\dagger}$ is the formal adjoint of $\nabla$ with respect to the usual inner product in $L^2(\vbn)$,  $\phi$ is a smooth real-valued function on $M$,  the symbol $(d\phi)^{\sharp}$ stands for the vector field corresponding to $d\phi$ via $g$, $X$ is a smooth \emph{real} vector field on $M$, and $V\in L^{\infty}_{\loc}(\End \vbn)$ such that $V(x)\colon\vbn_x\to\vbn_x$ is a self-adjoint operator for all $x\in M$. The operator $P^{\nabla}$ includes, as a special case with $\phi\equiv 0$, the covariant Schr\"odinger operator with a drift. The latter operator was studied in~\cite{Mil-19}, where it was shown, in analogy with the $\RR^n$-situation considered in~\cite{So-12}, that on a complete manifold $M$ and under suitable conditions on $X$ and $V$, (the negative of) the maximal realization $\hvmaxp$ of $P^{\nabla}$ (with $\phi\equiv 0$) in $L^p(\vbn)$, $1<p<\infty$, generates a contractive $C_0$-semigroup, and, furthermore, (again for $\phi\equiv 0$) the closure of the minimal realization $\overline{\hvminp}:=\overline{P^{\nabla}|_{\vcomp}}$ coincides with $\hvmaxp$ (see sections~\ref{SS:s-2-1}--\ref{SS:min-max-rel-1} for more details about the notations used thus far in this paragraph and the notion of maximal/minimal realization).

We are interested in seeing to what extent the semigroup generation results of~\cite{So-Yo-13} carry over to the operator $P^{\nabla}$, suitably realized in the weighted space $L^p_{\mu}(\vbn)$, $1<p<\infty$, with $d\mu:=e^{-\phi}\,d\nu_{g}$, where $\nu_{g}$ stands for the Riemannian volume measure on $M$. As the reader can see in the first part of our theorem~\ref{T:main-1}, if $M$ is geodesically complete and if $\phi$, $X$ and $V$ satisfy the conditions specified in (V1)--(V2) and (A1)--(A4) in section~\ref{assump_V_X_etc} below, which (in the case of an operator acting on scalar functions) are analogous to those in~\cite{So-Yo-13}, then $\hvmaxp=\overline{\hvminp}$ and $-\hvmaxp$ generates a quasi-contractive analytic $C_0$-semigroup in $L^p_{\mu}(\vbn)$ (for more details on the semigroup terminology see section~\ref{SS:accretive}). In the proof of this part we use the transformation $P^{\nabla}\mapsto e^{-\phi/p}P^{\nabla}e^{\phi/p}$ to reduce the  problem to the case of the covariant Schr\"odinger operator with a drift (acting in the usual space $L^p(\vbn)$), and, subsequently, we apply to the latter operator the generation result from~\cite{Mil-19} mentioned in the preceding paragraph.

The second part of theorem~\ref{T:main-1} features more stringent conditions on $\phi$, $X$ and $V$, described in (V1)--(V2) and (A1)--(A5), as well as the condition (A6) which imposes a requirement on the various constants present in (A1)--(A5). Assuming the geodesic completeness of $M$ and adapting the refined integration by parts technique from~\cite{So-Yo-13} to our setting, we obtain the coercive estimate (see theorem~\ref{T:main-1} for a more precise formulation)
\begin{equation}\label{E:intro-2}
\|hu\|_{p,\mu}\leq C\|(\lambda+\hvmaxp)u\|_{p,\mu},
\end{equation}
for all $u\in\dom(\hvmaxp)$ and all $\lambda\geq \lambda_1$, where $h$ is a minorizing (non-negative) function (of class $C^1$) for the potential $V$ (see condition (V2)), $C$ is a constant (independent of $u$), and  $\lambda_1>0$ is a sufficiently large real number.

The estimate~(\ref{E:intro-2}) crucially relies on the ``separation-type" condition (A5), which (at least in the form $\beta_3=0$) goes back to~\cite{Everitt-Giertz77} in the context of Schr\"odinger operators on $\RR^n$. Specializing for the moment to the case $\phi\equiv 0$ and $X\equiv 0$ in~(\ref{E:intro-1-a}), we say (in the terminology of~\cite{Everitt-Giertz77}) that the expression $\nabla^{\dagger}\nabla+V$ is \emph{separated} in $L^p(\vbn)$ if the following condition is met: for all $u\in\dom (\hvmaxp)$, we have $Vu\in L^p(\vbn)$, where we understand that $\hvmaxp$ now refers to the maximal realization of  $\nabla^{\dagger}\nabla+V$. Besides~\cite{Everitt-Giertz77}, the reader can see the papers~\cite{boi-2,BH-99} and references therein for more on the separation of Shro\"dinger-type operators on $\mathbb{R}^n$. Some results on the separation of the perturbations (by an operator of order 0) of the Laplacian (also biharmonic and quadharmonic operator) in $L^2$-space on a Riemannian manifold (with curvature bound assumptions in the biharmonic and quadharmonic case), can be found in the articles~\cite{Mil-18, Mil-Sar-19, Sarat-21}. Regarding the $L^p$-situation, in the earlier article~\cite{Mil-Sar-19} we proved the separation property for $\nabla^{\dagger}\nabla+V$ in $L^p(\vbn)$, $1<p<\infty$, by assuming (V1)--(V2) and (A5), the geodesic completeness of $M$, and lower boundedness (by a constant) of the Ricci curvature of $M$. In contrast to the procedure used in~\cite{Mil-Sar-19}, the refined integration by parts procedure used in the proof of~(\ref{E:intro-2}) does not rely on a sequence of (genuine) Laplacian cut-off functions (see section 3.1 in~\cite{Gue-book}), which is known to exist if one assumes lower bounded Ricci curvature or, more generally, as in~\cite{BS-18}, if one assumes that the Ricci curvature is bounded from below by a (possibly unbounded) non-positive function depending on the distance from a reference point. Thus, apart from being instrumental in latter parts of the paper, the coercive estimate~(\ref{E:intro-2}) paves the way to the separation property for $\nabla^{\dagger}\nabla+V$ in $L^p(\vbn)$, $1<p<\infty$, without any assumptions on $M$ other than the geodesic completeness (see corollary~\ref{C:main-1}). In the case $p \neq 2$, this result represents the first time a separation result for such operators on non-compact Riemannian manifolds has been obtained without any curvature bound assumptions.

In theorem~\ref{T:main-2} we consider the following version of~(\ref{E:intro-1-a}):
\begin{equation}\label{E:intro-3}
P^{d}u:=\Delta u+(d\phi)^{\sharp}u-Xu+Vu,
\end{equation}
where $\Delta$ is the (non-negative) scalar Laplacian, $\phi$ is a smooth real-valued function on $M$, $X$ is a smooth vector field, $V$ is a \emph{real-valued} function of the class $C^1(M)$, and $(d\phi)^{\sharp}u$ and $X u$ indicate the actions of the vector-fields $(d\phi)^{\sharp}$ and $X$ on the function $u$. We define a realization $H^{d}$ of $P^{d}$ in $L^p_{\mu}(M)$, where $d\mu=e^{-\phi}\,d\nu_{g}$ (with $d\nu_{g}$ being the Riemannian volume measure), as follows: $H^{d}u:=P^{d}u$, for all $u\in\dom(H^{d})$, where
\begin{equation*}
\dom (H^{d})=\{u\in L^p_{\mu}(M)\cap W^{2,p}_{\loc}(M)\colon  P_{j}u\in L^p_{\mu}(M),\,\, j=1,\dots, 4\}
\end{equation*}
and
\begin{equation*}
P_1u:=\Delta u,\,\, P_2u:=(d\phi)^{\sharp}u,\,\,  P_3u:=Xu,\,\, P_{4}u:=Vu.
\end{equation*}
In part (i) of theorem~\ref{T:main-2} we show that $H^{d}$ generates a quasi-contractive analytic $C_0$-semigroup in $L^p_{\mu}(M)$. The proof rests on the coercive estimates~(\ref{more_coer_est_1_1}) and~(\ref{more_coer_est_1_2}) below, which (if one looks at the definition of $\dom (H^{d})$) play a ``separating" role by targeting (in analogy with how the estimate~(\ref{E:intro-2}) separates out $hu$) the items $(d\phi)^{\sharp}u$ and $\hess u$, where $\hess u$ refers to the Hessian of $u$ (defined in section~\ref{SS:hess}). The reader will note that, unlike~(\ref{E:intro-2}), the estimates (\ref{more_coer_est_1_1}) and~(\ref{more_coer_est_1_2}) are stated for $u\in \mcomp$, which, alongside the equality $\dom(H^{d})= W^{2,p}_{\mu}(M)$ claimed in part (ii) of theorem~\ref{T:main-2}, hints that we are likely to need a density result of $\mcomp$ in the  weighted Sobolev space
\begin{equation}\label{E:intro-4}
W^{2,p}_{\mu}(M):=\{u\in W^{2,p}_{\loc}(M)\colon u\in L^p_{\mu},\, du\in L^p_{\mu},\,\textrm{ and } \hess u\in L^p_{\mu}\}.
\end{equation}
To our knowledge, when it comes to $L^p$-Sobolev spaces of order $k\geq 2$, $1<p<\infty$, the density of $\mcomp$ is guaranteed under more stringent conditions than the geodesic completeness of $M$; see the recent papers~\cite{HMRV-2021, IRV-19, IRV-20, M-V-21-CH} for up-to-date results on this problem. As discussed in lemma~\ref{main-2_density} below, in the case of our space $W^{2,p}_{\mu}(M)$  we require the existence of weak Hessian cut-off functions (described in remark~\ref{R:HYP-H}). In fact, there are two additional forces in play in the proof of theorem~\ref{T:main-2}: (i) the domination estimate of lemma~\ref{more_coer_est_3}, crucially important for~(\ref{more_coer_est_1_1}) and~(\ref{more_coer_est_1_2}), which says that if $U\geq c_0>0$ is a function satisfying (A5), then for all (sufficiently small) $\vep>0$ there exists a constant $\widetilde{C}_{\vep}>0$ such that
\begin{equation}\label{E:intro-4-a}
\|Udv\|_{p}\leq \vep \|\Delta v\|_{p}+\widetilde{C}_{\vep}\|Uv\|_{p},
\end{equation}
for all $v\in\mcomp$, where $\|\cdot\|_{p}$ is the usual $L^p$-norm on $M$;  (ii) the $L^p$-Calder\'on--Zygmund inequality, which says (see section~\ref{SS:c-z} for a precise formulation and for pointers to some recent studies of this problem) that there exists a constant $C>0$, such that
\begin{equation}\tag{CZ(p)}
\|\hess u\|_{p}\leq C(\|\Delta u\|_{p}+\|u\|_{p}),
\end{equation}
for all $u\in\mcomp$. It turns out that the mentioned density, the inequality (CZ(p)), the domination property~(\ref{E:intro-4-a}), the coercive estimates~(\ref{more_coer_est_1_1}) and~(\ref{more_coer_est_1_2}) become accessible to us (see remark~\ref{R:HYP-H} for further discussion) if $M$ has a positive injectivity radius and bounded Ricci curvature. To get to the key domination estimate~(\ref{E:intro-4-a}), we first establish its analogue (see lemma~\ref{more_coer_est_2}) on a Euclidean ball in $\mathbb{R}^n$ by using certain properties from~\cite{MPR-05}. After that, we globalize the mentioned local estimate with the help of harmonic coordinates (recalled in appendix~A), whose feasibility rests on the property~(\ref{E:har-rad-lbound}), which holds true in the setting of positive injectivity radius and bounded Ricci curvature.

Going back to the expression $P^{\nabla}$ in~(\ref{E:intro-1-a}) and assuming $V\geq 0$, the geodesic completeness of $M$, (A1)--(A2), and (A3)--(A4) with $h$ replaced by zero (for conditions (Aj), $j=1,2,3,4$, we refer to section~\ref{assump_V_X_etc}), in theorem~\ref{T:main-3} we obtain a Feynman--Kac representation of the (quasi-contractive $C_0$-semigroup) $(e^{-t\hvmaxp})_{t\geq 0}$ in $L_{\mu}^p(\vbn)$. This representation is very similar to the one for the case of the covariant Schr\"odinger operator $\nabla^{\dagger}\nabla+V$ in $L^2(\vbn)$ with a potential $0\leq V\in\lloc^2(\End\vbn)$ studied in~\cite{Gu-10}, with the difference occurring in the underlying diffusion (solution to a suitable stochastic differential equation)---in the setting of~\cite{Gu-10} this is the usual Brownian motion on $M$, while in our case the corresponding diffusion takes into account the presence of the vector  fields $X$ and $(d\phi)^{\sharp}$. We refer the reader to the section~\ref{SS:FK-OU} below, where we describe the stochastic analysis setting needed for understanding the statement of theorem~\ref{T:main-3}. Lastly, we mention that in the recent article~\cite{Boldt-Guneysu-20}, the authors gave a probabilistic representation  for the semigroup in $L^2(\vbn)$ corresponding to a (closed sectorial) realization of the expression $\nabla^{\dagger}\nabla +F$ in $L^2(\vbn)$, where $F\colon C^{\infty}(\vbn)\to C^{\infty}(\vbn)$ is an arbitrary differential operator of order $\leq 1$ with not necessarily ``real coefficients." We point to section~\ref{SS:FK-OU} below for more details about the result of~\cite{Boldt-Guneysu-20} and its relationship to our context.

The paper contains nine sections and an appendix. In section~\ref{S:res} we describe the main objects used in this paper: the geometric setting (including the basic differential operators), the (weighted) $L^p$-Sobolev spaces, and the generalized Ornstein--Uhlenbeck differential expression and its operator realizations in the weighted $L^p$-space. Additionally, section~\ref{S:res} contains the statements of the main results: for the first semigroup generation result see subsection~\ref{SS:state-res-1}, for the separation result see subsection~\ref{SS:sep-eg}, for the  second semigroup generation result see subsection~\ref{SS:state-res-2}, and for a description of the stochastic analysis preliminaries and the statement of the Feynman--Kac formula see subsection~\ref{SS:FK-OU}. In section~\ref{S:prelim-1} we collect some basic calculus facts (various product and chain rules), several elementary inequalities, and the statement of the $L^p$-Calder\'on--Zygmund inequality. We end section~\ref{S:prelim-1} with a formula that serves as the roadmap for the integration by parts process for subsequent parts of the paper. After recalling some abstract terminology for the semigroups on Banach spaces, in section~\ref{S:hmax-1} we prove the first semigroup generation result (part (i) of theorem~\ref{T:main-1}). In section~\ref{S:coercive-1}, we prove the first coercive estimate of this paper (part (ii) of theorem~\ref{T:main-1}). Using this coercive estimate, in section~\ref{S:proof-cor-1} we establish the separation result for the covariant Schr\"odinger operator, corollary~\ref{C:main-1}. In Section~\ref{More Coerc} we prove two additional coercive estimates and a domination property, all of which are crucially important for section~\ref{S:pf-thm-2}, where we establish the second semigroup generation result, theorem~\ref{T:main-2}. In section~\ref{SS:FK-proof-1} we prove theorem~\ref{T:main-3}, which features the mentioned Feynman--Kac formula for the semigroup corresponding to the maximal realization of $P^{\nabla}$ in $L^p_{\mu}(\vbn)$. Lastly, in the appendix~A we recall the basics of harmonic coordinates, which we use in our analysis in section~\ref{More Coerc}.

\section{Notations and Results}\label{S:res}
\subsection{Basic Notations}\label{SS:s-2-1} In this article we work on a smooth connected Riemannian $n$-manifold $(M,g)$ without boundary.  We use $d\nu_{g}$ to indicate the usual volume measure on $M$: in local coordinates $x^1,x^2,\dots, x^n$, we have $d\nu_{g}=\sqrt{\det (g^{ij})}\,dx$, where $(g^{ij})$ is the inverse of the matrix $g=(g_{ij})$ and $dx=dx^1\,dx^2\dots dx^n$ is the Lebesgue measure. For $x_0\in M$ and $\rho>0$ we define
\begin{equation}\label{E:ball-r}
B_{\rho}(x_0):=\{x\in M\colon r_{g}(x)<\rho\},
\end{equation}
where $r_{g}(x):=d_{g}(x_0,x)$ and $d_{g}$ is the distance function corresponding to $g$.

We denote the the tangent and cotangent bundles of $M$ by $TM$ and $T^*M$ respectively. Let $T^{k,l}(M)$ stand for the $(k,l)$-tensor bundle over $M$, with $k$ indicating the contravariant and $l$ the covariant index. The metric $g=(g_{ij})$ on $TM$ gives rise (in the usual way) to the Euclidean structure on $T^*M$, and this, in turn, leads to the Euclidean structure on the bundles $T^{k,l}(M)$.

The symbol $\vbn \to M$ stands for a smooth Hermitian vector bundle over $M$ equipped with a Hermitian structure $\langle\cdot, \cdot\rangle_{x}$, linear in the first and antilinear in the second variable. We use $|\cdot|_{x}$ to indicate the fiberwise norms on $\vbn_{x}$, usually omitting $x$ to simplify the notations. The notations $C^{\infty}(\vbn)$ and $\vcomp$ indicate smooth sections of $\vbn$ and smooth compactly supported sections of $\vbn$, respectively. For the corresponding spaces of complex-valued functions on $M$, we use the symbols $C^{\infty}(M)$ and $\mcomp$. Additionally, $C^1(M)$ indicates continuously differentiable (complex-valued) functions on $M$.

Furthermore, we denote the smooth (complex) one-forms on $M$ by $\Omega^1(M)$ and their compactly supported counterparts by $\Omega^1_{c}(M)$. 
With the Einstein summation convention understood, the notation $\langle \kappa, \omega\rangle_{1,x}$, $x\in M$, where $\kappa=\kappa_jdx^j$ and $\omega=\omega_jdx^j$, is interpreted as
\begin{equation}\label{E:form-inner-x}
\langle \kappa, \omega\rangle_{1,x}:=g^{jk}\kappa_j(x)\omega_k(x),
\end{equation}
where  $(g^{jk})$ is the inverse of the matrix $g=(g_{jk})$. To make the notation simpler, we usually omit the subscript ``$x$" from $\langle \kappa, \omega\rangle_{1,x}$.

We will also use basic ``musical" isomorphisms coming from $g$: for a vector field $Y$ on $M$, the symbol $Y^{\flat}$ indicates the one-form associated to $Y$, while $\omega^{\sharp}$ refers to the vector field associated to the one-form $\omega$.

\subsection{Measure $d\mu$} Using the volume measure $d\nu_{g}$ and a real-valued function $\phi\in C^{\infty}(M)$, we define the measure $d\mu$ on $M$ as follows:
\begin{equation}\label{E:meas-mu}
d\mu(x):=e^{-\phi(x)}d\nu_{g}(x).
\end{equation}

\subsection{Weighted $L^p$-spaces} For $1\leq p<\infty$, the notation $L_{\mu}^p(\vbn)$ refers to the space of $p$-integrable complex-valued sections of $\vbn$ with the norm
\begin{equation}\label{E:lp-norm}
\|u\|^{p}_{p,\mu}:=\int_{M}|u(x)|^p\,d\mu,
\end{equation}
where $\mu$ is as in~(\ref{E:meas-mu}) and $|\cdot|$ is the fiberwise norm in $\vbn_{x}$.

In the special case $d\mu=d\nu_{g}$, instead of $L_{\mu}^p(\vbn)$, we write $L^p(\vbn)$ and we denote the local versions of these spaces by $\lloc^{p}(\vbn)$. Lastly, the symbol $L_{\mu}^{\infty}(\vbn)$ stands for the space of essentially bounded sections of $\vbn$, and $\lloc^{\infty}(\vbn)$ indicates the corresponding local version.

In the case $p=2$ we get a Hilbert space $L_{\mu}^2(\vbn)$ with the inner product

\begin{equation}\label{E:inner-mu}
(u,v)_{\mu}=\int_{M}\langle u(x),v(x)\rangle\,d\mu.
\end{equation}
In the special case $d\mu=d\nu_{g}$, instead of writing  $(\cdot,\cdot)_{\nu_{g}}$ and $\|\cdot\|_{p,\nu_{g}}$, we write $(\cdot,\cdot)$ and $\|\cdot\|_{p}$ respectively.

For the $L^p_{\mu}$-spaces of functions on $M$ we use the same symbols as for the corresponding spaces of sections of $\vbn$; we just replace $\vbn$ by $M$, fiberwise norm by the absolute value, and $\langle u(x),v(x)\rangle$ by $u(x)\overline{v(x)}$, where $\overline{z}$ indicates the conjugate of a complex number $z$.

\subsection{Generalized Ornstein--Uhlenbeck Operator}\label{SS:ornst-uhl-def}
Having introduced the basic function spaces, we turn to differential operators.
The first item is the operator  $\nabla\colon C^{\infty}(\vbn)\to C^{\infty}(T^*M\otimes\vbn)$, a (smooth) metric covariant derivative on $\vbn$. The next item is $\nabla^{\dagger}\colon C^{\infty}(T^*M\otimes\vbn)\to C^{\infty}(\vbn)$, which stands for the formal adjoint of $\nabla$ with respect to $(\cdot,\cdot)$, the inner product in $L^2(\vbn)$. The so-called \emph{Bochner Laplacian} is defined as the composition $\nabla^{\dagger}\nabla$. Specializing to functions, we have the usual differential $d\colon C^{\infty}(M)\to \Omega^1(M)$ and its formal adjoint $d^{\dagger}\colon \Omega^1(M)\to C^{\infty}(M)$, understood with respect to the inner product $(\cdot,\cdot)$ in $L^2(M)$. The composition $d^{\dagger}d$, which we denote by $\Delta$, is known as \emph{the scalar Laplacian} on $M$. We note that in our article $\nabla^{\dagger}\nabla$ and $\Delta$ are non-negative operators. For a smooth vector field $Y$, we define the divergence of $Y$ as $\div Y:=-d^{\dagger}(Y^{\flat})$.

With the basic notations at hand, we are ready to describe the main object of this paper. Let $\phi$ be as in~(\ref{E:meas-mu}), let $X$ be a smooth \emph{real} vector field on $M$ and let $V\in L^{\infty}_{\loc}(\End \vbn)$ such that $V(x)\colon\vbn_x\to\vbn_x$ is a self-adjoint operator for all $x\in M$. We consider the expression
\begin{equation}\label{E:def-H}
P^{\nabla}u:=\nabla^{\dagger}\nabla u+\nabla_{(d\phi)^{\sharp}}u-\nabla_{X}u+Vu.
\end{equation}
We call the operator $P^{\nabla}$ \emph{the generalized Ornstein--Uhlenbeck operator (with a potential $V$)}. Following the terminology of~\cite{Gue-book}, in the case $X\equiv 0$ and $\phi\equiv 0$ the operator $P^{\nabla}$ reduces to the so-called \emph{covariant Schr\"odinger operator}. The usual Ornstein--Uhlenbeck operator on $\mathbb{R}^n$ is a special case of $P$ in~(\ref{E:intro-1}) with $\phi\equiv 0$, $V\equiv 0$, and $X(x)=-x$, $x\in\mathbb{R}^n$.

\subsection{Minimal and Maximal Realization}\label{SS:min-max-rel-1}
We now explain two simple ways in which the expression~(\ref{E:def-H}) can be realized as an operator in $L^p_{\mu}(\vbn)$,  $1<p<\infty$. We define $\hvminp u:= P^{\nabla}u$ with $\dom(\hvminp)=\vcomp$ and $\hvmaxp u:=P^{\nabla}u$ with
\[
\dom(\hvmaxp):=\{u\in L^p_{\mu}(\vbn)\colon P^{\nabla}u\in L^p_{\mu}(\vbn)\},
\]
understanding the expressions $\nabla^{\dagger}\nabla u$, $\nabla_{X}u$, and $\nabla_{(d\phi)^{\sharp}}u$ in the sense of distributions.

\subsection{Hessian Operator}\label{SS:hess} We take a moment to review the definition of the Hessian operator, which plays an important role in various places in this paper. To start off, recall that the Levi--Civita connection $\nabla^{lc}$ on $M$ induces the Euclidean covariant derivative on $T^*M=T^{0,1}$, also denoted by $\nabla^{lc}$, as follows:
\begin{equation}\label{E:cov-lc}
(\nabla^{lc}_{Y_1}\omega)(Y_2):=Y_1(\omega(Y_2))-\omega(\nabla^{lc}_{Y_1}Y_2),
\end{equation}
for all smooth one-forms $\omega$ and smooth vector fields $Y_1$ and $Y_2$.

Denoting the smooth sections of the bundle $T^{k,l}(M)$ by $C^{\infty}(T^{k,l}(M))$ and using the covariant derivative $\nabla^{lc}$ in~(\ref{E:cov-lc}), we define the \emph{Hessian operator} $\hess\colon C^{\infty}(M)\to C^{\infty}(T^{0,2}(M))$ as
\begin{equation}\label{E:hess}
\hess u:=\nabla^{lc}du.
\end{equation}

\subsection{Assumptions on $V$, $X$, and $\phi$.}\label{assump_V_X_etc}
Here we collect various hypotheses on $V$, $X$, and $\phi$, remarking that some results will hold, as explained in the corresponding statements, under a subset of these hypotheses.

We assume that $V$ satisfies the following conditions:
\begin{enumerate}
  \item [(V1)] $V\in L^{\infty}_{\loc}(\End \vbn)$ and $V(x)\colon\vbn_x\to \vbn_x$ is self-adjoint for all $x\in M$;

  \item [(V2)] there exists a number $\zeta\geq 1$ and a non-negative function $h\in C^{1}(M)$ such that
  \begin{equation*}
 h(x)\leq V(x)\leq \zeta h(x),
  \end{equation*}
for all $x\in M$, where the inequalities are understood in the sense of quadratic forms on $\vbn_{x}$.
\end{enumerate}

Furthermore, we make the following assumptions on the function $\phi$ from~(\ref{E:meas-mu}), the vector field $X$, and the function $h$ mentioned in (V2):

\begin{enumerate}
  \item [(A1)] $\phi\in C^{\infty}(M)$ is a real-valued function and $X$ is a smooth real vector field on $M$;
  \item [(A2)] for every $\varepsilon>0$, there exists $C_{\varepsilon}>0$ such that $|(\hess \phi)(x)|\leq \varepsilon |d\phi(x)|^2+C_{\varepsilon}$, for all $x\in M$;
  \item [(A3)] there exist $\theta\in\mathbb{R}$, $\beta_1\in\mathbb{R}$ such that $(\div X)(x)-(X\phi)(x)+\theta h(x)\geq-\beta_1$, for all $x\in M$;
  \item [(A4)] there exist $\kappa\geq 0$ and $\beta_2\geq 0$ such that $|X(x)|\leq \kappa(|d\phi(x)|^2+h(x)+\beta_2)^{1/2}$, for all $x\in M$;
  \item [(A5)] there exist $\gamma> 0$ and $\beta_3\geq 0$ such that $|dh(x)|\leq \gamma(h(x))^{3/2}+\beta_3$, for all $x\in M$;
  \item [(A6)] the numbers $\theta$, $\gamma$, $\kappa$ and the number $p$ appearing in $L^{p}_{\mu}(M)$ satisfy the following inequality:
    \begin{equation}\label{E:ineq-hyp-main-1}
    \frac{\theta}{p}+(p-1)\gamma\left(\frac{\kappa}{p}+\frac{\gamma}{4}\right)<1.
    \end{equation}
    \end{enumerate}
Here, the notations $|d\phi(x)|$ and $|(\hess \phi)(x)|$ indicate the fiberwise norms with respect to the Euclidean structures on the bundles $T^*M$ and $T^{0,2}(M)$ induced by $g$, as discussed in section~\ref{SS:s-2-1} above. Additionally, the symbol $\div X$ in (A3) refers to the divergence of $X$, and $X\phi$ indicates the action of $X$ on $\phi$, that is, $X\phi=(d\phi)(X)$.

\subsection{Statement of the First Semigroup Generation Result}\label{SS:state-res-1}
For future reference, we record that $\overline{T}$ denotes the closure of a (closable) operator $T$ in $L_{\mu}^{p}(\vbn)$. We are now ready to state our first generation result.
\begin{theorem} \label{T:main-1} Assume that $M$ is a geodesically complete $n$-dimensional Riemannian manifold without boundary. Let $\vbn$ be a Hermitian vector bundle over $M$ equipped with a metric connection $\nabla$. Let $1<p<\infty$ and let $d\mu$ be as in~(\ref{E:meas-mu}).  Then, the following properties hold:
\begin{enumerate}
  \item [(i)] Assume that  $V$ satisfies (V1) and the first inequality in (V2). Additionally, assume that (A1)--(A4) are satisfied together with the condition $\theta<p$. Then, we have $\overline\hvminp=\hvmaxp$. Furthermore, $-\hvmaxp$ generates a quasi-contractive analytic  $C_0$-semigroup on $L_{\mu}^p(\vbn)$.
  \item [(ii)] Assume that  $V$ satisfies (V1) and (V2). Furthermore, assume that (A1)--(A6) are satisfied. Then, there exists a number $\lambda_1>0$ such that
  \begin{equation}\label{E:coercive-1}
[1-p^{-1}\theta-(p-1)\gamma(p^{-1}\kappa+4^{-1}\gamma)]\|hu\|_{p,\mu}\leq (1+p\kappa\gamma)\|(\hvmaxp +\lambda)u\|_{p,\mu},
\end{equation}
for all $u\in\dom (\hvmaxp)$ and $\lambda\geq\lambda_1$, where $\|\cdot\|_{\mu, p}$ is as in~(\ref{E:lp-norm}) and $h$, $\kappa$, $\theta$, $\gamma$  are as in (A1)--(A6).
\end{enumerate}
\end{theorem}

\subsection{Separation Property for the Covariant Schr\"odinger Operator}\label{SS:sep-eg}
In this section we consider the expression~(\ref{E:def-H}) with $X=0$ and $\phi=0$, which makes $d\mu=d\nu_g$, the volume measure. In general, it is not true that for all $u\in\dom (\hvmaxp)$ we have $\nabla^{\dagger}\nabla u\in L^p(\vbn)$ and $Vu\in L^p(\vbn)$. (Here, $\hvmaxp$ is as in section~\ref{SS:min-max-rel-1} with $X=0$ and $\phi=0$.)

Adopting the term used in~\cite{Everitt-Giertz77}, we call the expression $\nabla^{\dagger}\nabla+V$ \emph{separated} in $L^p(\vbn)$ if the following condition is met: for all $u\in\dom (\hvmaxp)$, we have $Vu\in L^p(\vbn)$.

Assuming (V1)--(V2) and (A5), the geodesic completeness of $M$, and lower boundedness (by a constant) of the Ricci curvature of $M$, the separation property for $\nabla^{\dagger}\nabla+V$ in $L^p(\vbn)$, $1<p<\infty$, was established in~\cite{Mil-Sar-19}. The integration by parts procedure leading to the key coercive estimate in~\cite{Mil-Sar-19} required the existence of a sequence of (genuine) Laplacian cut-off functions (for a description of such sequence see section 3.1 in~\cite{Gue-book}), which is known to exist under the aforementioned assumption on the Ricci curvature or, more generally, as in~\cite{BS-18}, assuming that the Ricci curvature is bounded from below by a (possibly unbounded) non-positive function depending on the distance from a reference point.

In the following corollary, we establish the separation property of $\nabla^{\dagger}\nabla+V$ in $L^p(\vbn)$, $1<p<\infty$, assuming, in addition to (V1)--(V2) and (A5), just the geodesic completeness of $M$:

\begin{cor} \label{C:main-1} Let $M$, $\vbn$, $\nabla$ be as in theorem~\ref{T:main-1}, and let $1<p<\infty$.  Assume that $V$ satisfies the hypotheses (V1) and (V2). Furthermore, assume that the function $h$ (mentioned in (V2)) satisfies (A5) with $0<\gamma<2/\sqrt{p-1}$. Then, the expression $\nabla^{\dagger}\nabla+V$ is separated in $L^p(\vbn)$.
\end{cor}

In our second semigroup generation result, we consider the special case of $P^{\nabla}$ acting on functions:
\begin{equation}\label{E:def-H-d}
P^{d}u:=\Delta u+(d\phi)^{\sharp}u-Xu+Vu,
\end{equation}
where $\Delta$ is the (non-negative) scalar Laplacian, $\phi$ is as in~(\ref{E:meas-mu}), $X$ is a smooth vector field, $V$ is a \emph{real-valued} function of the class $C^1(M)$, the symbol $(d\phi)^{\sharp}$ stands for the vector field corresponding to $d\phi$ via $g$, and $(d\phi)^{\sharp}u$ and $X u$ indicate the actions of the vector-fields $(d\phi)^{\sharp}$ and $X$ on the function $u$.

To state the second semigroup generation result we need a suitable class of second-order weighted Sobolev spaces on $M$.

\subsection{Weighted Second-Order Sobolev Spaces on $M$}\label{SS:w-sob}

In the symbol $W^{k,p}_{\loc}(M)$, by which we denote the \emph{local Sobolev spaces} of functions on $M$, the numbers $k\in\{0,1,2,\dots\}$ and $1\leq p\leq \infty$ indicate the highest order of derivatives and the corresponding $L^{p}_{\loc}$-space, respectively. The space of compactly supported elements of $W^{k,p}_{\loc}(M)$ will be indicated by $W^{k,p}_{\comp}(M)$. For reference, we remark that for the corresponding local spaces of sections of $\vbn$ we use the same notations, with $M$ replaced by $\vbn$.

We now define a class of the second-order weighted Sobolev spaces (of functions) on $M$. For $1<p<\infty$ and $d\mu$ as in~(\ref{E:meas-mu}) we define
\begin{equation}\label{E:sob-mu}
W^{2,p}_{\mu}(M):=\{u\in W^{2,p}_{\loc}(M)\colon u\in L^p_{\mu},\, du\in L^p_{\mu},\,\textrm{ and } \hess u\in L^p_{\mu}\}.
\end{equation}
To make the notations simpler, we suppressed the bundle designations in the corresponding $L^{p}_{\mu}$ space. Looking at~(\ref{E:sob-mu}), $u$ is a function, $du$ is a section of $T^{0,1}(M)=T^*M$, and $\hess u$ is a section of $T^{0,2}(M)$. The $L^p_{\mu}$-norms of the last two items, are defined as in~(\ref{E:lp-norm}) with the absolute value $|\cdot|$ replaced by the fiberwise norm $|\cdot|$ coming from the Euclidean structure (as discussed in section~\ref{SS:s-2-1}) of the corresponding bundle $T^*M$ or $T^{0,2}(M)$.

\subsection{A Realization $H^{d}$ of the expression $P^{d}$}\label{SS:operator-H}
Here we describe another way to realize $P^{d}$ in~(\ref{E:def-H-d}) as an operator in $L^p_{\mu}(M)$, $1<p<\infty$.
Let $V\in C^1(M)$ be a real-valued function, let $\phi$ be as in~(\ref{E:meas-mu}), and let $X$ be a smooth real vector field on $M$. We define an operator $H^{d}$ in $L^p_{\mu}(M)$ as follows: $H^{d}u:=P^{d}u$, for all $u\in\dom(H^{d})$, where
\begin{equation}\label{dom-H}
\dom (H^{d})=\{u\in L^p_{\mu}(M)\cap W^{2,p}_{\loc}(M)\colon  P_{j}u\in L^p_{\mu}(M),\,\, j=1,\dots, 4\}
\end{equation}
and
\begin{equation*}
P_1u:=\Delta u,\,\, P_2u:=(d\phi)^{\sharp}u,\,\,  P_3u:=Xu,\,\, P_{4}u:=Vu.
\end{equation*}
Here, $(d\phi)^{\sharp}$ is the vector field corresponding to $d\phi$ via $g$.
\subsection{Statement of the Second Semigroup Generation Result}\label{SS:state-res-2}
Before stating the result, we indicate that $\ricc_{M}$ and $r_{\textrm{inj}}(M)$ stand for the Ricci curvature tensor of $M$ and the injectivity radius of $M$, respectively.

\begin{theorem}\label{T:main-2} Let $M$ be a Riemannian manifold without boundary. Assume that $r_{\textrm{inj}}(M)>0$ and $\|\ricc_{M}\|_{\infty}<\infty$. Let $1<p<\infty$ and let $d\mu$ be as in~(\ref{E:meas-mu}). Assume that $V$ is a real-valued function of class $C^1(M)$. Additionally, assume that the vector field $X$ and the function $\phi$ from the expression~(\ref{E:def-H-d}) satisfy the hypotheses (A1)--(A6), with $h$ replaced by $V$ in (A3)--(A5). Then, the following properties hold:
\begin{itemize}
  \item [(i)] The operator $-H^{d}$, with $H^{d}$ as in section~\ref{SS:operator-H}, generates a quasi-contractive analytic $C_0$-semigroup in $L_{\mu}^p(M)$;
  \item [(ii)] $\dom(H^{d})=\{u\in W^{2,p}_{\mu}(M)\colon Vu\in L_{\mu}^p(M)\}$, where $W^{2,p}_{\mu}(M)$ is as in~(\ref{E:sob-mu}).
\end{itemize}
\end{theorem}
\begin{remark}\label{R:HYP-H}
In various places in sections~\ref{More Coerc} and~\ref{S:pf-thm-2}, which culminate in the proof of theorem~\ref{T:main-2}, we impose the following conditions on $M$:
\begin{itemize}
  \item [(H1)] $(M, g)$ is a complete Riemannian manifold.

\item [(H2)] $(M,g)$ satisfies the $L^p$-Calder\'on--Zygmund inequality for all $1<p<\infty$. (For a precise formulation of this condition, see the inequality (CZ(p)) in section~\ref{SS:c-z} below.)

\item [(H3)] $(M,g)$ admits a sequence of \emph{weak Hessian cut-off functions}. (The latter, in the terminology of~\cite{IRV-19}, means that there exists a sequence $\psi_k\in\mcomp$ such that (i) $0\leq \psi_k(x)\leq 1$ for all $k\in\ZZ_{+}$ and all $x\in M$; (ii) for every compact set $K\subset M$ there exists $k_0\in\ZZ_{+}$ such that for all $k\geq k_0$ we have $\psi_k|_{K}=1$; furthermore, there exist constants $C_1$, $C_2$ such that for all $k\in\ZZ_{+}$ we have the properties (iii) $\|d\psi_k\|_{\infty}\leq C_1$ and  (iv) $\|\hess \psi_k\|_{\infty}\leq C_2$.)
\end{itemize}
It turns out that the assumptions $\|\ricc_{M}\|_{\infty}<\infty$ and $r_{\textrm{inj}}(M)>0$ guarantee that $(M,g)$ satisfies (H1), (H2) and (H3). In particular, the fulfillment of the condition (H2) is ensured by theorem 4.11 in~\cite{GP-2015}. Furthermore (H3) is fulfilled because $\|\ricc_{M}\|_{\infty}<\infty$ and $r_{\textrm{inj}}(M)>0$ guarantee (see the discussion after theorem 7.3 in~\cite{Pigola-20}) that $M$ has a sequence of \emph{genuine Hessian cut-off functions}. The latter concept is more stringent than \emph{weak Hessian cut-off functions} in that the conditions (iii) and (iv) in (H3) are replaced, respectively,  by (iii') $\|d\psi_k\|_{\infty}\to 0$ as $k\to\infty$ and (iv') $\|\hess \psi_k\|_{\infty}\to 0$ as $k\to\infty$.

As the reader can see in lemma~\ref{main-2_density}, the condition (H3) ensures the density of $\mcomp$ in $W^{2,p}_{\mu}(M)$. Lastly, we mention that the assumptions $\|\ricc_{M}\|_{\infty}<\infty$ and $r_{\textrm{inj}}(M)>0$ play a decisive role in establishing the domination estimate of lemma~\ref{more_coer_est_3}, which we get by globalizing the local estimate~(\ref{E:eucl-est-1}) via harmonic coordinates (recalled in appendix~A), whose feasibility is based on the property~(\ref{E:har-rad-lbound}), which is satisfied (see theorem B.4 in~\cite{GP-2015}) if $\|\ricc_{M}\|_{\infty}<\infty$ and $r_{\textrm{inj}}(M)>0$.
\end{remark}
\subsection{Feynman--Kac Formula for Generalized Ornstein--Uhlenbeck Operators}\label{SS:FK-OU}
In this section $(M,g)$ is a geodesically complete Riemannian manifold. Furthermore, $\vbn$ is a Hermitian vector bundle over $M$ with Hermitian structure $\langle\cdot,\cdot\rangle_{\vbn_{x}}$ and a metric covariant derivative $\nabla$. We consider the expression $P^{\nabla}$ as in~(\ref{E:def-H}) with the following assumptions on $V$, $X$, $\phi$:

\begin{itemize}
  \item [(F1)] $V\in \lloc^{\infty}(\End\vbn)$ and $V\geq 0$;
  \item [(F2)] $\phi$ and $X$ satisfy the assumptions (A1) and (A2);
  \item [(F3)] $\phi$ and $X$ satisfy the assumptions (A3) and (A4) with $h$ replaced by $0$.
\end{itemize}

Let $1<p<\infty$ and let $\hvmaxp$ be as in section~\ref{SS:min-max-rel-1}. Then, by theorem~\ref{T:main-1} we have a quasi-contractive $C_0$-semigroup $(e^{-t\hvmaxp})_{t\geq 0}$ in $L_{\mu}^p(\vbn)$, where $\mu$ is as in~(\ref{E:meas-mu}). Similarly to case of the covariant Schr\"odinger operator $\nabla^{\dagger}\nabla+V$ in $L^2(\vbn)$ with a potential $0\leq V\in\lloc^2(\End\vbn)$ studied in~\cite{Gu-10}, the next theorem gives a Feynman--Kac representation of the semigroup $e^{-t\hvmaxp}$ under the assumptions (F1)--(F3) and the stochastic completeness assumption described in (SC) below.

We now pause to describe the probabilistic ingredients of the statement. For more details on the stochastic analysis set up described below, see chapters 2 and 3 as well as appendix C of~\cite{Gu-disert-10}.

We start from a filtered probability space $(\Omega,\mathcal{F},\mathcal{F}_{*},\mathbb{P})$, with a right-continuous filtration $\mathcal{F}_{*}$ and the (measure theoretically) complete pair $(\mathbb{P},\mathcal{F}_{t})$ for all $t\geq 0$. We assume that this probability space carries a Brownian motion $W$ on $\mathbb{R}^{l}$, where $l\geq \dim M=n$ is large enough so that there is an isometric embedding $M\subset \RR^{l}$. As we want the second-order component of the generator for the diffusion described below to have a coefficient $1$, we ``speed up" $W$ so that the covariation $[\cdot,\cdot]$ satisfies $d[W_{t}^{j},W_{t}^{k}]=2\delta_{jk}\,dt$, with $\delta_{jk}$ being the Kronecker delta. Additionally, we  assume that $\mathcal{F}_{*}=\mathcal{F}_{*}(W)$.

Let $A\colon M\times \RR^{l}\to TM$ be a morphism of vector bundles defined by specifying $A(x)\colon \RR^{l}\to T_{x}M$, $x\in M$, as the orthogonal projection of $\RR^{l}$ onto  $T_{x}M$. Furthermore, denoting by $(e_{j})_{j=1}^{l}$ the standard basis of $\RR^l$, we define the (smooth) vector fields $A_j(\cdot):=A(\cdot)e_j$.

In this context, we can construct a diffusion $Y_{\bullet}(x)\colon [0,\zeta(x))\times \Omega\to M$, starting at $x\in M$ and with lifetime $\zeta(x)$,  as the maximally defined solution of the (stochastic differential) equation
\begin{equation}\label{E:diff-start-0}
dY_{t}(x)=\sum_{j=1}^{l}A_{j}(Y_{t}(x))\underbar{d}W_{t}^{j}+Z(Y_{t}(x))dt,\quad Y_{0}(x)=x,
\end{equation}
where $\underbar{d}$ denotes the Stratonovich differential, and
\begin{equation}\label{E:Z-X-PHI}
Z:=X-(d\phi)^{\sharp}
\end{equation}
with $(d\phi)^{\sharp}$ indicating the vector field corresponding to $d\phi$ via $g$.

For the rest of this section we make an additional assumption on the manifold, which we use in~(\ref{E:FKI-1-b}) below to implement a martingale-type argument from~\cite{Gu-10}:

\begin{enumerate}

\item [(SC)] We assume that the manifold $M$ is stochastically complete with respect to the diffusion $Y_t(x)$,
that is, we assume that for all $x\in M$, the lifetime of $Y_t(x)$ is $\zeta(x)=\infty$.

\end{enumerate}

\begin{remark} By the property (B.1) in~\cite{Gu-10} we have $\displaystyle\sum_{j=1}^{l}A^2_j=-\Delta$, where $\Delta=d^{\dagger}d$ (in particular, our $\Delta$ is a non-negative operator). Therefore, keeping in mind our ``speeding" convention for $W_t$, the diffusion $Y_t$ is sometimes described in the literature as \emph{the diffusion to} $-\Delta+Z$, with $Z$ as in~(\ref{E:Z-X-PHI}).

In the remainder of this remark, we recall some known conditions under which the property (SC) is satisfied in the case $X\equiv 0$ (that is, the case of the diffusion $Y_t$ to the Witten Laplacian $-\Delta-(d\phi)^{\sharp}$, where $\Delta$ is the non-negative Laplacian on $M$). Denoting by $\mu(B_{\rho}(x))$ the volume of the ball $B_{\rho}(x)$ in~(\ref{E:ball-r}), with $\mu$ as in~(\ref{E:meas-mu}), we first recall the ``volume condition" for weighted manifolds from theorem 11.8 of the book~\cite{grigoryan}:  if $(M,g)$ is geodesically complete and if
\begin{equation}\label{E:vol-grig}
\int_{\rho_0}^{\infty}\frac{\rho\,d\rho}{\log(\mu(B_{\rho}(x)))}=\infty,
\end{equation}
for some $x\in M$ and some $\rho_0>0$, then the property (SC) is satisfied (where $Z$ in~(\ref{E:diff-start-0}) has the form $Z=-(d\phi)^{\sharp}$).

It turns out (see theorem 1.1 (and its remark) in~\cite{FY-Wang}) that the geodesic completeness of $(M,g)$ together with the hypothesis
\begin{equation}\label{E:alpha-cond-sc}
\ricc_{M}(W,W)+\hess\phi(W,W)-\alpha^{-1}\left[g(W,(d\phi)^{\sharp})\right]^2\geq C,
\end{equation}
for some $\alpha\geq 1$, for some $C\in\RR$, and for all $W\in TM$, guarantee the fulfillment of certain estimates of the diffusion kernel (with respect to $\mu$ in~(\ref{E:meas-mu})) corresponding to $-\Delta-(d\phi)^{\sharp}$, which in turn ensures that~(\ref{E:vol-grig}) is satisfied, thus granting the property (SC) with $Z=-(d\phi)^{\sharp}$.

Using an approach based on ``differentiating" the semigroup corresponding to $-\Delta-(d\phi)^{\sharp}$ and assuming the geodesic completeness of $(M,g)$ together with the (weaker than~(\ref{E:alpha-cond-sc})) condition
\begin{equation}\label{E:cond-sc-bakry}
\ricc_{M}(W,W)+\hess\phi(W,W)\geq C,
\end{equation}
for some $C\in\RR$ and for all $W\in TM$, the author of~\cite{Bakry-86} showed that $Y_t(x)$ in~(\ref{E:diff-start-0}) with $Z=-(d\phi)^{\sharp}$ has an infinite lifetime for all $x\in M$, that is, the property (SC) with $Z=-(d\phi)^{\sharp}$ is satisfied. Lastly, we remark that, to our knowledge, the fulfillment of~(\ref{E:cond-sc-bakry}) does not guarantee the property~(\ref{E:vol-grig}).
\end{remark}

Let $m:=\rank(\vbn)$, let $\mathscr{U}(m)$ be the space of unitary $m\times m$ (complex) matrices, and let $\pi\colon \mathscr{P}(\vbn)\to M$ the principal $\mathscr{U}(m)$-bundle of (unitary) frames in $\vbn$. Let $u\colon \Omega\to \mathscr{P}(\vbn)$ be a $\mathscr{F}_{0}$-measurable random variable such that $\pi(u)=x$. Keeping in mind our covariant derivative $\nabla$ on $\vbn$, we can describe the \emph{stochastic $\nabla$-horizontal lift} $U_t(u)\colon [0,\infty)\times \Omega\to \mathscr{P}(\vbn)$ \emph{of} $Y_t(x)$ \emph{from }$\pi(u)=x$, $\mathbb{P}$~a.s.~, as the maximally defined solution of (the Stratonovich) equation
\begin{equation}\label{E:stoch-diff-1}
dU_{t}(u)=\sum_{j=1}^{l}A^{*}_{j}(U_{t}(u))\underbar{d}W_{t}^{j}+Z^*(U_{t}(u))dt,\quad U_{0}(u)=u,
\end{equation}
where $A^{*}_{j}\in C^{\infty}(\mathscr{P}(\vbn), T\mathscr{P}(\vbn))$ and $Z^{*}\in C^{\infty}(\mathscr{P}(\vbn), T\mathscr{P}(\vbn))$ are $\nabla$-lifts of $A_j$ and $Z$ respectively. (We remark that it is known that the lifetime of $U_{t}(u)$ is the same as that of $Y_t(x)$, which equals $\infty$ according to our assumption (SC).)

We can now describe \emph{the stochastic parallel transport} $\slash\slash^{t}_{x}\colon \vbn_{x}\to \vbn_{Y_{t}(x)}$ \emph{corresponding to} $\nabla$ \emph{and} $Y_t(x)$ as
\begin{equation}\label{E:spt-1}
\slash\slash^{t}_{x}:=U_t\,u^{-1},\qquad \mathbb{P}\textrm{~a.s},\quad t\geq 0,
\end{equation}
where $u=U_0(x)$.

It is known that this definition does not depend on the choice of $u$ with $\pi(u)=x$. Furthermore, the map $\slash\slash^{t}_{x}$ is an isometry.

This brings us to the last probabilistic ingredient, the $\End \vbn_{x}$-valued process $\mathscr{V}^{x}_{t}$, defined as the unique pathwise solution to
\begin{equation}\label{E:def-script-v}
d\mathscr{V}^{x}_{t}=-\mathscr{V}^{x}_{t}(\slash\slash^{t,-1}_{x}V(Y_{t}(x)))\slash\slash_{t}^{x})\,dt, \qquad \mathscr{V}^{x}_{0}=I,
\end{equation}
where $\slash\slash^{t,-1}_{x}$ is the inverse of $\slash\slash_{t}^{x}$ and $I$ is the identity endomorphism.

We are ready to state the mentioned Feynman--Kac formula.

\begin{theorem}\label{T:main-3} Let $1<p<\infty$ and let $M$, $\vbn$, and $\nabla$ be as in theorem~\ref{T:main-1}. Assume that $M$ is geodesically complete. Additionally, assume that $V$, $X$, and $\phi$ satisfy the hypotheses (F1)--(F3). Moreover, assume that $M$ satisfies the stochastic completeness assumption (SC). Let $\hvmaxp$ be as in section~\ref{SS:min-max-rel-1} with $\mu$ as in~(\ref{E:meas-mu}), and let $Y_t(x)$ and $\mathscr{V}^{x}_{t}$ be as in~(\ref{E:diff-start-0}) and~(\ref{E:def-script-v}) respectively. Then,
\begin{equation}\label{E:F-K}
(e^{-t\hvmaxp}f)(x)=\mathbb{E}\left[\mathscr{V}^{x}_{t}\slash\slash^{t,-1}_{x}f(Y_{t}(x)) \right],
\end{equation}
for all $f\in L^p_{\mu}(\vbn)$, for all $t\geq 0$, a.e.~$x\in M$.
\end{theorem}
\begin{remark} Recently, in the context of a general manifold (not necessarily geodesically or stochastically complete), the authors of~\cite{Thalmaier-17} established, among other things, a representation formula for the solutions $f_t(\cdot)=f(t,\cdot)$ of the heat equation
\begin{equation*}
\pa_{t}f_t=(-\nabla^{\dagger}\nabla+\nabla_{X}-V)f_t, \qquad f_0=f, \quad (t,x)\in[0,T]\times M,
\end{equation*}
where $X$ is a smooth real vector field, $V\in C^{\infty}(\End\vbn)$, $f\in C^{\infty}(\vbn)$, and the operators on the right hand side are applied with respect to the spatial variable. In corollary 1.7 of~\cite{Thalmaier-17} it was shown that the process
\begin{equation*}
N_t:=\mathscr{V}^{x}_{t}\slash\slash^{t,-1}_{x}f_{T-t}(Y_{t}(x))1_{\{t<\zeta(x)\}}, \quad t\in[0,T],
\end{equation*}
where $\zeta(x)$ is the lifetime of $Y_t(x)$, is a local martingale. Furthermore, under an additional hypothesis that $N_t$ is a martingale, the authors of~\cite{Thalmaier-17} obtained (in the case $\phi\equiv 0$) the formula like~(\ref{E:F-K}), with the mentioned solution $f_t$ replacing our semigroup on the left hand side.

The representation~(\ref{E:F-K}) provides a link between the path integral on the right hand side and the maximal realization $\hvmaxp$ of $P^{\nabla}$ in $L^p_{\mu}(\vbn)$, $1<p<\infty$. In this regard, the formula~(\ref{E:F-K}) is modeled after the result of~\cite{Gu-10} in the context of the operator $\nabla^{\dagger}\nabla+V$ in $L^2(\vbn)$ with $0\leq V\in L_{\loc}^2(\End\vbn)$.  We should point out that for the operator $\nabla^{\dagger}\nabla+V$ in $L^2(\vbn)$ whose potential $V\in L_{\loc}^1(\End\vbn)$ is not necessarily bounded from below (but satisfying certain Kato-class related conditions guaranteeing the lower semi-boundedness of the operator $\nabla^{\dagger}\nabla+V$ in $L^2(\vbn)$), the corresponding Feynman--Kac formula was established in~\cite{Gu-12}.

In the very recent paper~\cite{Boldt-Guneysu-20}, in the context of (not necessarily geodesically complete or stochastically complete) Riemannian manifolds, the authors proved a general Feynman--Kac formula for the semigroup corresponding to the appropriate sectorial realization of the expression $\nabla^{\dagger}\nabla +F$ in $L^2(\vbn)$, where $F\colon C^{\infty}(\vbn)\to C^{\infty}(\vbn)$ is an arbitrary differential operator of order $\leq 1$ with not necessarily ``real coefficients." According to proposition 2.6 in~\cite{Boldt-Guneysu-20}, one set of hypotheses on $F=\sigma_1(F)\nabla+F_{\nabla}$, where $\sigma_1(F)$ is the principal symbol of $F$ and $F_{\nabla}$ is a zero-order operator, ensuring that the form $s(u):=(\nabla^{\dagger}\nabla u +Fu,u)$ with domain $\dom(s)=\vcomp$ is sectorial in $L^2(\vbn)$ can be described as follows:
\begin{equation}\label{E:sectorial-BG}
|\RE (\sigma_1(F))|\in L^{\infty}(M), \quad \RE(F_{\nabla})\geq c,\quad |\IM(\sigma_1(F))|\in \mathcal{K}(M),
\end{equation}
where $c\in\RR$, the notation $\mathcal{K}(M)$ refers to the Kato class of functions on $M$, and $\RE A$, $\IM A$ stand for the real and imaginary parts of the endomorphism $A$ of $\vbn$. It turns out that, under these conditions, the form $s$ is closable, the closure $\overline{s}$ of $s$ is a closed sectorial form, and the (closed sectorial) operator $S$ associated (via an abstract fact) to $\overline{s}$  generates an analytic semigroup $e^{-zS}$ in $L^2(\vbn)$, where $z$ belongs to some sector of the complex plane (see chapter VI in~\cite{Kato80} for the theory of sectorial forms and associated $m$-sectorial operators in a Hilbert space).

The Feynman--Kac formula from~\cite{Boldt-Guneysu-20} reads:
\begin{equation}\label{E:FKI-BG}
(e^{-tS}f)(x)=\mathbb{E}[\mathscr{F}^{x}_{t}\slash\slash^{t,-1}_{x}f(B_{t}(x))1_{\{t<\zeta(x)\}}], \qquad x\in M,\, t>0,\, f\in L^2(\vbn),
\end{equation}
where $B_t(x)$, the usual Brownian motion on $M$ with lifetime $\zeta(x)$, is the unique solution to~(\ref{E:diff-start-0}) with $Z=0$, and $\mathscr{F}^{x}_{t}$ is the unique solution to the It\^o equation
\begin{equation}\nonumber
d\mathscr{F}^{x}_{t}=-\mathscr{F}^{x}_{t}\slash\slash^{t,-1}_{x}\left((\sigma_1(F))^{\flat}(dB_{t}(x))
+F_{\nabla}(B_{t}(x))\,dt\right)\slash\slash_{t}^{x}, \qquad \mathscr{F}^{x}_{0}=I,
\end{equation}
where $\slash\slash^{t,-1}_{x}$ is the inverse of the stochastic parallel transport $\slash\slash_{t}^{x}$ corresponding to $\nabla$ and $B_t(x)$, and $I$ is the identity endomorphism.

It turns out that the conditions~(\ref{E:sectorial-BG}) ensure that $\mathscr{F}^{x}_{t}$ is well-behaved in the sense that
\begin{equation*}
\displaystyle\sup_{x\in M}\mathbb{E}[1_{\{t<\zeta(x)\}}|\mathscr{F}^{x}_{t}|^2]<\infty,
\end{equation*}
for all $t>0$.

Going back to the expression~(\ref{E:def-H}) and assuming $p=2$, $\phi\equiv 0$, $|X|\in L^{\infty}(M)$, $V\in C^{\infty}(\End\vbn)$ and $V\geq c$, where $c\in\RR$, one can use the formula~(\ref{E:FKI-BG}) from~\cite{Boldt-Guneysu-20}, where $S$ is the above described sectorial realization of $\nabla^{\dagger}\nabla-\nabla_{X}+V$ in $L^2(\vbn)$. Note that in the conditions~(\ref{E:sectorial-BG}) from~\cite{Boldt-Guneysu-20}, no assumptions are made on $\div X$. Furthermore, $(M,g)$ is not assumed to be geodesically complete or stochastically complete.
\end{remark}

\section{Preliminaries}\label{S:prelim-1}
In this section we collect various product/chain rules, state the $L^p$-Calder\'on--Zygmund inequality, and record a few basic inequalities used in subsequent parts of the paper. We also state and prove a version of the integration by parts formula tailored to the expression $P^{\nabla}$ in~(\ref{E:def-H}).

In all relevant formulas below, $(M, g)$ is an $n$-dimensional Riemannian manifold with Riemannian metric $g$, and the operators $\nabla$,  $\nabla^{\dagger}$, $d$, $d^{\dagger}$, $\Delta$, $\nabla^{lc}$, and $\hess$ are as in section~\ref{S:res}.
\subsection{Product Rules}\label{SS:pr}
Let $f \in C^{\infty}(M)$, let $\omega$ be a one-form of class $W^{1,1}_{\loc}(\Lambda^1 T^*M)$, let $w \in W^{2,1}_{\loc}(M)$, let $u\in W^{2,1}_{\loc}(\vbn)$, and let $Z$ be a smooth vector field on $M$. Then
\begin{enumerate}
\item [(p1)] $d(fw)=(df)w+fdw$,
\item [(p2)] $d^{\dagger}(f \omega) = fd^{\dagger}\omega- \langle df,\omega\rangle_{1}$,
\item [(p3)] $d^{\dagger}(f dw) = f\Delta w - \langle df, dw\rangle_{1}$,
\item [(p4)] $\Delta (fw) = f\Delta w - 2\langle df, dw\rangle_{1} + w\Delta f$,
\item [(p5)] $\hess (fw) = f\hess w +2 df \otimes dw + w\hess f$,
\item [(p6)] $\nabla (fu)=f\nabla u+df\otimes u$,
\item [(p7)] $\nabla_{Z}(fu)=f \nabla_{Z}u+(Zf)u$,
\item [(p8)] $\nabla^{\dagger}(f\nabla u) = f \nabla^{\dagger}\nabla u - \nabla_{(df)^\#}u$,
\item [(p9)] $\nabla^{\dagger}\nabla(fu) = f\nabla^{\dagger}\nabla u - 2\nabla_{(df)^\#}u + u\Delta f$,
\end{enumerate}
where  $\langle \cdot,\cdot \rangle_{1}$ is as in~(\ref{E:form-inner-x}), $Zf$ indicates the action of $Z$ on $f$, and $(df)^{\#}$ stands for the vector field corresponding to $df$ via the metric $g$.

The formulas (p1), (p6) and (p7) are basic product rules for the indicated operators.
The rule (p2) can be obtained from the definition of the operator $d^{\dagger}$, while the rules (p3)--(p5) follow from (p1)--(p2) and the definitions $\Delta:=d^{\dagger}d$ and $\hess:=\nabla^{lc}d$. The rule (p8) follows by integration by parts and the formula
\begin{equation*}
(\nabla u, df\otimes v)=(\nabla_{(df)^\#} u,v), \qquad v\in\vcomp,
\end{equation*}
where $(\cdot,\cdot)$ is as in~(\ref{E:inner-mu}) with the usual volume measure $d\mu=d\nu_{g}$. The rule (p9) is a consequence of (p6) and (p8) and the following fact (for which we refer to the formula~(III.7) of~\cite{Gue-book}):
\begin{equation*}
\nabla^{\dagger}(\sigma\otimes z)= (d^{\dagger}\sigma)z-\nabla_{\sigma^\#}z,
\end{equation*}
where $z\in W^{1,2}_{\loc}(E)$ and $\sigma\in\Omega^1(M)$ is a one-form.

\subsection{Chain Rules and Laplacian--Hessian Inequality}\label{SS:cr}
Let $w \colon M \rightarrow \RR$ and $f \colon U \rightarrow \RR$ be smooth functions, where $U\subseteq\RR$ is an open set containing the range of $w$. Then,
\begin{enumerate}
\item [(c1)] $d(f\circ w)=f'(w)dw$,
\item [(c2)] $\Delta(f \circ w) = -f''(w)|dw|^2 +f'(w)\Delta w$,
\item [(c3)] $\hess(f \circ w) = f''(w)dw\otimes dw + f'(w)(\hess w)$,
\end{enumerate}
where $|\cdot|$ in the above formulas is understood, depending on the context, as the absolute value of a number or the fiberwise norm in $T^*M$ and $T^{0,2}(M)$, with the Euclidean structures induced by $g$.
For the formula (c1) see exercise 3.4 of \cite{grigoryan}.  For the formula (c2), see exercises 3.4 and 3.9 in \cite{grigoryan}, keeping in mind that $\Delta$ in our article is set up as a non-negative operator (in contrast to~\cite{grigoryan}, where $\Delta$ is non-positive). Formula (c3) follows from formula (c1) and the definition of the Hessian.

We will also use the inequality (see~(III.24) in~\cite{Gue-book})
\begin{equation}\label{E:lap-hess}
|\Delta w(x)|\leq \sqrt{n}|(\hess w)(x)|,
\end{equation}
for all $x\in M$, where $n=\dim M$.

\subsection{$L^p$-Calder\'on--Zygmund Inequality}\label{SS:c-z}
We say that $(M,g)$ satisfies the $L^p$-Calder\'on--Zygmund Inequality for some $1<p<\infty$, if there exists a constant $C>0$, such that
\begin{equation}\tag{CZ(p)}
\|\hess u\|_{p}\leq C(\|\Delta u\|_{p}+\|u\|_{p}),
\end{equation}
for all $u\in\mcomp$, where $\|\cdot\|_{p}$ is as in~(\ref{E:lp-norm}) with $\mu=\nu_{g}$, the volume measure.

The authors of~\cite{GP-2015} showed (see theorem 4.15 there) that $(M,g)$ satisfies (CZ(p)) for all $1<p<\infty$ if the Ricci curvature $\ricc_{M}$ is bounded and if the injectivity radius $r_{\textrm{inj}}(M)$ is positive. (The constant $C$ in (CZ(p)) depends on $n=\dim M$, $p$, $\|\ricc_{M}\|_{\infty}$, and $r_{\textrm{inj}}(M)$.) Another sufficient condition for (CZ(p)), as specified in theorem 5.18 in~\cite{Pigola-20}, requires the geodesic completeness of $M$, the boundedness of $\textrm{Sect}_{M}$, and the fulfillment of $p\in[2,\infty)\cap(n/2, \infty)$, where $n=\dim M$.  (Here, $\textrm{Sect}_{M}$ stands for the sectional curvature of $M$.)

It turns out in order for $(M,g)$ to satisfy the property (CZ(p)), some assumptions on the geometry of $M$ are needed, as shown in the papers~\cite{GP-2015, Li-20, V-20} where the authors constructed various examples of manifolds (with $\ricc_{M}$ unbounded from below) for which (CZ(p)) is  not satisfied. We should add that for a few years the researchers wondered (going as far back as remark 2.7 in~\cite{Gu-seq-14}) whether (CZ(p)) holds for all $1<p<\infty$ if one assumes just the geodesic completeness of $M$ and the condition $\ricc_{M}\geq -K$, where $K>0$ is some constant. Remarkably, the authors of~\cite{M-V-21} showed very recently (see Theorem B there) that for every $n\in\{2,3,\dots\}$ and $p>n$, there exists a complete, non-compact Riemannian manifold $(M, g)$ with $\dim M=n$ and $\textrm{Sect}_{M}>0$, such that (CZ(p)) fails.

\subsection{Basic Inequalities with Numbers}\label{SS:bin}
Let $a,b,c,t\in\RR$ and let $p>1$.  Then,
\begin{enumerate}
\item [(n1)] for all $\kappa_1>1$ and $\kappa_2>1$ such that $1/\kappa_1+1/\kappa_2=1$,
we have
\begin{equation}\nonumber
|ab|\leq \frac{|a|^{\kappa_1}}{\kappa_1}+\frac{|b|^{\kappa_2}}{\kappa_2};
\end{equation}
\item [(n2)] $\displaystyle |a+b+c|^{p}\leq 3^{p-1}(|a|^p+|b|^p+|c|^p)$;
\item [(n3)] $\displaystyle |a+b+c|^{1/p}\leq (|a|^{1/p}+|b|^{1/p}+|c|^{1/p})$;
\item [(n4)] $|ab+ct|\leq (a^2+c^2)^{1/2}(b^2+t^2)^{1/2}$.
\end{enumerate}

\subsection{Integration by Parts Formula for $P^{\nabla}$}\label{SS:ibp}
Let $(\cdot,\cdot)$ be as in~(\ref{E:inner-mu}) with the usual volume measure $d\mu=d\nu_{g}$.
\begin{lemma}\label{L:lemma-1} Let $1<p<\infty$, let $\phi\in C^{\infty}(M)$ be real-valued, let $X$ be a smooth (real) vector field, and let $V\in\lloc^{\infty}(\End \vbn)$ be a self-adjoint section. Assume that $w\in W^{1,1}_{\loc}(\vbn)$ and $u\in C^{\infty}(\vbn)$, where $w$ or $u$ has compact support. Then
\begin{equation}\label{E:ibp-1}
(\nabla^{\dagger}\nabla u+\nabla_{(d\phi)^{\sharp}}u,w)=(\nabla u, \nabla w)+(u, (\Delta\phi)w)-(u,\nabla_{(d\phi)^{\sharp}}w).
\end{equation}
Furthermore,
\begin{equation}\label{E:ibp-2}
(-\nabla_{X}u+Vu,w)=-q^{-1}(\nabla_{X}u,w)+p^{-1}(u, \nabla_{X}w)+((V+p^{-1}\div X)u,w),
\end{equation}
where $q$ is the H\"older conjugate of $p$, that is, $q^{-1}+p^{-1}=1$.
\end{lemma}
\begin{proof} We recall (see proposition 1.4 in appendix C of~\cite{Taylor}) that the formal adjoint $(\nabla_{X})^{\dagger}$ with respect to $(\cdot,\cdot)$ of the operator $\nabla_{X}$ is given by
\begin{equation}\label{E:adj-X}
(\nabla_{X})^{\dagger}w=-\nabla_{X}w-(\div X)w.
\end{equation}
The formula~(\ref{E:ibp-1}) follows right away if we use~(\ref{E:adj-X}) with $X=(d\phi)^{\sharp}$ and if we notice that
\begin{equation}\label{E:calc-diverg}
\div ((d\phi)^{\sharp})=-d^{\dagger}(((d\phi)^{\sharp})^{\flat})=-d^{\dagger}d\phi=-\Delta\phi.
\end{equation}

To prove the formula~(\ref{E:ibp-2}), we write
\begin{align}
&(-\nabla_{X}u+Vu,w)=-(q^{-1}+p^{-1})(\nabla_{X}u,w)+(Vu, w)\nonumber\\
&=-q^{-1}(\nabla_{X}u,w)-p^{-1}(u, (\nabla_{X})^{\dagger}w)+(Vu, w)\nonumber
\end{align}
and combine this with~(\ref{E:adj-X}).
\end{proof}

\section{Proof of Part (i) of Theorem~\ref{T:main-1}}\label{S:hmax-1}
We take moment to recall some abstract terminology.
\subsection{Accretive, $m$-Accretive, and $m$-Sectorial Operators}\label{SS:accretive}
We say that a linear operator $T$ on a Banach space $\mathscr{B}$ is \emph{accretive}
if
\[
\|(\xi+T)u\|_{\mathscr{B}}\geq \xi\|u\|_{\mathscr{B}},
\]
for all $\xi>0$ and all $u\in\dom(T)$.
We say that a (densely defined) operator $T$ on $\mathscr{B}$ is \emph{$m$-accretive} if it is accretive and $\xi+T$ is surjective for all $\xi>0$. It is known (see theorem II.3.15 in~\cite{engel-nagel}) that if $T$ is an $m$-accretive operator on $\mathscr{B}$, then its negative $-T$ generates a contractive strongly continuous (often labeled as $C_0$) semigroup on $\mathscr{B}$. Similarly, if $\lambda>0$ and $T+\lambda$ is $m$-accretive, then $-T$ generates a quasi-contractive $C_0$-semigroup on $\mathscr{B}$. Specializing to $\mathscr{B}=L^p_{\mu}(\vbn)$, $1<p<\infty$, with $\vbn$ as in section~\ref{SS:s-2-1} and $\mu$ as in~(\ref{E:meas-mu}), we say that $T$ is an \emph{$m$-sectorial operator of type $\mathscr{S}(c)$}, $c>0$, on $L^p_{\mu}(\vbn)$  if $T$ is $m$-accretive on $L^p_{\mu}(\vbn)$ and
\begin{equation}\nonumber
\{(T u,|u|^{p-2}u)_{\mu}\colon u\in\dom(T)\}\subset \mathscr{S}(c),
\end{equation}
where
\begin{equation}\label{E:sector-analyticity}
\mathscr{S}(c):=\{z\in \mathbb{C}\colon |\IM z|\leq c \RE z\}.
\end{equation}
By an abstract fact (see corollary 2.27 in~\cite{Sch-evol-20}), if $T$ is an $m$-sectorial operator on $L^p_{\mu}(\vbn)$, $1<p<\infty$, then $-T$ generates a contractive analytic $C_0$-semigroup on $L^p_{\mu}(\vbn)$. In relation to $c$ in~(\ref{E:sector-analyticity}), a more precise description of the ``angles" for the sectors of analyticity (respectively, contractiveness) of this semigroup is given in the referenced corollary. Similarly, if $\lambda>0$ and $T+\lambda$ is $m$-sectorial, then $-T$ generates a quasi-contractive analytic $C_0$-semigroup on $L^p_{\mu}(\vbn)$, $1<p<\infty$.

\subsection{Reduction to a Covariant Schr\"odinger Operator With a Drift}
Following the procedure from~\cite{So-Yo-13} we will conveniently transform the expression $P^{\nabla}$ in~(\ref{E:def-H}). Clearly, the statements in part (i) of theorem~\ref{T:main-1} hold in $L^p_{\mu}(\vbn)$ if and only if the same statements hold in $L^p(\vbn)$ (the usual $L^p$-space with the volume measure $\nu_{g}$) for the operator $S_{p,\max}$ defined as in section~\ref{SS:min-max-rel-1}, with $P^{\nabla}$ replaced by $S^{\nabla}:=e^{-\phi/p}P^{\nabla}e^{\phi/p}$ and with $L^p_{\mu}(\vbn)$ replaced by $L^p(\vbn)$.  As we will see in the sequel, the differential expression corresponding to $S_{p,\max}$ has the form
\begin{equation}\nonumber\\
S^{\nabla}v=\nabla^{\dagger}\nabla v+\nabla_{Y}v+Gv,
\end{equation}
for some smooth vector field $Y$ and some self-adjoint section $G\in L^{\infty}_{\loc}(\End\vbn)$. We will also see that for a sufficiently large number $\xi_0>0$ the operator $S_{p,\max}+\xi_0$ is covered by the following proposition:
\begin{prop} \label{P:prop-cov-1}
Let $M$ is a geodesically complete Riemannian manifold and let $\vbn$ be a Hermitian vector bundle over $M$ with a metric covariant derivative $\nabla$. Let $Z$ be a smooth vector field on $M$, let $Q\in \lloc^{\infty}(\End \vbn)$ be a self-adjoint section, and let $1<p<\infty$.
Let $T_{p,\max}$ and  $T_{p,\min}$ be as in section~\ref{SS:min-max-rel-1}, with $P^{\nabla}$ replaced by $\nabla^{\dagger}\nabla +\nabla_{Z}+Q$ and with $L^p_{\mu}(\vbn)$ replaced by $L^p(\vbn)$, that is, the $L^p$-space with the volume measure $\nu_{g}$. Assume that there exists a constant $\gamma_1>0$ and a function $0\leq f\in L^{\infty}_{\loc}(M)$ such that
\begin{itemize}
  \item [(i)] $|Z(x)|\leq \gamma_1(f(x))^{1/2}$, for all $x\in M$;
  \item [(ii)] $Q(x)-p^{-1}(\div Z)(x)\geq f(x)$, for all $x\in M$, where the inequality is interpreted in the sense of quadratic forms on ${\vbn}_{x}$.
\end{itemize}
Then, the following properties hold:
\begin{itemize}
  \item [(i)] $T_{p,\max}=\overline{T_{p,\min}}$;
  \item [(ii)] $T_{p,\max}$, as an operator in the space $L^p(\vbn)$, is $m$-sectorial of type $\mathscr{S}({c_{p,\gamma_1}})$,
where
\begin{equation}
\nonumber c_{p,\gamma_1}:=\left[2^{-1}(p-1)^{-1}(p-2)^2+2^{-1}\gamma_1^2\right]^{1/2};
\end{equation}
\item [(iii)] $-T_{p,\max}$ generates a contractive analytic $C_0$-semigroup on $L^p(\vbn)$.
\end{itemize}
\end{prop}
\begin{proof}
For the proofs of parts (i) and (ii), we refer the reader to theorems 1 and 2 in~\cite{Mil-19}, respectively. Part (iii) is a consequence of part (ii) and the discussion in section~\ref{SS:accretive} above.
\end{proof}

\subsection{Proof of Part (i) of Theorem~\ref{T:main-1}}\label{SS:thm-1-2-resume}
We start by writing down the expression $e^{-\phi/p}P^{\nabla}e^{\phi/p}$. To keep track of the calculations more easily, we refer to the chain rules (c1) and (c2) to record
\begin{equation}\label{E:l-u-2}
d(e^{\phi/p})=e^{\phi/p}p^{-1}d\phi,\quad \Delta(e^{\phi/p})=-e^{\phi/p}p^{-2}|d\phi|^2+e^{\phi/p}p^{-1}\Delta\phi.
\end{equation}

Corresponding to each term of $P^{\nabla}$ in~(\ref{E:def-H}), we have
\begin{align}\label{E:l-u-2-a}
&\nabla^{\dagger}\nabla (e^{\phi/p}v)=e^{\phi/p}\nabla^{\dagger}\nabla v-2\nabla_{(d(e^{\phi/p}))^{\sharp}}v+(\Delta(e^{\phi/p}))v\nonumber\\
&=e^{\phi/p}\nabla^{\dagger}\nabla v-2p^{-1}e^{\phi/p}\nabla_{(d\phi)^{\sharp}}v+(-e^{\phi/p}p^{-2}|d\phi|^2+e^{\phi/p}p^{-1}\Delta\phi)v,
\end{align}
where we used (p9) and~(\ref{E:l-u-2});
\begin{align}\label{E:l-u-2-b}
&\nabla_{(d\phi)^{\sharp}}(e^{\phi/p}v)=e^{\phi/p}\nabla_{(d\phi)^{\sharp}}v+[(d\phi)^{\sharp}(e^{\phi/p})]v\nonumber\\
&=e^{\phi/p}\nabla_{(d\phi)^{\sharp}}v+p^{-1}e^{\phi/p}[(d\phi)^{\sharp}\phi]u=e^{\phi/p}[\nabla_{(d\phi)^{\sharp}}v+p^{-1}|d\phi|^2v]
\end{align}
where we used (p7) and~(\ref{E:l-u-2});
\begin{align}\label{E:l-u-2-c}
&\nabla_{X}(e^{\phi/p}v)=e^{\phi/p}\nabla_{X}v+[X(e^{\phi/p})]v\nonumber\\
&=e^{\phi/p}\nabla_{X}u+p^{-1}e^{\phi/p}(X\phi)v,
\end{align}
where we used (p7) and~(\ref{E:l-u-2}).

Combining~(\ref{E:l-u-2-a}),~(\ref{E:l-u-2-b}), and~(\ref{E:l-u-2-c}), we can write the action of $S^{\nabla}:=e^{-\phi/p}P^{\nabla}e^{\phi/p}$ on $v$ as
\begin{equation}\label{E:s-nabla}
S^{\nabla}v=\nabla^{\dagger}\nabla v+\nabla_{Y}v+Gv,
\end{equation}
where
\begin{equation}\nonumber
Y:=-X+(1-2p^{-1})(d\phi)^{\sharp},
\end{equation}
\begin{equation}\nonumber
G:=V+(p^{-2}(p-1)|d\phi|^2+p^{-1}\Delta\phi-p^{-1}X\phi)I,
\end{equation}
where $I_{x}\colon \vbn_{x}\to \vbn_{x}$ is the identity endomorphism.

It remains to see that for a sufficiently large number $\xi_0>0$ the expression $S^{\nabla}+\xi_0$ satisfies the hypotheses of proposition~\ref{P:prop-cov-1} with the following function playing the role of $f$:
\begin{equation}\label{E:def-funct-f}
f(x):=(1-p^{-1}\theta)h(x)+2^{-1}p^{-2}(p-1)|d\phi(x)|^2+1.
\end{equation}
First, we note that $f\in L^{\infty}_{\loc}(M)$ and, in view of our hypothesis $\theta<p$, we have $f\geq 0$.
Next, we use (A4) to get
\begin{align}\label{E:p-accr-1}
&|Y|\leq |X|+|1-2p^{-1}||d\phi|\leq \kappa(|d\phi(x)|^2+h(x)+\beta_2)^{1/2}+p^{-1}|p-2||d\phi|\nonumber\\
&\leq (\kappa^2+p^{-2}(p-2)^2)^{1/2}(2|d\phi|^2+h+\beta_2)^{1/2}\leq (\kappa^2+p^{-2}(p-2)^2)^{1/2}(\delta f)^{1/2},\nonumber
\end{align}
where in the third estimate we used the inequality (n4), and in the fourth estimate we used~(\ref{E:def-funct-f}) together with
\begin{equation*}
\delta:=\max\{p(p-\theta)^{-1}, 4p^2(p-1)^{-1}, \beta_2\}.
\end{equation*}
Thus, $Y$ satisfies the hypothesis (i) of proposition~\ref{P:prop-cov-1}.

By the definitions of $G$ and $Y$ and the formula~(\ref{E:calc-diverg}) we have
\begin{align}\label{E:p-accr-2}
&G-p^{-1}\div Y=V+p^{-1}(\div X-X\phi)+p^{-2}(p-1)|d\phi(x)|^2\nonumber\\
&+2p^{-2}(p-1)\Delta\phi.
\end{align}

Using the condition $V\geq h$ and (A3) we get
\begin{align}\label{E:p-accr-3}
&V+p^{-1}(\div X-X\phi)\geq h+p^{-1}(\div X-X\phi)\nonumber\\
&\geq -p^{-1}\beta_1+(1-p^{-1}\theta)h.
\end{align}

Furthermore,
\begin{align}\label{E:p-accr-4}
&p^{-2}(p-1)\left(|d\phi|^2+2\Delta\phi\right)\geq p^{-2}(p-1)\left(|d\phi|^2-2\sqrt{n}|\hess\phi|\right)\nonumber\\
&\geq p^{-2}(p-1)\left[|d\phi|^2-2\sqrt{n}(\vep|d\phi|^2+C_{\vep})\right],
\end{align}
where in the first inequality we used~(\ref{E:lap-hess}) and in the second inequality we used~(A2).

Combining~(\ref{E:p-accr-2}),~(\ref{E:p-accr-3}), and~(\ref{E:p-accr-4}) we obtain
\begin{align}\label{E:p-accr-5}
&G-p^{-1}\div Y\geq (1-p^{-1}\theta)h+p^{-2}(p-1)(1-2\vep\sqrt{n})|d\phi|^2\nonumber\\
&-p^{-1}\beta_1-2p^{-2}(p-1)C_{\vep}\sqrt{n}.
\end{align}

Looking at~(\ref{E:p-accr-5}), we see that imposing the condition  $\vep\leq 1/(4\sqrt{n})$ ensures that the coefficient of the term containing $|d\phi|^2$ is greater than or equal to that of the corresponding term in~(\ref{E:def-funct-f}). Therefore, choosing $0<\vep_1\leq 1/(4\sqrt{n})$ and putting
\begin{equation*}
\xi_0:=p^{-1}\beta_1+2p^{-2}(p-1)C_{\vep_1}\sqrt{n}+1,
\end{equation*}
we see that
\begin{equation*}
\xi_0+G-p^{-1}\div Y\geq f,
\end{equation*}
where $f$ is as in~(\ref{E:def-funct-f}).

Therefore, the condition (ii) of proposition~\ref{P:prop-cov-1} is satisfied with $Q:=\xi_0+G$.
Thus, the properties (i) and (ii) described in the conclusion of proposition~\ref{P:prop-cov-1} hold for the operator $S_{p,\max}+\xi_0$, where $S_{p,\max}$ is the maximal realization of~(\ref{E:s-nabla}) in $L^p(\vbn)$. Thus, $S_{p,\max}=\overline{S_{p,\min}}$, and, furthermore, $-S_{p,\max}$ generates a quasi-contractive analytic $C_0$-semigroup in $L^p(\vbn)$. As mentioned above, this allows us to infer that the same properties hold for $\hvmaxp$ in the space $L^p_{\mu}(\vbn)$. This concludes the proof of part (i) of theorem~\ref{T:main-1}. $\hfill\square$

\section{Proof of part (ii) of Theorem~\ref{T:main-1}}\label{S:coercive-1}

In this section we  prove the coercive estimate in part (ii) of Theorem~\ref{T:main-1}.
We begin with a few preliminary lemmas.

\subsection{Preliminary Lemmas} \label{SS:coercive-1-prelim}
For the remainder of this section, $(M,g)$ is an $n$-dimensional Riemannian manifold without boundary (and without any other assumptions on the geometry of $M$, except for geodesic completeness, which we use in the last subsection). We assume that $V$ satisfies the assumptions (V1)--(V2) and the function $\phi$ and the vector field $X$ satisfy the assumptions (A1)--(A6) (or a certain subset of those assumptions, as specified in each lemma below).

We begin with a definition. Assume that $\phi\in C^{\infty}(M)$ satisfies (A2). For $\vep>0$ and $C_{\vep}$ as in (A2), define
\begin{equation}\label{E:uep}
U_{\vep}(x):=4(|d\phi(x)|^2+\vep^{-1}C_{\vep}).
\end{equation}

\begin{lemma}\label{L:lemma-aux-U} Let $U_{\vep}(x)$ be as in~(\ref{E:uep}), with $\phi$ satisfying (A2). Then, for all $x\in M$, we have
\begin{equation}\label{E:grad-uep}
|dU_{\vep}(x)|\leq \vep(U_{\vep}(x))^{3/2}.
\end{equation}
\end{lemma}
\begin{proof} Starting from~(\ref{E:uep}) and writing $|d\phi(x)|^2=\langle d\phi,d\phi\rangle$, we have
\begin{align}\label{E:similar-calc-8}
&|dU_{\vep}|=|8\langle {\nabla}^{lc}_{\bullet} d\phi,d\phi\rangle|=8|(\hess\phi)(\bullet, (d\phi)^{\sharp})|\nonumber\\
&\leq 8|\hess\phi||d\phi|\leq 8(\vep|d\phi|^2+C_{\vep})|d\phi|=2\vep U_{\vep}|d\phi|\leq \vep(U_{\vep})^{3/2},
\end{align}
where $(d\phi)^{\sharp}$ is the vector field corresponding to the form $d\phi$ via $g$, and in the last inequality we used
$2|d\phi|\leq U_{\vep}^{1/2}$, which follows from~(\ref{E:uep}).
\end{proof}

We now take a moment to describe a few more symbols. Given a section $u\in W^{1,1}_{\loc}(\vbn)\cap\lloc^{\infty}(\vbn)$, the notation $\nabla_{\bullet}u$ is understood as a $\vbn$-valued 1-form with the following property: the evaluation of this form at a smooth vector field $Y$ yields $\nabla_{Y}u$. With this property in mind, for $u\in W^{1,1}_{\loc}(\vbn)$ and $x\in M$, we define  $\omega_{u,x}\in T_{x}^*M$ and $\sigma_{u,x}\in T_{x}^*M$as
\begin{equation}\label{E:omega-x}
\omega_{u,x}:=\RE\langle\nabla_{\bullet}u,u\rangle_{\vbn_{x}},\quad \sigma_{u,x}:=\IM\langle\nabla_{\bullet}u,u\rangle_{\vbn_{x}}
\end{equation}
where $\langle\cdot,\cdot\rangle_{\vbn_{x}}$ is the fiberwise inner product in $\vbn_{x}$ with $x\in M$.

Thus, the assignments $x\mapsto \omega_{u,x}$ and $x\mapsto \sigma_{u,x}$ yield real-valued 1-forms $\omega_{u}$ and $\sigma_{u}$ on $M$ with $\lloc^1$-type regularity. To avoid overcomplicating our notations, we will simply use $\langle\cdot,\cdot\rangle$ and $|\cdot|$ as the inner product (respectively, the norm) in the fiber $\mathcal{W}_{x}$, with $\mathcal{W}_{x}$ referring to $\vbn_{x}$, $T_{x}^*M$, or $(T^*M\otimes \vbn)_{x}$.
We denote by $\chi_{\{u\neq 0\}}$ the indicator function of the set $\{x\in M\colon u(x)\neq 0\}$.

In the next lemma, proved in appendix A of~\cite{Mil-19}, we list some formulas which will help us organize our computations:

\begin{lemma}\label{L:om-sig} Assume that $u\in C^{\infty}(\vbn)$. Additionally, assume that $\xi\in\Omega^1(M)$ is a real-valued 1-form on $M$ and $Z$ is a smooth real vector field on $M$. Then, the following properties hold:
\begin{enumerate}
\item [(i)] $d|u|=\chi_{\{u\neq 0\}}|u|^{-1}\omega_{u}$.

\medskip

\item [(ii)] $Z|u|=\chi_{\{u\neq 0\}}|u|^{-1}\omega_{u}(Z)$, where $Z|u|$ indicates the action of $Z$ on the function $|u|$.

\medskip

\item [(iii)] $\langle \xi\otimes u, \nabla u\rangle= \langle u, \nabla_{{\xi}^{\sharp}}u\rangle$, where ${\xi}^{\sharp}$ is the vector field corresponding to the form $\xi$ via the metric $g$ of $M$.
\medskip

\item [(iv)] $\RE \langle \omega_{u}\otimes u, \nabla u\rangle=\langle \omega_{u},\omega_{u}\rangle =|\omega_{u}|^2$, for all $x\in M$.

\medskip

\item [(v)] $\RE \langle (\omega_{u}(Z)) u, \nabla_{Z} u\rangle=|\omega_{u}(Z)|^2$, for all $x\in M$. (Here, $|\cdot|$ on the right hand side is just the absolute value of a real number.)

\medskip




\end{enumerate}
\end{lemma}

In the proof of the coercive estimate~(\ref{E:coercive-1}), the following lemma plays the central part.

\begin{lemma}\label{L:lemma-u-hmax} Let $1<p<\infty$ and let $U_{\vep}(x)$ be as in~(\ref{E:uep}). Assume that $V$ is a section of $\End\vbn$ satisfying the assumption (V1) and the first inequality in (V2). Furthermore, assume that the hypotheses (A1)--(A4) are satisfied, with $\theta<p$. Then there exist $\vep_{0}>0$ and $\lambda_{0}>0$ such that
\begin{equation}\label{E:est-u-hmax}
\|U_{\vep_0}u\|_{p,\mu}\leq \frac{8p^2}{p-1}\|(\hvmaxp+\lambda)u\|_{p,\mu},
\end{equation}
for all $u\in\vcomp$ and all $\lambda\geq\lambda_0$.
\end{lemma}

\begin{proof}
We will prove the lemma in the case $p\geq 2$. At the end of the proof, we will make some remarks about the case $1<p<2$.

Our first goal is to estimate from below the term

\begin{equation}\label{E:l-u-1}
\RE (\nabla^{\dagger}\nabla u+\nabla_{(d\phi)^{\sharp}}u,U_{\vep}^{p-1}u|u|^{p-2})_{\mu},
\end{equation}
where $u\in\vcomp$ is arbitrary and $(\cdot,\cdot)_{\mu}$ is as in~(\ref{E:inner-mu}).

As the first step, we transform this expression using $v:=e^{-\phi/p}u$. One consequence of this transformation is
\begin{equation}\label{E:l-u-3}
U_{\vep}^{p-1}u|u|^{p-2}\,d\mu=U_{\vep}^{p-1}v|v|^{p-2}e^{-\phi/p}\,d\nu_{g},
\end{equation}
where $d\nu_{g}$ is the volume measure corresponding to $g$.

Before continuing further, we remind the reader that the symbol $(\cdot,\cdot)$ is as in~(\ref{E:inner-mu}) with the usual volume measure $d\mu=d\nu_{g}$.

Keeping in mind that $v\in\vcomp$ and $u=e^{\phi/p}v$ and referring to~(\ref{E:l-u-2-a}), we get
\begin{align}\label{E:l-u-4}
&(\nabla^{\dagger}\nabla u,U_{\vep}^{p-1}u|u|^{p-2})_{\mu}=(\nabla^{\dagger}\nabla v,U_{\vep}^{p-1}v|v|^{p-2})\nonumber\\
&-p^{-2}(v|d\phi|^2,U_{\vep}^{p-1}v|v|^{p-2})+p^{-1}(v\Delta \phi,U_{\vep}^{p-1}v|v|^{p-2})\nonumber\\
&-2p^{-1}(\nabla_{(d\phi)^{\sharp}}v,U_{\vep}^{p-1}v|v|^{p-2}).
\end{align}
For the second term of the expression~(\ref{E:l-u-1}), after looking at~(\ref{E:l-u-2-b}), we obtain
\begin{align}\label{E:l-u-5}
&(\nabla_{(d\phi)^{\sharp}}u,U_{\vep}^{p-1}u|u|^{p-2})_{\mu}=(\nabla_{(d\phi)^{\sharp}}v,U_{\vep}^{p-1}v|v|^{p-2})\nonumber\\
&+p^{-1}(v|d\phi|^2,U_{\vep}^{p-1}v|v|^{p-2}).
\end{align}
Combining~(\ref{E:l-u-4}) and~(\ref{E:l-u-5}) we get

\begin{align}\label{E:l-u-6}
&(\nabla^{\dagger}\nabla u+\nabla_{(d\phi)^{\sharp}}u,U_{\vep}^{p-1}u|u|^{p-2})_{\mu}=(\nabla^{\dagger}\nabla v,U_{\vep}^{p-1}v|v|^{p-2})\nonumber\\
&+(p^{-1}-p^{-2})(v|d\phi|^2,U_{\vep}^{p-1}v|v|^{p-2})+(1-2p^{-1})(\nabla_{(d\phi)^{\sharp}}v,U_{\vep}^{p-1}v|v|^{p-2})\nonumber\\
&+p^{-1}(v\Delta \phi,U_{\vep}^{p-1}v|v|^{p-2}).
\end{align}

The second step entails the integration by parts on the first and third term on the right hand side of~(\ref{E:l-u-6}). We will do this for the case $p\geq 2$; the case $1<p<2$ is handled similarly.

To start off, we use lemma~\ref{L:om-sig}(i) to get

\begin{equation}\label{E:l-u-7}
d (|v|^{p-2}v)=|v|^{p-2}\nabla v+(p-2)\chi_{\{v\neq 0\}}|v|^{p-4}v\otimes \omega_{v}
\end{equation}
where and $\omega_{v}$ and $\chi_{\{v\neq 0\}}$ are as in lemma~\ref{L:om-sig}.

As $v\in\vcomp$, the inequality
\begin{equation*}
|v|v|^{p-4}\omega_{v}|\leq |v|^{p-2}|\nabla v|
\end{equation*}
tells us that, in particular, $(|v|^{p-2}v)\in W^{1,1}_{\comp}(\vbn)$. As $U_{\vep}^{p-1}v\in C^{\infty}(\vbn)$, we have $(U_{\vep}^{p-1}|v|^{p-2}v)\in W^{1,1}_{\comp}(\vbn)$.

To get a clearer view of our calculations, we apply the formula~(\ref{E:ibp-1}) to each of the two designated terms separately.
The formulas
\begin{equation}\nonumber
dU_{\vep}^{p-2}=(p-1)U_{\vep}^{p-2}dU_{\vep}
\end{equation}
and~(\ref{E:l-u-7}) give, after using the usual product rule for $\nabla$,

\begin{align}\label{E:l-u-8}
&\nabla (U_{\vep}^{p-1}v|v|^{p-2})=U_{\vep}^{p-1}|v|^{p-2}\nabla v+(p-2)\chi_{\{v\neq 0\}}U_{\vep}^{p-1}|v|^{p-4}v\otimes\omega_{v}\nonumber\\
&+(p-1)U_{\vep}^{p-2}v|v|^{p-2}\otimes dU_{\vep},
\end{align}
which, after looking at~(\ref{E:ibp-1}), further leads to
\begin{align}\label{E:l-u-8-a}
&\RE (\nabla^{\dagger}\nabla v, U_{\vep}^{p-1}v|v|^{p-2})= \RE (\nabla v, U_{\vep}^{p-1}|v|^{p-2}\nabla v)\nonumber\\
&+(p-2)\RE (\nabla v, \chi_{\{v\neq 0\}}U_{\vep}^{p-1}|v|^{p-4}v\otimes\omega_{v})\nonumber\\
&+(p-1)\RE (\nabla v, U_{\vep}^{p-2}|v|^{p-2}v\otimes dU_{\vep}).
\end{align}

Using lemma~\ref{L:om-sig}(iv), we see that the second term on the right hand side of the last equality contains $|\omega_{v}|^2$, while
the first term contains $|\nabla v|^2|v|^{p-2}$. We pause to rewrite
\begin{align}\label{E:eq-1-est}
&|\nabla v|^2|v|^{p-2}=\chi_{\{v\neq 0\}}|v|^{p-4}|v|^2|\nabla v|^2\geq \chi_{\{v\neq 0\}}|v|^{p-4}|\langle \nabla_{\bullet}v, v\rangle|^2\nonumber\\
&= \chi_{\{v\neq 0\}}|v|^{p-4}(|\omega_{v}|^2+|\sigma_{v}|^2),
\end{align}
where the last inequality follows by the definition of $\omega_{v}$ and $\sigma_{v}$, yielding
\begin{align}\label{E:eq-1-est-1}
&\RE (\nabla v, U_{\vep}^{p-1}|v|^{p-2}\nabla v)\geq (\chi_{\{v\neq 0\}}\sigma_{v}, U_{\vep}^{p-1}|v|^{p-4}\sigma_{v})+\nonumber\\
&(\chi_{\{v\neq 0\}}\omega_{v}, U_{\vep}^{p-1}|v|^{p-4}\omega_{v}).
\end{align}
Taking into account~(\ref{E:eq-1-est-1}),
we have the following lower estimate for the first term on the right hand side of~(\ref{E:l-u-6}):
\begin{align}\label{E:l-u-9}
&\RE (\nabla^{\dagger}\nabla v, U_{\vep}^{p-1}v|v|^{p-2})\geq (p-1)(\chi_{\{v\neq 0\}}\omega_{v}, U_{\vep}^{p-1}|v|^{p-4}\omega_{v})\nonumber\\
&+(\chi_{\{v\neq 0\}}\sigma_{v}, U_{\vep}^{p-1}|v|^{p-4}\sigma_{v})+(p-1)(\omega_{v}, U_{\vep}^{p-2}|v|^{p-2}dU_{\vep}),
\end{align}
where we used the definition of $\omega_{v}$ to rewrite the terms with coefficients $p-1$ and $p-2$ on the right hand side of~(\ref{E:l-u-8-a}).

We now turn to the third term on the right hand side of~(\ref{E:l-u-6}). Looking at the part of~(\ref{E:ibp-1}) containing $\phi$ and referring to~(\ref{E:l-u-8}) we have
\begin{align}
&\RE (\nabla_{(d\phi)^{\sharp}}v,U_{\vep}^{p-1}v|v|^{p-2})=\RE (v,(\Delta\phi)U_{\vep}^{p-1}v|v|^{p-2})-\RE (v, \nabla_{(d\phi)^{\sharp}}(U_{\vep}^{p-1}v|v|^{p-2}))\nonumber\\
&=\RE (v,(\Delta\phi)U_{\vep}^{p-1}v|v|^{p-2})-(p-1)\RE (v, [(d\phi)^{\sharp}U_{\vep}]U_{\vep}^{p-2}v|v|^{p-2})\nonumber\\
&-\RE (v, U_{\vep}^{p-1}|v|^{p-2}\nabla_{(d\phi)^{\sharp}}v)-(p-2)\RE (v,\omega_{v}((d\phi)^{\sharp})\chi_{\{v\neq 0\}} U_{\vep}^{p-1}v|v|^{p-4}),\nonumber
\end{align}
where $[(d\phi)^{\sharp}U_{\vep}]$ indicates the action of the vector field $(d\phi)^{\sharp}$ on the function $U_{\vep}$.

Keeping in mind the definition of $\omega_{v}$, note that the item on the left hand side of the first equality and the last two items on the right hand side of the second equality are of the same type. Therefore, after rearranging we get
\begin{align}\label{E:l-u-11}
&\RE (\nabla_{(d\phi)^{\sharp}}v,U_{\vep}^{p-1}v|v|^{p-2})=p^{-1}\RE (v,(\Delta\phi)U_{\vep}^{p-1}v|v|^{p-2})\nonumber\\
&-(1-p^{-1})\RE (v, [(d\phi)^{\sharp}U_{\vep}]U_{\vep}^{p-2}v|v|^{p-2}).
\end{align}

Combining~(\ref{E:l-u-6}),~(\ref{E:l-u-9}), and ~(\ref{E:l-u-11}) we obtain
\begin{align}\label{E:l-u-12}
&\RE(\nabla^{\dagger}\nabla u+\nabla_{(d\phi)^{\sharp}}u,U_{\vep}^{p-1}u|u|^{p-2})_{\mu}\geq (p-1)(\chi_{\{v\neq 0\}}\omega_{v}, U_{\vep}^{p-1}|v|^{p-4}\omega_{v})\nonumber\\
&+(\chi_{\{v\neq 0\}}\sigma_{v}, U_{\vep}^{p-1}|v|^{p-4}\sigma_{v})+(p-1)(\omega_{v}, U_{\vep}^{p-2}|v|^{p-2}dU_{\vep})\nonumber\\
&+(p^{-1}-p^{-2})(v|d\phi|^2,U_{\vep}^{p-1}v|v|^{p-2})\nonumber\\
&-(1-2p^{-1})(1-p^{-1})(v, [(d\phi)^{\sharp}U_{\vep}]U_{\vep}^{p-2}v|v|^{p-2})\nonumber\\
&+2p^{-2}(p-1)(v,(\Delta\phi)U_{\vep}^{p-1}v|v|^{p-2}),
\end{align}
where we dropped the designation ``Re" from those terms $(\cdot,\cdot)$ that are real by default.

Having completed the integration by parts procedure, we now further estimate~(\ref{E:l-u-12}) from below. As the second term on the right hand side of~(\ref{E:l-u-12}) is non-negative, we drop it right away. For the third term, after estimating the integrand
\begin{equation}\nonumber
(p-1)|\langle\omega_{v}, U_{\vep}^{p-2}|v|^{p-2}dU_{\vep}\rangle|\leq (p-1)|\omega_{v}|U_{\vep}^{p-2}|dU_{\vep}||v|^{p-2},
\end{equation}
we use the inequality~(n1) with  $\kappa_1=\kappa_2=2$ and
\begin{equation}\nonumber
a=(2(p-1))^{1/2}U_{\vep}^{(p-1)/2}|\omega_{v}||v|^{p/2-2},\quad b=2^{-1/2}(p-1)^{1/2}U_{\vep}^{(p-3)/2}|dU_{\vep}||v|^{p/2}.
\end{equation}
As a result, we obtain the following lower estimate of~(\ref{E:l-u-12}):
\begin{align}\label{E:l-u-13}
&\RE(\nabla^{\dagger} u+\nabla_{(d\phi)^{\sharp}}u,U_{\vep}^{p-1}u|u|^{p-2})_{\mu}\geq
-\frac{p-1}{4}\int_{M}U_{\vep}^{p-3}|dU_{\vep}|^2|v|^{p}\,d\nu_{g}\nonumber\\
&+(p^{-1}-p^{-2})(v|d\phi|^2,U_{\vep}^{p-1}v|v|^{p-2})\nonumber\\
&-(1-2p^{-1})(1-p^{-1})(v, [(d\phi)^{\sharp}U_{\vep}]U_{\vep}^{p-2}v|v|^{p-2})\nonumber\\
&+2p^{-2}(p-1)(v,(\Delta\phi)U_{\vep}^{p-1}v|v|^{p-2}),
\end{align}
where the term corresponding to $a$ canceled with the first term on the right hand side of~(\ref{E:l-u-12}).

We now estimate from below the terms on the right hand side of~(\ref{E:l-u-13}), starting with the first one,
\begin{equation}\label{E:first-compos-est}
-\frac{p-1}{4}\int_{M}U_{\vep}^{p-3}|dU_{\vep}|^2|v|^{p}\,d\nu_{g}\geq -\frac{(p-1)\vep^2}{4}\int_{M}U_{\vep}^{p}|v|^{p}\,d\nu_{g}=-\frac{(p-1)\vep^2}{4}\int_{M}U_{\vep}^{p}|u|^{p}\,d\mu,
\end{equation}
where we used~(\ref{E:grad-uep}),~(\ref{E:meas-mu}), and the property $v=e^{-\phi/p}u$.

Before continuing with our lower estimate for the remaining terms on the right hand side of~(\ref{E:l-u-13}), note that
from the definition of $U_{\vep}$ we have
\begin{equation}\label{E:l-u-14}
|d\phi|\leq \frac{U_{\vep}^{1/2}}{2},\quad |d\phi|^2=\frac{U_{\vep}}{4}-\frac{C_{\vep}}{\vep}.
\end{equation}

Furthermore by the assumption (A2) and~(\ref{E:lap-hess}) we have
\begin{equation}\label{E:l-u-15}
|\Delta \phi|\leq \frac{\vep\sqrt{n}U_{\vep}}{4}.
\end{equation}

Using the equality from~(\ref{E:l-u-14}), we can rewrite the second term on the right hand side of~(\ref{E:l-u-13}) as
\begin{equation}\label{E:l-u-15-a}
\frac{p-1}{4p^2}\int_{M}U_{\vep}^{p}|u|^{p}\,d\mu-\frac{(p-1)C_{\vep}}{p^2\vep}\int_{M}U_{\vep}^{p-1}|u|^p\,d\mu,
\end{equation}
where we used the formula $v=e^{-\phi/p}u$ and~(\ref{E:meas-mu}).

Keeping in mind $v=e^{-\phi/p}u$ and~(\ref{E:meas-mu}) and using the inequality in~(\ref{E:l-u-14})
and the estimate~(\ref{E:grad-uep}), we get
\begin{align}\label{E:l-u-16}
&-(1-2p^{-1})(1-p^{-1})(v, [(d\phi)^{\sharp}U_{\vep}] U_{\vep}^{p-2}v|v|^{p-2})\nonumber\\
&\geq -\frac{(p-1)(p-2)\vep}{2p^2}\int_{M}U_{\vep}^{p}|u|^{p}\,d\mu,
\end{align}
and, finally, from~(\ref{E:l-u-15}) we obtain
\begin{align}\label{E:l-u-17}
&2p^{-2}(p-1)(v,(\Delta\phi)U_{\vep}^{p-1}v|v|^{p-2})\nonumber\\
&\geq -\frac{(p-1)\vep\sqrt{n}}{2p^2}\int_{M}U_{\vep}^{p}|u|^{p}\,d\mu.
\end{align}

Combining~(\ref{E:l-u-13}),~(\ref{E:first-compos-est}),~(\ref{E:l-u-15-a}),~(\ref{E:l-u-16}), and~(\ref{E:l-u-17}), we have
\begin{align}\label{E:l-u-18}
&\RE(\nabla^{\dagger}\nabla u+\nabla_{(d\phi)^{\sharp}}u,U_{\vep}^{p-1}u|u|^{p-2})_{\mu}\nonumber\\
&\geq \frac{p-1}{4p^2}(1-p^2\vep^2-2(p-2)\vep-2\vep\sqrt{n})\int_{M}U_{\vep}^{p}|u|^{p}\,d\mu\nonumber\\
&-\frac{(p-1)C_{\vep}}{p^2\vep}\int_{M}U_{\vep}^{p-1}|u|^p\,d\mu,
\end{align}
and this accomplishes the goal set at the beginning of the proof.

Our next goal is to estimate from below the term $\RE(\nabla_{X}u+Vu,U_{\vep}^{p-1}u|u|^{p-2})_{\mu}$. We will reach this goal after doing integration by parts, for which we refer to~(\ref{E:ibp-2}):
\begin{align}\label{E:l-u-19}
&\RE (-\nabla_{X}u, U_{\vep}^{p-1}u|u|^{p-2})_{\mu}=\RE (-\nabla_{X}u,U_{\vep}^{p-1}u|u|^{p-2}e^{-\phi})\nonumber\\
&=-q^{-1}\RE(\nabla_{X}u,U_{\vep}^{p-1}u|u|^{p-2}e^{-\phi})+p^{-1}\RE (u, X[U_{\vep}^{p-1}u|u|^{p-2}e^{-\phi}])\nonumber\\
&+p^{-1}(u,(\div X)U_{\vep}^{p-1}u|u|^{p-2}e^{-\phi}),
\end{align}
where $q$ is the H\"older conjugate of $p$.

Looking at~(\ref{E:l-u-8}) with $u$ in place of $v$ and keeping in mind $de^{-\phi}=-e^{-\phi}d\phi$ and the formula (ii) from lemma~\ref{L:om-sig}, we have
\begin{align}\label{E:l-u-20}
&p^{-1}\RE (u, X[U_{\vep}^{p-1}u|u|^{p-2}e^{-\phi}])=p^{-1}(p-1)\RE (u, (XU_{\vep})U_{\vep}^{p-2}u|u|^{p-2}e^{-\phi})\nonumber\\
&+p^{-1}\RE (u, (\nabla_{X}u)U_{\vep}^{p-1}|u|^{p-2}e^{-\phi})\nonumber\\
&+p^{-1}(p-2)(u, \chi_{\{v\neq 0\}}(\omega_{u}(X))U_{\vep}^{p-1}u|u|^{p-4}e^{-\phi})\nonumber\\
&-p^{-1}(u,(X\phi) U_{\vep}^{p-1}u|u|^{p-2}e^{-\phi}).
\end{align}
Remembering the definition~(\ref{E:omega-x}) and combining the second and the third term on the right hand side of~(\ref{E:l-u-20}), we obtain
\begin{equation}\nonumber
q^{-1}\RE (u, (\nabla_{X}u)U_{\vep}^{p-1}|u|^{p-2}e^{-\phi}),
\end{equation}
where $q^{-1}=1-p^{-1}$.

With the help of this information, looking at~(\ref{E:l-u-19}) and recalling the condition $V\geq h$, we get
\begin{align}\label{E:l-u-21}
&\RE(-\nabla_{X}u+Vu,U_{\vep}^{p-1}u|u|^{p-2})_{\mu}\geq p^{-1}(p-1)\RE (u, (XU_{\vep})U_{\vep}^{p-2}u|u|^{p-2})_{\mu}\nonumber\\
&+\int_{M}\left[p^{-1}(\div X-X\phi)+h\right]U_{\vep}^{p-1}|u|^p\,d\mu.
\end{align}

Before making a lower estimate, note that (A3) implies
\begin{equation}\label{E:aux-a3}
p^{-1}(\div X)-p^{-1}X\phi+h\geq (1-p^{-1}\theta)h-p^{-1}\beta_1.
\end{equation}

Furthermore from (A4) and~(\ref{E:grad-uep}) we get

\begin{equation}\label{E:aux-a4}
|(XU_{\vep})U_{\vep}^{p-2}|\leq |X||dU_{\vep}|U_{\vep}^{p-2}\leq \kappa\vep(|d\phi|^2+h+\beta_2)^{1/2}U_{\vep}^{p-1/2}.
\end{equation}

Furthermore,
\begin{align}\label{E:l-u-22}
&(|d\phi|^2+h+\beta_2)^{1/2}U_{\vep}^{p-1/2}|u|^p\leq (|d\phi|+(h+\beta_2)^{1/2})U_{\vep}^{p-1/2}|u|^p\nonumber\\
&\leq 2^{-1}U_{\vep}^{p}|u|^p+(h+\beta_2)^{1/2}U_{\vep}^{p-1/2}|u|^p\leq  2^{-1}U_{\vep}^{p}|u|^p\nonumber\\
&+(h+\beta_2)U_{\vep}^{p-1}|u|^p+4^{-1}U_{\vep}^{p}|u|^p,
\end{align}
where in the first inequality we used (n3) with $p=2$, $a=|d\phi|^2$, $b=h+\beta_2$ and $c=0$, in the second inequality we used~(\ref{E:l-u-14}), and in the third inequality we used~(n1) with $\kappa_1=\kappa_2=2$ and
\begin{equation*}
a=2^{1/2}(h+\beta_2)^{1/2}U_{\vep}^{(p-1)/2}|u|^{p/2},\quad b=2^{-1/2}U_{\vep}^{p/2}|u|^{p/2}.
\end{equation*}

Aided by~(\ref{E:aux-a4}),~(\ref{E:l-u-22}) and~(\ref{E:aux-a3}) we make a lower estimate in~(\ref{E:l-u-21}):
\begin{align}\label{E:l-u-23}
&\RE(-\nabla_{X}u+Vu,U_{\vep}^{p-1}u|u|^{p-2})_{\mu}\geq p^{-1}(p-\theta-(p-1)\kappa\vep)\int_{M}hU_{\vep}^{p-1}|u|^p\,d\mu\nonumber\\
&-p^{-1}(\beta_1+(p-1)\kappa\vep\beta_2)\int_{M}U_{\vep}^{p-1}|u|^p\,d\mu\nonumber\\
&-(4p)^{-1}3(p-1)\kappa\vep \int_{M}U_{\vep}^{p}|u|^p\,d\mu.
\end{align}

We now look at the coefficient of the integral with the integrand $hU_{\vep}^{p-1}|u|^p$ in~(\ref{E:l-u-23}) and the combined coefficient of the integrals with the integrand $U_{\vep}^{p}|u|^p$ in~(\ref{E:l-u-23}) and~(\ref{E:l-u-18}). Remembering the condition $\theta<p$, it is easy to see that we can choose a sufficiently small $\vep_0>0$ such that for all $0<\vep\leq\vep_0$ the following two conditions are satisfied:
\begin{equation}\nonumber
1-p^2\vep^2-2(p-2)\vep-2\vep\sqrt{n}-3p\kappa\vep\geq\frac{1}{2}, \quad p-\theta-(p-1)\kappa\vep>0.
\end{equation}

With this in mind, the estimates~(\ref{E:l-u-23}) and~(\ref{E:l-u-18}) yield
\begin{align}
&\RE(\hvmaxp u,U_{\vep}^{p-1}u|u|^{p-2})_{\mu}\geq \frac{p-1}{8p^2}\int_{M}U_{{\vep}_0}^{p}|u|^{p}\,d\mu\nonumber\\
&-\left(\frac{\beta_1}{p}+\frac{(p-1)\kappa\vep_0\beta_2}{p}+\frac{(p-1)C_{{\vep}_0}}{p^2{\vep}_0}\right)\int_{M}U_{\vep_0}^{p-1}|u|^p\,d\mu\nonumber.
\end{align}

Denoting by $\lambda_0$ the coefficient of the integral with integrand $U_{\vep_0}^{p-1}|u|^p$, we see that for all $\lambda\geq\lambda_0$ and for all $u\in\vcomp$ we have
\begin{align}\label{E:l-u-24}
&\frac{p-1}{8p^2}\|U_{{\vep}_0}u\|_{p,\mu}^p\leq \RE((\hvmaxp+\lambda)u,U_{{\vep}_0}^{p-1}u|u|^{p-2})_{\mu}\nonumber\\
&\leq \|(\hvmaxp+\lambda)u\|_{p,\mu}\|U_{{\vep}_0}u\|_{p,\mu}^{p-1},
\end{align}
where in the second estimate we used H\"older's inequality.

Rearranging~(\ref{E:l-u-24}) we obtain~(\ref{E:est-u-hmax}), which concludes the proof of the lemma for the case $p\geq 2$.

The case $1<p<2$ can be handled in a similar manner if instead of $u|u|^{p-2}$ we consider $(|u|^2+\tau)^{p/2-1}u$, with $\tau>0$.
For brevity we will not include the details here. Instead, we remark that in the context $1<p<2$ an argument with the same flavor was carried out in section 2 of~\cite{So-12} for the Schr\"odinger operator with a drift on $\RR^n$, that is, the case $\phi\equiv 0$ in~(\ref{E:def-H}); see also~\cite{Mil-19} for the covariant Schr\"odinger operator (with a drift) on manifolds.
\end{proof}

\subsection{Proof of Part (ii)  of Theorem~\ref{T:main-1}}\label{SS:coercive-1}

We will prove the result in the case $p\geq 2$. We just remark that the case $1<p<2$ can be handled by a similar procedure, as indicated at the end of the proof of lemma~\ref{L:lemma-u-hmax}.

If we replace $h$ with $h+c_1$, where $c_1>0$ is a number, we may assume that (A5) is satisfied with $\beta_3=0$. This will not affect the generality of the argument, as we can replace $\lambda_1$ with $\lambda_1+c_1$.

We start with
\begin{align}\label{E:l-u-25}
&\RE (\nabla^{\dagger}\nabla u+\nabla_{(d\phi)^{\sharp}}u, h^{p-1}u|u|^{p-2})_{\mu}=\RE (\nabla^{\dagger}\nabla u,h^{p-1}u|u|^{p-2})_{\mu}\nonumber\\
&+\RE (\nabla_{(d\phi)^{\sharp}}u,h^{p-1}u|u|^{p-2})_{\mu}.
\end{align}
Remembering $d\mu=e^{-\phi}\nu_{g}$, where $\nu_{g}$ is the volume measure, and performing the integration by parts on the term $(\cdot,\cdot)_{\mu}$ containing $\nabla^{\dagger}\nabla u$, we observe that the product rule for $\nabla(h^{p-1}u|u|^{p-2}e^{-\phi})$ yields four terms: (i) the item $-h^{p-1}u|u|^{p-2}e^{-\phi}\otimes d\phi$ corresponds to
\begin{equation*}
-(\nabla u, h^{p-1}u|u|^{p-2}e^{-\phi}\otimes d\phi),
\end{equation*}
which, due to lemma~\ref{L:om-sig}(iii), cancels the second item on the right hand side of~(\ref{E:l-u-25}); (ii) the remaining three terms coming from $\nabla(h^{p-1}u|u|^{p-2})e^{-\phi}$ lead to the three items $(\cdot,\cdot)_{\mu}$ which look the same as those the right hand side of~(\ref{E:l-u-8-a}) if in the latter equation we change $(\cdot,\cdot)$ to $(\cdot,\cdot)_{\mu}$ and replace $v$ by $u$ and $U_{\vep}$ by $h$. Therefore, using the same reasoning that led from~(\ref{E:l-u-8-a}) to~(\ref{E:l-u-9}), we get
\begin{align}\label{E:l-u-26}
&\RE (\nabla^{\dagger}\nabla u+\nabla_{(d\phi)^{\sharp}}u, h^{p-1}u|u|^{p-2})_{\mu}\geq(p-1)(\chi_{\{u\neq 0\}}\omega_{u}, h^{p-1}|u|^{p-4}\omega_{u})_{\mu}\nonumber\\
&+(\chi_{\{u\neq 0\}}\sigma_{u}, h^{p-1}|u|^{p-4}\sigma_{u})+(p-1)(\omega_{u}, h^{p-2}|u|^{p-2}dh)_{\mu},
\end{align}
where $\omega_u$ and $\sigma_u$ are as in~(\ref{E:omega-x}).

We now drop the (non-negative) second term on the right hand side of~(\ref{E:l-u-26}) and continue estimating from below with the help of the inequality (keeping in mind (A5) with $\beta_3=0$)
\begin{equation}\nonumber
(p-1)|\langle\omega_{u}, h^{p-2}|u|^{p-2}dh\rangle|\leq (p-1)h^{p-2}|\omega_{u}||dh||u|^{p-2}\leq (p-1)\gamma h^{p-1/2}|\omega_{u}||u|^{p-2}
\end{equation}
and the inequality~(n1) with  $\kappa_1=\kappa_2=2$ and
\begin{equation}\nonumber
a=(2(p-1))^{1/2}h^{(p-1)/2}|\omega_{u}||u|^{p/2-2},\quad b=2^{-1/2}\gamma(p-1)^{1/2}h^{p/2}|u|^{p/2}.
\end{equation}
Consequently, we obtain
\begin{equation}\label{E:l-u-27}
\RE (\nabla^{\dagger}\nabla u+\nabla_{(d\phi)^{\sharp}}u, h^{p-1}u|u|^{p-2})_{\mu}\geq-\frac{(p-1)\gamma^2}{4}\int_{M}h^{p}|u|^p\,d\mu.
\end{equation}

To perform the integration by parts in
\begin{equation*}
\RE (-\nabla_{X}u, h^{p-1}u|u|^{p-2})_{\mu}+ \RE (Vu, h^{p-1}u|u|^{p-2})_{\mu},
\end{equation*}
we observe that the first item is the same as the term located on the left hand side of~(\ref{E:l-u-19}) (provided that we replace $U_{\vep}$ by $h$). Therefore, we can simply copy~(\ref{E:l-u-21}) with the indicated replacement:
\begin{align}\label{E:l-u-28}
&\RE(-\nabla_{X}u+Vu,h^{p-1}u|u|^{p-2})_{\mu}\geq p^{-1}(p-1)\RE (u, (Xh)h^{p-2}u|u|^{p-2})_{\mu}\nonumber\\
&+\int_{M}\left[p^{-1}(\div X-X\phi)+h\right]h^{p-1}|u|^p\,d\mu.
\end{align}
It turns out that the lower estimate~(\ref{E:l-u-23}) is also applicable in this context with the following changes: replace $\vep$ by $\gamma$, replace $U_{\vep}$ by $h$, and make the last term on the right hand side of~(\ref{E:l-u-23}) look like this:
\begin{equation*}
-(2p)^{-1}(p-1)\kappa\gamma \int_{M}h^{p-1}|d\phi|^2|u|^p\,d\mu.
\end{equation*}
The mentioned changes come from the following two sources:  (i) replacement of~(\ref{E:aux-a4}) by
\begin{equation}\label{E:aux-a5}
|(Xh)h^{p-2}|\leq |X||dh|h^{p-2}\leq \kappa\gamma(|d\phi|^2+h+\beta_2)^{1/2}h^{p-1/2},
\end{equation}
where we used (A5) (as opposed to (\ref{E:l-u-14}));
\noindent (ii) replacement of~(\ref{E:l-u-22}) by
\begin{align}\nonumber
&(|d\phi|^2+h+\beta_2)^{1/2}h^{p-1/2}|u|^p\leq (|d\phi|+(h+\beta_2)^{1/2})h^{p-1/2}|u|^p\nonumber\\
&\leq |d\phi|h^{p-1/2}|u|^p+(h+\beta_2)^{1/2}h^{p-1/2}|u|^p\leq  2^{-1}|d\phi|^2h^{p-1}|u|^p+2^{-1}h^{p}|u|^p\nonumber\\
&+2^{-1}(h+\beta_2)h^{p-1}|u|^p+2^{-1}h^{p}|u|^p,
\end{align}
in which, similarly to~(\ref{E:l-u-22}), we used a procedure based on the inequality~(n1).

Therefore,
\begin{align}\label{E:l-u-29}
&\RE(-\nabla_{X}u+Vu,h^{p-1}u|u|^{p-2})_{\mu}\geq p^{-1}(p-\theta-(p-1)\kappa\gamma)\int_{M}h^p|u|^p\,d\mu\nonumber\\
&-p^{-1}(\beta_1+2^{-1}(p-1)\kappa\gamma\beta_2)\int_{M}h^{p-1}|u|^p\,d\mu\nonumber\\
&-(2p)^{-1}(p-1)\kappa\gamma\int_{M}h^{p-1}|d\phi|^2|u|^p\,d\mu.
\end{align}

Putting $\rho_0:=p^{-1}(\beta_1+2^{-1}(p-1)\kappa\gamma\beta_2)$ and combining~(\ref{E:l-u-27}) and~(\ref{E:l-u-29}), we get
\begin{align}\nonumber
&\RE((\hvmaxp +\lambda)u,h^{p-1}u|u|^{p-2})_{\mu}\nonumber\\
&\geq (1-p^{-1}\theta-p^{-1}(p-1)\kappa\gamma-4^{-1}(p-1)\gamma^2)\int_{M}h^p|u|^p\,d\mu\nonumber\\
&-(2p)^{-1}(p-1)\kappa\gamma\int_{M}h^{p-1}|d\phi|^2|u|^p\,d\mu,
\end{align}
for all $\lambda\geq\rho_0$.

Putting $\lambda_1:=\max\{\rho_0,\lambda_0\}$, where $\lambda_0$ is as in the proof of lemma~\ref{L:lemma-u-hmax}, the last estimate leads to
\begin{align}\label{E:l-u-30}
&(1-p^{-1}\theta-p^{-1}(p-1)\kappa\gamma-4^{-1}(p-1)\gamma^2)\|hu\|_{p,\mu}^{p}\nonumber\\
&\leq \|(\hvmaxp +\lambda)u\|_{p,\mu}\|hu\|_{p,\mu}^{p-1}+ (2p)^{-1}(p-1)\kappa\gamma\int_{M}h^{p-1}|d\phi|^2|u|^p\,d\mu\nonumber\\
&\leq \|(\hvmaxp +\lambda)u\|_{p,\mu}\|hu\|_{p,\mu}^{p-1}+ (2p)^{-1}(p-1)\kappa\gamma\||d\phi|^2u\|_{p,\mu}\|hu\|_{p,\mu}^{p-1}.
\end{align}
for all $\lambda\geq\lambda_1$, where in the last two estimates we used H\"older's inequality.

Before going further, we note that the hypothesis~(\ref{E:ineq-hyp-main-1}) ensures the positivity of the coefficient of $\|hu\|_{p,\mu}^{p}$.
Finally, referring to~(\ref{E:l-u-14}), we see that
\begin{align}\nonumber
&(2p)^{-1}(p-1)\kappa\gamma\||d\phi|^2u\|_{p,\mu}^p\leq (8p)^{-1}(p-1)\kappa\gamma\|U_{\vep_0}u\|_{p,\mu}\nonumber\\
&\leq p\kappa\gamma \|(\hvmaxp +\lambda)u\|_{p,\mu},\nonumber
\end{align}
where in the last inequality we used~(\ref{E:est-u-hmax}).

Combining the last estimate with~(\ref{E:l-u-30}) leads to~(\ref{E:coercive-1}) for all $u\in\vcomp$. Using a closure argument and part (i) of theorem~\ref{T:main-1}, we see that the estimate~(\ref{E:coercive-1}) is true for all $u\in\dom(\hvmaxp)$.  $\hfill\square$

\section{Proof of Corollary~\ref{C:main-1}}\label{S:proof-cor-1}
We specialize theorem~\ref{T:main-1} to the case $\phi=0$ and $X=0$. Clearly, the assumptions (A1)--(A5) are satisfied with $\kappa=\theta=\beta_1=\beta_2=0$. Using the condition $V\leq \zeta h$ we have $|Vu|\leq |V||u|\leq \zeta h|u|$, where $|V|$ is the norm of the endomorphism $V(x)\colon\vbn_x\to\vbn_x$. With this information, the estimate~(\ref{E:coercive-1}) leads to
\begin{equation*}
\|V u\|_{p}\leq \zeta \|hu\|_{p}\leq C\|(\hvmaxp+\lambda)u\|_{p},
\end{equation*}
for all $u\in \dom(\hvmaxp)$, where $C$ is some constant. Thus, we showed that if $u\in\dom(\hvmaxp)$ then $Vu\in L^p(\vbn)$, that is, the expression $\nabla^{\dagger}\nabla+ V$ is separated in $L^p(\vbn)$, $1<p<\infty$. $\hfill\square$

\section{More coercive estimates}\label{More Coerc}

\begin{prop}\label{more_coer_est_1} Let $M$ be a Riemannian manifold without boundary. Assume that $r_{\textrm{inj}}(M)>0$ and $\|\ricc_{M}\|_{\infty}<\infty$. Let $1<p<\infty$ and let $d\mu$ be as in~(\ref{E:meas-mu}). Assume that $V$ is a real-valued function of class $C^1(M)$. Additionally, assume that the vector field $X$ and the function $\phi$ from the expression~(\ref{E:def-H-d}) satisfy the hypotheses (A1)--(A6), with $h$ replaced by $V$ in (A3)--(A5). Then there exist constants $\widetilde{C}'$, $\widetilde{C}'' > 0$ (depending on $p$,
$n = dim(M)$ and the constants in (A2)--(A5)) such that
\begin{equation}\label{more_coer_est_1_1}
\vert\vert (\vert d\phi\vert^2 + V + 1)^{1/2}\vert du\vert
\vert\vert_{p, \mu} \leq \widetilde{C}'\vert\vert \lambda_1u + \hvmaxp u
\vert\vert_{p, \mu}
\end{equation}
\begin{equation}\label{more_coer_est_1_2}
\vert\vert \hess(u)\vert\vert_{p, \mu} \leq \widetilde{C}''
\vert\vert \lambda_1u + \hvmaxp u\vert\vert_{p, \mu},
\end{equation}
for all $u \in C_c^{\infty}(M)$, where $\lambda_1$ is the constant determined by theorem \ref{T:main-1}(ii).
\end{prop}

This proposition can be seen as a manifold variant of lemma 3.4 in \cite{So-Yo-13} whose proof depends on the global maximal elliptic regularity result (see corollary 9.10. in the book~\cite{G-T}).  In the case of a non-compact Riemannian manifold such a result is only true locally. To circumvent this issue, we require that the manifold $M$ satisfy the $L^p$-Calder\'on--Zygmund Inequality (CZ(p)), as discussed in section~\ref{SS:c-z} above.

The proof of the proposition will be given at the end of this section, after we have established a few key lemmas.

\begin{lemma}\label{more_coer_est_2}
Let $R > 0$ and let $B_R^E(0)$ denote a Euclidean ball in $\R^n$ of radius $R$ centered at $0\in\RR^n$.
Let $U \in C^1(\R^n)$ such that $U \geq c_0 > 0$ and
$\vert d U\vert \leq \gamma U^{3/2}$, for some constants
$c_0$, $\gamma > 0$. Then there exists a constant $\vep_1 > 0$ such that for all $0 < \vep \leq \vep_1$ and
$1 < p <\infty$ and $v \in C_c^{\infty}(\R^n)$, we have
\begin{equation}\label{E:eucl-est-1}
\vert\vert \chi_{B_{R}^E(0)}U^{1/2}d v\vert\vert^p_{p} \leq
C^p\vep^p\vert\vert\chi_{B_{R}^E(0)}\Delta v\vert\vert^p_{p} +
A_{\vep}^p\vert\vert \chi_{B_{R}^E(0)}U v\vert\vert_{p}^p,
\end{equation}
where $C$ is a constant independent of $\varepsilon$ (but depending on $p$ and $n$) and $A_{\vep}$ is a constant depending on $\varepsilon$, $p$, $n$. Here, $\chi_{B_{R}^E(0)}$ denotes the characteristic function of $B_R^E(0)$, $\Del$ is the standard Laplacian on $\RR^n$, and $\|\cdot\|_{p}$ is the usual norm in $L^p(\RR^n)$.
\end{lemma}
\begin{proof}
Let $x_0$ be a fixed point in $B_R^E(0)$ and let $B_R^E(x_0)$ denote the Euclidean ball of radius $R$ centered at $x_0$. For $U$ as in the hypothesis of this lemma and for $r := \frac{1}{\gamma}U(x_0)^{-1/2}$, the authors of~\cite{MPR-05} established the following estimate (see equation (2.5) in~\cite{MPR-05}): there exists a constant $K>0$ such that for all $\varepsilon>0$ and all $v\in C_c^{\infty}(\R^n)$ we have
\begin{equation}\label{more_coer_est_2_1}
\vert\vert \chi_{B_{r/2}^E(x_0)}U^{1/2}d v\vert\vert_{p} \leq
\vep \vert\vert\chi_{B_{r}^E(x_0)}\Delta v\vert\vert_{p} +
\vep K\vert\vert \chi_{B_{r}^E(x_0)}U^{1/2} d v\vert\vert_{p} +
\frac{K}{\vep}\vert\vert \chi_{B_{r}^E(x_0)}U v\vert\vert_{p}.
\end{equation}

From our assumptions on $U$ we see that $\vert\frac{1}{\gamma}d (U^{-1/2})\vert
\leq \frac{1}{2}$. If we then consider the covering
$\mathscr{C}:=\{B^E_{r_x}(x): x \in B_R(0)\}$ of $B_R^E(0)$ by the balls $B^E_{r_x}(x)$ centered at $x\in B_R^E(0)$ and with radius
$r_x := \frac{1}{\gamma}U^{-1/2}(x)$, we have that
\begin{align*}
\vert r_x\vert &\leq \frac{c_0^{-1/2}}{\gamma}, \\
\vert r_x - r_y\vert &= \frac{1}{\gamma}\vert U^{-1/2}(x) - U^{-1/2}(y)\vert
\leq \frac{1}{\gamma}\vert d U^{-1/2}(\widehat{\theta})\vert\vert x - y\vert
\leq \frac{1}{2}\vert x - y\vert,
\end{align*}
where $\widehat{\theta}$ lies on the straight line joining the point $x$ to $y$, and
in the second estimate we have used the mean value theorem.

Thus, the covering $\mathscr{C}$ satisfies the hypotheses of lemma 2.2 in \cite{MPR-05}, which is a variant of Besicovitch Covering Theorem (see~\cite{Guzman-75} for a thorough discussion of this topic). According to the mentioned lemma from  \cite{MPR-05}, there exists a countable collection $\widetilde{\mathscr{C}}:=\{B^E_{\frac{r_{x_k}}{2}}(x_k)\}$, where $x_k\in B_R(0)$ and $r_{x_k}$ are as in~$\mathscr{C}$, such that the following properties are satisfied: (i) the collection $\widetilde{\mathscr{C}}$ covers $B_R(0)$; and (ii) at most $T(n)$ among the balls $\{B^E_{r_{x_k}}(x_k)\}$ overlap, where $T(n)$ depends only on $n$.

We then have
\begin{align*}
\vert\vert \chi_{B_{R}^E(0)}U^{1/2}d v\vert\vert^p_{p} &\leq
\sum_{k=1}^{\infty}
\vert\vert \chi_{B_{\frac{r_{x_k}}{2}}^E(x_k)}U^{1/2}d v\vert\vert^p_{p} \\
&\leq \sum_{k=1}^{\infty}\bigg{[}
\vep \vert\vert\chi_{B_{r_{x_k}}^E(x_k)}\Delta v\vert\vert_{p} +
\vep K\vert\vert \chi_{B_{r_{x_k}}^E(x_k)}U^{1/2} d v\vert\vert_{p} +
\frac{K}{\vep}\vert\vert \chi_{B_{r_{x_k}}^E(x_k)}U v\vert\vert_{p}
\bigg{]}^p \\
&\leq \sum_{k=1}^{\infty}3^{p-1}\bigg{[}
\vep^p \vert\vert\chi_{B_{r_{x_k}}^E(x_k)}\Delta v\vert\vert_{p}^p +
\vep^p K^p\vert\vert \chi_{B_{r_{x_k}}^E(x_k)}U^{1/2}  dv\vert\vert_{p}^p +
\frac{K^p}{\vep^p}\vert\vert \chi_{B_{r_{x_k}}^E(x_k)}U v\vert\vert_{p}^p
\bigg{]} \\
&\leq 3^{p-1}T(n)\bigg{[}
\vep^p \vert\vert\chi_{B_{R}^E(0)}\Delta v\vert\vert_{p}^p +
\vep^p K^p\vert\vert \chi_{B_{R}^E(0)}U^{1/2} d v\vert\vert_{p}^p +
\frac{K^p}{\vep^p}\vert\vert \chi_{B_{R}^E(0)}U v\vert\vert_{p}^p
\bigg{]},
\end{align*}
where in the second estimate we used~(\ref{more_coer_est_2_1}) and in the third estimate we used (n2) from section \ref{SS:bin}.

This allows us to obtain
\begin{equation*}
[1 - 3^{p-1}T(n)K^p\vep^p]
\vert\vert \chi_{B_{R}^E(0)}U^{1/2}d v\vert\vert^p_{p} \leq
3^{p-1}T(n)\vep^p
\vert\vert\chi_{B_{R}^E(0)}\Delta v\vert\vert_{p}^p +
\frac{3^{p-1}T(n)K^p}{\vep^p}
\vert\vert \chi_{B_{R}^E(0)}U v\vert\vert_{p}^p.
\end{equation*}

Choosing $\vep_1$ small enough so that for all $0<\vep<\vep_1$ we have
\begin{equation*}
1 - 3^{p-1}T(n)K^p\vep^p>0\quad\textrm{ and }\quad \frac{3^{p-1}T(n)}{1 - 3^{p-1}T(n)K^p\vep^p}<2(3^{p-1}T(n)),
\end{equation*}
and we arrive at the estimate
\begin{equation*}
\vert\vert \chi_{B_{R}^E(0)}U^{1/2}d v\vert\vert^p_{p} \leq
2(3^{p-1}T(n))\vep^p
\vert\vert\chi_{B_{R}^E(0)}\Delta v\vert\vert_{p}^p +
\frac{3^{p-1}T(n)K^p}{\vep^p(1 - 3^{p-1}T(n)K^p\vep^p)}
\vert\vert \chi_{B_{R}^E(0)}U v\vert\vert_{p}^p.
\end{equation*}

Defining $C$ and $A_{\vep}$ so that
\begin{equation*}
C^p=2(3^{p-1}T(n)), \quad A_{\vep}^{p} = \frac{3^{p-1}T(n)K^p}{\vep^p(1 - 3^{p-1}T(n)K^p\vep^p)},
\end{equation*}
we obtain~(\ref{E:eucl-est-1}).
\end{proof}

We will globalize the estimate~(\ref{E:eucl-est-1}) with the help of lemma 1.6 from~\cite{hebey} (or lemma~4.15 from~\cite{GP-2015}), which we recall here for convenience:

\begin{lemma}\label{L:gp-15} Let $M$ be a geodesically complete $n$-dimensional Riemannian manifold such that $\ricc_{M}\geq c_1$, where $c_1$ is a constant. Then, for all $r > 0$ there exists a sequence
of points $\{x_i\} \subset M$ and a natural number
$N = N(n, r, c_1)$, such that
\begin{itemize}
\item $B_{r/4}(x_i) \cap B_{r/4}(x_j) = \emptyset$ for all $i$, $j \in \mathbb{N}$ with $i \neq j$, where $B_{r/4}(x_i)$ is as in~(\ref{E:ball-r}) with $\rho=r/4$;

\item $\cup_{i \in \mathbb{N}}B_{r/2}(x_i) = M$;

\item the intersection multiplicity of the system $\{B_{2r}(x_i) : i \in
\mathbb{N}\} \leq N$.
\end{itemize}
\end{lemma}
We will apply this lemma to the estimate~(\ref{E:eucl-est-1}) with the help of harmonic coordinates, which we recalled in appendix~A.

\begin{lemma}\label{more_coer_est_3} Let $M$ be a $n$-dimensional Riemannian manifold with $r_{\textrm{inj}}(M)>0$ and $\|\ricc_{M}\|_{\infty}<\infty$. Let $U \in C^1(M)$ such that $U \geq c_0 > 0$ and
$\vert d U\vert \leq \gamma U^{3/2}$  for some constants
$c_0$, $\gamma > 0$.
Then there exists a constant $\vep_2 > 0$
such that for all $0 < \vep < \vep_2$ we have
\begin{equation*}
\vert\vert U^{1/2}d v\vert\vert_{p} \leq
\vep\vert\vert \Delta v\vert\vert_{p} +
\widetilde{C}_{\vep}\vert\vert U v\vert\vert_{p},
\end{equation*}
for all $v \in C_c^{\infty}(M)$, where $\widetilde{C}_{\vep} > 0$ is a constant depending on $n$, $p$, $\vep$ and $N$ (with $N$ as in lemma~\ref{L:gp-15}). Here, the symbol $\|\cdot\|_{p}$ indicates the usual norm in $L^p(M)$ and $\Delta$ is the scalar Laplacian on $M$.
\end{lemma}
\begin{proof}
Thanks to the assumption $r_{\textrm{inj}}(M)>0$ and $\|\ricc_{M}\|_{\infty}<\infty$, we can use theorem B.4 in \cite{GP-2015} to infer that there exists $D= D(n, r_{\textrm{inj}}(M), \vert\vert \ricc_{M}\vert\vert_{\infty}) > 0$ such that
\begin{equation}\label{E:har-rad-lbound}
r_{2, 1, 1/2}(M) \geq D;
\end{equation}
see definition~\ref{D:def-a-2} and the equation~(\ref{notation-a-1}) in appendix A for the meaning of this notation.
Let $ 0 < r <r^*:= \frac{\min\{D, 1\}}{2}$. For any
$\widetilde{x} \in M$, there exists a $C^{1,1/2}$-harmonic coordinate chart
$\psi : B_{r^*}(\widetilde{x}) \rightarrow \R^n$ with accuracy 2, such that
\begin{align*}
\frac{1}{2}\delta_{ij} \leq g_{ij} &\leq 2\delta_{ij}; \\
\vert\vert g_{ij}\vert\vert_{\infty}, \vert\vert g^{ij}\vert\vert_{\infty}
&\leq C_1(n); \\
\vert\vert\partial_l g_{ij}\vert\vert_{\infty},
\vert\vert\partial_l g^{ij}\vert\vert_{\infty} &\leq C_2(D, n),
\end{align*}
for all $i,j,l$.

Using these properties we have
\begin{equation}\label{coordinate_chart}
\psi(B_{r/2}(\widetilde{x})) \subset B^E_{\frac{r}{\sqrt{2}}}(0) \subset
B^E_{\sqrt{2}r}(0) \subset \psi(B_{2r}(\widetilde{x})).
\end{equation}

Keeping in mind the inclusions (\ref{coordinate_chart}) and using lemma \ref{more_coer_est_2} with $R = r\sqrt{2}$, we infer that there exists $\vep_1$ such that for all $0<\widehat{\vep} <\vep_1$  and all $v\in\mcomp$ we have
\begin{equation}\label{more_coer_est_3_1}
\vert\vert \chi_{B_{r/2}(\widetilde{x})}U^{1/2}d v\vert\vert^p_{p} \leq
C^p\widehat{\vep}^p\vert\vert\chi_{B_{2r}(\widetilde{x})}\Delta v\vert\vert^p_{p} +
A_{\widehat{\vep}}^p\vert\vert \chi_{B_{2r}(\widetilde{x})}U v\vert\vert_{p}^p,
\end{equation}
where $C$ and $A_{\widehat{\vep}}$ are constants as in lemma~\ref{more_coer_est_2}.

Using lemma~\ref{L:gp-15} and~(\ref{more_coer_est_3_1}), for all $v\in\mcomp$ we have
\begin{align*}
\int_{M}\vert U^{1/2}d v\vert^p\,d\nu_g &\leq \sum_{i=1}^{\infty}
\int_{B_{r/2}(x_i)}\vert U^{1/2}d v\vert^p\,d\nu_g \\
&\leq \sum_{i=1}^{\infty}\bigg{[}
\widehat{\vep}^pC^p\int_{B_{2r}(x_i)}\vert\Delta v\vert^p\,d\nu_g +
A_{\widehat{\vep}}^p\int_{B_{2r}(x_i)}\vert Uv\vert^p\,d\nu_g\bigg{]}
\\
&\leq N\bigg{[}\widehat{\vep}^pC^p\vert\vert\Delta v\vert\vert^p_{p} +
A_{\widehat{\vep}}^p\vert\vert Uv\vert\vert^p_{p}\bigg{]},
\end{align*}
where $d\nu_g$ is the volume measure of $M$. (Here, for the first estimate we used the second property in lemma~\ref{L:gp-15}, for the second inequality we used~(\ref{more_coer_est_3_1}), and for the last inequality we used the third property in lemma~\ref{L:gp-15}.)

Using the inequality (n3) in section \ref{SS:bin}, the previous estimate leads to
\begin{align*}
\vert\vert U^{1/2}d v\vert\vert_{p} &\leq \bigg{(}
N\bigg{[}\widehat{\vep}^pC^p\vert\vert\Delta v\vert\vert^p_{p} +
A_{\widehat{\vep}}^p\vert\vert Uv\vert\vert^p_{p}\bigg{]}
\bigg{)}^{1/p} \\
&\leq N^{1/p}\widehat{\vep}C\vert\vert\Delta v\vert\vert_{p} +
N^{1/p}A_{\widehat{\vep}}\vert\vert Uv\vert\vert_{p}.
\end{align*}
If we set $\vep_2:=CN^{1/p}\vep_1$, where $\vep_1$ is as in lemma
\ref{more_coer_est_2}, we see that if $0<\vep<\vep_2$, then $\widehat{\vep}:= \frac{\vep}{CN^{1/p}}$ satisfies $0<\widehat{\vep}<\vep_1$. Plugging $\widehat{\vep}=\frac{\vep}{CN^{1/p}}$ into the previous estimate, we get
\begin{equation*}
\vert\vert U^{1/2}d v\vert\vert_{p} \leq
\vep\vert\vert \Delta v\vert\vert_{p} +
\widetilde{C}_{\vep}\vert\vert U v\vert\vert_{p},
\end{equation*}
for all $v\in\mcomp$, where $\widetilde{C}_{\vep}$ is a constant depending on $\vep$, $n$, $p$, and $N$.
\end{proof}

We are now ready to prove proposition \ref{more_coer_est_1}.

\begin{proof}[Proof of proposition \ref{more_coer_est_1}]

Let $\vep_{0}$ be as in lemma~\ref{L:lemma-u-hmax}, let $U_{\vep_0}$ be as in~(\ref{E:uep}) with $\vep=\vep_0$, and let $\beta_2$ be as in
assumption (A4) in section \ref{assump_V_X_etc}.  We define
\begin{equation}\label{E:def-Q}
Q(x) := \frac{1}{4}U_{\vep_0}(x) + V(x) + \beta_2 + 1,
\end{equation}
and referring to~(\ref{E:l-u-14}) we note that
\begin{equation}\label{E:d-phi-q}
\vert d\phi\vert \leq Q^{1/2}.
\end{equation}

Using the formula
\begin{equation}\label{E:cr-e-phi-u}
e^{-\phi/p}du=\frac{1}{p}(e^{-\phi/p}u)d\phi+d(e^{-\phi/p}u),
\end{equation}
which follows from (p1) and (c1), for all $u\in\mcomp$ we have
\begin{align}
\vert\vert Q^{1/2}\vert du\vert \vert\vert_{p, \mu} &=
\vert\vert Q^{1/2}\vert d u\vert e^{-\phi/p}\vert\vert_{p} \nonumber \\
&\leq \vert\vert Q^{1/2}\vert d(e^{-\phi/p}u)\vert\vert\vert_{p} +
\frac{1}{p}\vert\vert Q^{1/2}\vert d\phi\vert e^{-\phi/p}u\vert\vert_{p}
\nonumber \\
&\leq \vert\vert Q^{1/2}\vert d(e^{-\phi/p}u)\vert\vert\vert_{p} +
\frac{1}{p}\vert\vert Qu\vert\vert_{p, \mu}, \label{main_prop_eq_1}
\end{align}
where in the second inequality we used~(\ref{E:d-phi-q}).

From lemma \ref{L:lemma-aux-U} and the condition (A5) (see section \ref{assump_V_X_etc}), we get
\begin{align*}
\vert d Q\vert &\leq \frac{1}{4}\vert d U_{\vep_0}\vert +
\vert d V\vert \\
&\leq \frac{\vep_0}{4}U_{{\vep}_0}^{3/2} + \gamma V^{3/2} + \beta_3 \\
&\leq (\gamma + 2\vep_0)Q^{3/2} + \beta_3.
\end{align*}
The above estimate puts us in a position to use lemma \ref{more_coer_est_3} with
$Q =U$. Doing so we find that for every $\tau$ satisfying $0<\tau<\vep_2$, with $\vep_2$ as in lemma~\ref{more_coer_est_3}, there exists a constant $\widetilde{C}_{\tau}$ such that
\begin{equation}
\vert\vert Q^{1/2}\vert d(e^{-\phi/p}u)\vert\vert_{p} \leq
\tau\vert\vert\Delta(e^{-\phi/p}u)\vert\vert_p + \widetilde{C}_{\tau}
\vert\vert Qu\vert\vert_{p, \mu}, \label{main_prop_eq_2}
\end{equation}
for all $u\in\mcomp$.

Using the rules (p4) and (c2) and writing $2p^{-1}=(2-p+p)p^{-1}$, we obtain
\begin{align}
\vert\vert \Delta(e^{-\phi/p}u)\vert\vert_{p} &=
\vert\vert \Delta u +\frac{2}{p}\langle d\phi , d u\rangle -\frac{1}{p^2}\vert d\phi\vert^2u - \frac{1}{p}(\Delta\phi)u\vert\vert_{p, \mu}\nonumber\\
&\leq \vert\vert \Delta u + \langle d\phi , d u\rangle -
\langle X, d u\rangle + Vu\vert\vert_{p, \mu} +
\frac{\vert p -2\vert}{p}\vert\vert \langle d\phi , du\rangle
\vert\vert_{p, \mu} \nonumber \\
&\hspace{0.5cm}+
\vert\vert \langle X, d u\rangle\vert\vert_{p, \mu} +
\vert\vert Vu\vert\vert_{p, \mu} + \frac{1}{p^2}\vert\vert \vert d\phi\vert^2
u\vert\vert_{p, \mu} + \frac{1}{p}\vert\vert(\Delta\phi)u\vert\vert_{p, \mu}. \label{main_prop_eq_3}
\end{align}

Referring to~(\ref{E:d-phi-q}) and $(A4)$, we get
\begin{align}
\frac{\vert p - 2\vert}{p}\vert\vert\langle d\phi, du\rangle\vert\vert_{p,\mu} +
\vert\vert \langle X, d u\rangle\vert\vert_{p, \mu} &\leq
\vert\vert \vert d\phi\vert\vert d u\vert\vert\vert_{p,\mu} +
\kappa\vert\vert(\vert d\phi\vert^2 + V + \beta_2)^{1/2}\vert d u\vert
\vert\vert_{p,\mu} \nonumber \\
&\leq (1 + \kappa)\vert\vert Q^{1/2}\vert d u\vert\vert\vert_{p,\mu}. \label{main_prop_eq_4}
\end{align}

Furthermore,
\begin{align}\label{main_prop_eq_5}
&\vert\vert Vu\vert\vert_{p, \mu} + \frac{1}{p^2}\vert\vert \vert d\phi\vert^2
u\vert\vert_{p, \mu} + \frac{1}{p}\vert\vert(\Delta\phi)u\vert\vert_{p, \mu}\nonumber\\
&\leq
\bigg{(} 1 + \frac{1}{p^2}\bigg{)}\vert\vert(V + \vert d\phi\vert^2)u
\vert\vert_{p,\mu} + \frac{\sqrt{n}}{p}\vert\vert(\vep_0\vert d\phi\vert^2 +
C_{\vep_0})u\vert\vert_{p, \mu} \nonumber\\
&\leq \bigg{(}
1 + \frac{1}{p^2} + \frac{\sqrt{n}\vep_0}{p}
\bigg{)}\vert\vert Qu\vert\vert_{p,\mu},
\end{align}
where in the first inequality we used~(\ref{E:lap-hess}) and (A2), and in the second inequality we used~(\ref{E:l-u-14}), the definition~(\ref{E:def-Q}), and the property~(\ref{E:d-phi-q}).

Estimates (\ref{main_prop_eq_3}), (\ref{main_prop_eq_4}) and (\ref{main_prop_eq_5})
lead to
\begin{equation}\label{main_prop_eq_6}
\vert\vert \Delta(e^{-\phi/p}u)\vert\vert_{p} \leq
\vert\vert \hvmaxp u\vert\vert_{p, \mu} +
(1 + \kappa)\vert\vert Q^{1/2}\vert d u\vert\vert\vert_{p, \mu} +
\bigg{(}
1 + \frac{1}{p^2} + \frac{\sqrt{n}\vep_0}{p}
\bigg{)}\vert\vert Qu\vert\vert_{p,\mu}.
\end{equation}

As a consequence of (\ref{main_prop_eq_1}), (\ref{main_prop_eq_2}) and (\ref{main_prop_eq_6}), for every $\tau$ such that $0<\tau<\vep_2$,
we have
\begin{align}
\vert\vert Q^{1/2}\vert d u\vert\vert\vert_{p,\mu} &\leq
\tau\vert\vert\Delta(e^{-\phi/p}u)\vert\vert_{p} + \bigg{(}
\widetilde{C}_{\tau} + \frac{1}{p}\bigg{)}\vert\vert Qu\vert\vert_{p,\mu} \nonumber \\
&\leq \tau\vert\vert \hvmaxp u\vert\vert_{p,\mu} +
(1 + \kappa)\tau\vert\vert Q^{1/2}\vert d u\vert\vert\vert_{p, \mu} +
K_{\tau}\vert\vert Qu\vert\vert_{p,\mu}, \label{main_prop_eq_7}
\end{align}
for all $u\in\mcomp$,
where
\begin{equation*}
K_{\tau} = \widetilde{C}_{\tau} + \frac{1}{p} + \tau\bigg{(}
1 + \frac{1}{p^2} + \frac{\sqrt{n}\vep_0}{p}\bigg{)}.
\end{equation*}

Going back to~(\ref{E:def-Q}) and referring to~(\ref{E:uep}), lemma \ref{L:lemma-u-hmax}, and part (ii) of
theorem \ref{T:main-1} (with $V$ playing the role of $h$), we obtain
\begin{align}
\vert\vert Qu\vert\vert_{p,\mu} &\leq \frac{1}{4}
\vert\vert U_{\vep_0}u\vert\vert_{p,\mu} + \vert\vert Vu\vert\vert_{p,\mu}
+ (\beta_2 + 1)\vert\vert u\vert\vert_{p, \mu} \nonumber \\
&\leq \bigg{(} \frac{1}{4} + \frac{(\beta_2 + 1)\vep_0)}{4C_{\vep_0}}
\bigg{)}\vert\vert U_{\vep_0}u\vert\vert_{p,\mu} + \vert\vert Vu\vert\vert_{p,\mu} \nonumber \\
&\leq \widetilde{K}\vert\vert\lambda_1u + \hvmaxp u\vert\vert_{p,\mu}, \label{main_prop_eq_8}
\end{align}
where $\lambda_1 > 0$ is the constant determined in the statement of theorem
\ref{T:main-1} (ii) and
\begin{equation*}
\widetilde{K} := \bigg{(} 1 + \frac{(\beta_2 + 1)\vep_0)}{C_{\vep_0}}
\bigg{)}\frac{2p^2}{p-1} + (1 + p\kappa\gamma)\bigg{(}
1 - \frac{\theta}{p} - (p-1)\gamma\bigg{(}\frac{\kappa}{p} + \frac{\gamma}{4}
\bigg{)}
\bigg{)}^{-1}.
\end{equation*}
Using the property $Q \geq 1$ and writing $\hvmaxp u=(\hvmaxp +\lambda_1)u-\lambda_1u$, it follows from (\ref{main_prop_eq_7}) and
(\ref{main_prop_eq_8}) that
\begin{align}
(1- (1 + \kappa)\tau)\vert\vert Q^{1/2}\vert d u\vert\vert\vert_{p,\mu}
&\leq K_{\tau}\widetilde{K}\vert\vert\lambda_1u + \hvmaxp u\vert\vert_{p,\mu} + \tau\vert\vert \hvmaxp u\vert\vert_{p,\mu}\nonumber \\
&\leq
(\tau + K_{\tau}\widetilde{K})
\vert\vert\lambda_1u + \hvmaxp u\vert\vert_{p,\mu} +
\lambda_1\tau\vert\vert u\vert\vert_{p,\mu} \nonumber \\
&\leq
(\tau + K_{\tau}\widetilde{K})
\vert\vert\lambda_1u + \hvmaxp u\vert\vert_{p,\mu} +
\lambda_1\tau\vert\vert Qu\vert\vert_{p,\mu} \nonumber \\
&\leq
(\tau + K_{\tau}\widetilde{K} + \lambda_1\tau \widetilde{K})
\vert\vert\lambda_1u + \hvmaxp u\vert\vert_{p,\mu}. \label{main_prop_eq_9}
\end{align}
Choosing $\tau > 0$ sufficiently small so that $0<\tau<\vep_2$ and $(1+\kappa)\tau<1$, we
obtain the inequality (\ref{more_coer_est_1_1}), with
$\widetilde{C}' := \frac{(\tau + K_{\tau}\widetilde{K} + \lambda_1\tau \widetilde{K})}{(1- (1 + \kappa)\tau)}$.

We now move on to verify the inequality (\ref{more_coer_est_1_2}).

According to theorem 4.11 in~\cite{GP-2015}, the hypotheses $\|\ricc_{M}\|_{\infty}<\infty$ and $r_{\textrm{inj}}(M) > 0$, guarantee the fulfillment of the $L^p$-Calder\'on--Zygmund Inequality
$(CZ(p))$, $1<p<\infty$: there exists a constant
$A = A(n, p, \vert\vert \ricc_{M}\vert\vert_{\infty}, r_{\textrm{inj}}(M)) > 0$ such that for every
$u \in C_c^{\infty}(M)$
\begin{equation}\label{more_coer_est_CZ}
\vert\vert \hess(u)\vert\vert_p \leq A(\vert\vert\Delta u\vert\vert_p +
\vert\vert u\vert\vert_p),
\end{equation}
where $\|\cdot\|_{p}$ is the usual $L^p$-norm (with the Riemannian volume measure $d\nu_{g}$).

Using the product and chain rule for the Hessian (see (p5) in section \ref{SS:pr} and
(c3) in section \ref{SS:cr}), we can write for any $u \in C_c^{\infty}(M)$
\begin{align*}
e^{-\phi/p}\hess(u) &= \hess(e^{-\phi/p}u) +\frac{2}{p}e^{-\phi/p}d\phi\otimes du -
\hess(e^{-\phi/p})u \\
&= \hess(e^{-\phi/p}u) +\frac{2}{p}e^{-\phi/p}d\phi\otimes du
- \bigg{(}\frac{1}{p^2}e^{-\phi/p}d\phi\otimes d\phi\bigg{)}u + \bigg{(}
\frac{1}{p}e^{-\phi/p}\hess(\phi)
\bigg{)}u.
\end{align*}
Estimating the last equation term by term, we obtain
\begin{align*}
\vert\vert \hess(u)\vert\vert_{p, \mu} &= \vert\vert e^{-\phi/p}\hess(u)\vert\vert_p \\
&\leq \vert\vert \hess(e^{-\phi/p}u)\vert\vert_p + \frac{2}{p}
\vert\vert \vert d\phi\vert\vert  du\vert\vert\vert_{p, \mu} + \frac{1}{p^2}
\vert\vert \vert d\phi\vert^2u\vert\vert_{p,\mu} + \frac{1}{p}
\vert\vert \hess(\phi)u\vert\vert_{p, \mu} \\
&\leq A\vert\vert\Delta(e^{-\phi/p}u)\vert\vert_{p} +
A\vert\vert u\vert\vert_{p, \mu} +
\frac{2}{p}
\vert\vert \vert d\phi\vert\vert  du\vert\vert\vert_{p, \mu} + \frac{1}{p^2}
\vert\vert \vert d\phi\vert^2u\vert\vert_{p,\mu} + \frac{1}{p}
\vert\vert \hess(\phi)u\vert\vert_{p, \mu}\\
&\leq
A\vert\vert \hvmaxp u\vert\vert_{p,\mu} + A\vert\vert u\vert\vert_{p,\mu}
+ \bigg{(}A(1 + \kappa)+ \frac{2}{p}\bigg{)}
\vert\vert Q^{1/2}|d u|\vert\vert_{p,\mu} \\
&\hspace{0.5cm}
+ \bigg{(}
A\bigg{(}1 + \frac{1}{p^2} + \frac{\sqrt{n}\vep_0}{p}\bigg{)} + \frac{\vep_0}{p} +
\frac{1}{p^2}
\bigg{)}\vert\vert Qu\vert\vert_{p,\mu}.
\end{align*}
Here, to the get the second inequality we used~(\ref{more_coer_est_CZ}). In the third inequality, for the term containing $\Delta$ we used~(\ref{main_prop_eq_6}), for the terms containing $|d\phi|$ and $|d\phi|^2$ we used~(\ref{E:d-phi-q}), and for the term containing $\hess(\phi)$ we used the property
\begin{equation*}
p^{-1}\|\hess(\phi) u\|_{p,\mu}\leq p^{-1}\vep_0\|Q u\|_{p,\mu},
\end{equation*}
which is a consequence of~(A2),~(\ref{E:def-Q}) and~(\ref{E:uep}).

Finally, writing $\hvmaxp u = \hvmaxp u + \lambda_1u -\lambda_1u$ we have
\begin{equation*}
\vert\vert \hvmaxp u\vert\vert_{p,\mu} \leq
\vert\vert \hvmaxp u + \lambda_1u\vert\vert_{p,\mu}
+ \lambda_1\vert\vert u\vert\vert_{p,\mu}.
\end{equation*}
Furthermore, using
$\vert\vert u\vert\vert_{p,\mu} \leq \vert\vert Qu\vert\vert_{p, \mu}$ (recall that $Q\geq 1$), it follows
from \eqref{main_prop_eq_8}, \eqref{main_prop_eq_9} that
\begin{equation*}
\vert\vert \hess(u)\vert\vert_{p,\mu} \leq \widetilde{C}''
\vert\vert \lambda_1u + \hvmaxp u\vert\vert_{p,\mu},
\end{equation*}
for all $u\in\mcomp$, where
\begin{equation*}
\widetilde{C}'' = A + A(1 + \lambda_1)\widetilde{K} + \bigg{(}
A(1+ \kappa) + \frac{2}{p}\bigg{)}\widetilde{C}' + \bigg{(}
A\bigg{(}1 + \frac{1}{p^2} + \frac{\sqrt{n}\vep_0}{p}\bigg{)} + \frac{\vep_0}{p} +
\frac{1}{p^2}
\bigg{)}\widetilde{K}.
\end{equation*}
This proves inequality \eqref{more_coer_est_1_2} of the proposition.

\end{proof}

\section{Proof of theorem ~\ref{T:main-2}}\label{S:pf-thm-2}

\subsection{Preliminary lemmas}
In this subsection $(M, g)$ is a Riemannian manifold without boundary (with additional geometric conditions in some statements).
We start with the definition of weighted Sobolev space (of differential order 1) on one-forms:
\begin{equation}\label{sob-one-form}
W^{1,p}_{\mu}(\Lambda^1T^*M):=\left\{u\in L^p_{\mu}\colon \nabla^{lc}u\in  L^p_{\mu}\right\},
\end{equation}
where $\nabla^{lc}$ is as in~(\ref{E:cov-lc}) and $\mu$ is as in~(\ref{E:meas-mu}).

We now list a few preliminary lemmas which will be important in the proof of
part (ii) of theorem \ref{T:main-2}. We begin with a density lemma.

\begin{lemma}\label{main-2_density} Assume that $(M,g)$ admits a sequence of cut-off functions satisfying the properties (i), (ii), (iii) in part (H3) of remark~\ref{R:HYP-H}. Then, the following hold:
\begin{itemize}
  \item [(i)] The space of smooth compactly supported one-forms $\Omega^1_{c}(M)$ is dense in $W^{1,p}_{\mu}(\Lambda^1T^*M)$, $1<p<\infty$.
  \item [(ii)] If, in addition, the property (iv) in part (H3) of remark~\ref{R:HYP-H} is satisfied (that is, if $M$ admits a sequence of weak Hessian cut-off functions), then $C_c^{\infty}(M)$ is dense in $W^{2,p}_{\mu}(M)$, $1<p<\infty$, where $W^{2,p}_{\mu}(M)$ is as in~(\ref{E:sob-mu}).
\end{itemize}
\end{lemma}
\begin{proof} We first prove part (ii) of the lemma. By theorem 1 in~\cite{GGP-2015} we have that $C^{\infty}(M) \cap W^{2,p}_{\mu}(M)$ is
dense in $W^{2,p}_{\mu}(M)$. (The theorem is applicable because the inclusion $W^{2,p}_{\loc}\subset W^{1,p}_{\loc}$ tells us that regularity $W^{1,p}_{\loc}$, as required by the quoted theorem, is built into the definition~(\ref{E:sob-mu}) of $W^{2,p}_{\mu}(M)$.) Therefore, it suffices to prove that given a function
$f \in C^{\infty}(M) \cap W^{2,p}_{\mu}(M)$, we can approximate it by a sequence
$(f_k) \in C_c^{\infty}(M)$ in $W^{2,p}_{\mu}(M)$-norm.

Denote by $(\psi_k)$ a sequence of weak Hessian cut-off functions on $M$ and define $f_k := \psi_k f$. The desired approximation property follows from the assumption $f \in W^{2,p}_{\mu}(M)$, the formulas
\begin{equation*}
df_k=\psi_k df+fd\psi_k,\qquad \hess(f_k) = f\hess(\psi_k) + 2d\psi_k\otimes df + \psi_k\hess(f),
\end{equation*}
and the following observations: (i) the elements $(1-\psi_k)$, $d\psi_k$ and $\hess(\psi_k)$ are uniformly bounded and supported in $\supp(1-\psi_k)$; (ii) in view of property (ii) in (H3), for any compact set $S\subset M$, we have that $\supp(1-\psi_k)\subset (M\backslash S)$ for a sufficiently large $k$. For more details about this argument, see remark 2.5 in~\cite{IRV-19}.

For part (i) we start from $w \in\Omega^1(M) \cap W^{1,p}_{\mu}(\Lambda^1T^*M)$. Keeping in mind the density of the latter intersection in $W^{1,p}_{\mu}(\Lambda^1T^*M)$ (again referring to theorem 1 in~\cite{GGP-2015}), we carry out the proof by defining $w_k := \psi_k w$, where $\psi_k$ satisfies the properties (i)--(iii) in (H3), and using the formula
\begin{equation*}
\nabla^{lc}\omega_k=d\psi_k\otimes w+\psi_k\nabla^{lc}w,
\end{equation*}
together with the same argument as in part (ii) of this lemma.
\end{proof}

\begin{remark}\label{R:weak-hes-density} As mentioned in remark~\ref{R:HYP-H}, the assumptions
$r_{\textrm{inj}}(M)>0$ and $\|\ricc_{M}\|_{\infty}<\infty$ ensure that $M$ admits a sequence of genuine Hessian cut-off functions (and, therefore, weak Hessian cut-off functions).
\end{remark}

The next two lemmas are the analogues of lemmas 7.1 and 7.2 of~\cite{MPR-05}.

\begin{lemma}\label{main-2_bound_1} Assume that $(M,g)$ admits a sequence of cut-off functions satisfying the properties (i), (ii), (iii) in part (H3) of remark~\ref{R:HYP-H}. Let $\phi$ be as in~(\ref{E:meas-mu}). Then, the following hold:
\begin{itemize}
  \item [(i)] The map $u \mapsto \vert d\phi \vert u$ is bounded from $W^{1,p}_{\mu}(\Lambda^1T^*M)$ to $L^p_{\mu}(\Lambda^1T^*M)$, $1<p<\infty$.
  \item [(ii)] The map $v \mapsto \vert d\phi \vert v$ is bounded from $W^{1,p}_{\mu}(M)$ to $L^p_{\mu}(M)$, $1<p<\infty$.
\end{itemize}
\end{lemma}

\begin{proof} We will prove part (i) of this lemma, with the remark that part (ii) is justified in the same way if we replace $\nabla^{lc}u$ by $dv$.

Using part (i) of lemma \ref{main-2_density} it suffices to prove
\begin{equation*}
\vert\vert |d\phi| u\vert\vert_{p, \mu} \leq C\bigg{(}\vert\vert \nabla^{lc} u\vert\vert_{p, \mu}
+ \vert\vert u\vert\vert_{p, \mu}\bigg{)},
\end{equation*}
for all $u \in \Omega^1_{c}(M)$, where $C > 0$ is a constant (independent of $u$).

Observe that there exist constants $a$, $b > 0$ such that
\begin{equation}\label{E:rem-a-b-1}
\vert d\phi\vert^p \leq a(1+ \vert d\phi\vert^2)^{p/2 - 1}\vert d\phi\vert^2 + b.
\end{equation}
Thus, remembering~(\ref{E:meas-mu}) and $d(e^{-\phi})=-e^{-\phi}d\phi$, it suffices for us to estimate
\begin{equation*}
\int_M(1+ \vert d\phi\vert^2)^{p/2 - 1}\vert d\phi\vert^2\vert u\vert^pe^{-\phi}\,d\nu_{g}=-\int_M\vert u\vert^p(1+ \vert d\phi\vert^2)^{p/2 - 1}
\langle d\phi, d(e^{-\phi})\rangle\,d\nu_{g}.
\end{equation*}

As $u$ has compact support, we use integration by parts to get
\begin{align*}
&-\int_M\vert u\vert^p(1+ \vert d\phi\vert^2)^{p/2 - 1}
\langle d\phi, d(e^{-\phi})\rangle
d\nu_{g} =
-\int_M\vert u\vert^p(1+ \vert d\phi\vert^2)^{p/2 - 1}\Delta\phi\,d\mu \\
&+
(p-2)\int_M\vert u\vert^p(1+ \vert d\phi\vert^2)^{p/2 - 2}
\langle (\hess\phi)(\bullet, (d\phi)^{\sharp}), d\phi\rangle\,d\mu \\
&+
p\int_M\vert u\vert^{p-2}(1+ \vert d\phi\vert^2)^{p/2 - 1}
\langle d\phi, \omega_{u}\rangle\,d\mu,
\end{align*}
where $\omega_{u}$ is as in~(\ref{E:omega-x}) with $\vbn =\Lambda^1T^*M$. Here, the term containing $\hess$ was handled as in~(\ref{E:similar-calc-8}) and the term $d|u|^p$ was handled using
lemma~\ref{L:om-sig} (i).

Rewriting and using the inequality $|\langle d\phi, \omega_{u}\rangle|\leq |d\phi||\nabla^{lc} u||u|$, we get
\begin{align*}
&-\int_M\vert u\vert^p(1+ \vert d\phi\vert^2)^{p/2 - 1}
\langle d\phi, d(e^{-\phi})\rangle
d\nu_{g} \\
&=\int_M\bigg{(}
\vert u\vert^p(1+ \vert d\phi\vert^2)^{p/2 - 1}\bigg{)}\bigg{(}
-\Delta\phi+ (p-2)(1+ \vert d\phi\vert^2)^{-1}
\langle (\hess\phi)(\bullet, (d\phi)^{\sharp}), d\phi\rangle \bigg{)}d\mu \\
&+
p\int_M \vert u\vert^{p-2}(1+ \vert d\phi\vert^2)^{p/2 - 1}
\langle d\phi, \omega_{u}\rangle d\mu \\
&\leq
\int_M\bigg{(}
\vert u\vert^p(1+ \vert d\phi\vert^2)^{p/2 - 1}\bigg{)}\bigg{(}
-\Delta\phi + (p-2)(1+ \vert d\phi\vert^2)^{-1}
\langle (\hess\phi)(\bullet, (d\phi)^{\sharp}), d\phi\rangle \bigg{)}d\mu \\
&+
p\int_M\vert u\vert^{p-1}(1+ \vert d\phi\vert^2)^{p/2 - 1}
\vert d\phi\vert\vert\nabla^{lc} u\vert d\mu.
\end{align*}

We also have
\begin{align*}
&-\Delta\phi
+ (p-2)(1+ \vert d\phi\vert^2)^{-1}\langle (\hess\phi)(\bullet, (d\phi)^{\sharp}), d\phi\rangle
\\
&\leq
|\Delta\phi|
+ |p-2|(1+ \vert d\phi\vert^2)^{-1}
\vert\langle (\hess\phi)(\bullet, (d\phi)^{\sharp}), d\phi\rangle\vert \\
&\leq
|\Delta\phi|
+ |p-2|(1+ \vert d\phi\vert^2)^{-1}
\vert \hess(\phi)\vert\vert d\phi\vert^2
\leq
\sqrt{n}\vert \hess(\phi)\vert + |p-2|\vert \hess(\phi)\vert \\
&\leq
(\sqrt{n} + |p-2|)\vep\vert d\phi\vert^2 + (\sqrt{n} + |p-2|)C_{\vep},
 \end{align*}
for all $\vep>0$, where in the last inequality we used the hypothesis (A2) in section~\ref{assump_V_X_etc}.

Combining the previous two estimates we get
\begin{align*}
&\int_M(1+ \vert d\phi\vert^2)^{p/2 - 1}\vert d\phi\vert^2\vert u\vert^p\,d\mu
\leq
(\sqrt{n} + |p-2|))\vep\int_M\bigg{(}
\vert u\vert^p(1+ \vert d\phi\vert^2)^{p/2 - 1}\bigg{)}
\vert d\phi\vert^2\,d\mu \\
&+
(\sqrt{n} + |p-2|)C_{\vep}\int_M\bigg{(}
\vert u\vert^p(1+ \vert d\phi\vert^2)^{p/2 - 1}\bigg{)}\,d\mu \\
&+
p\int_M\vert u\vert^{p-1}(1+ \vert d\phi\vert^2)^{p/2 - 1}
\vert d\phi\vert\vert\nabla^{lc}u\vert\, d\mu,
\end{align*}
which implies
\begin{align*}
&\bigg{(}
1 - (\sqrt{n} + |p-2|)\vep
\bigg{)}
\int_M\bigg{(}
\vert u\vert^p(1+ \vert d\phi\vert^2)^{p/2 - 1}\bigg{)}
\vert d\phi\vert^2\,d\mu \\
&\leq (\sqrt{n} + |p-2|)C_{\vep}\int_M\bigg{(}
\vert u\vert^p(1+ \vert d\phi\vert^2)^{p/2 - 1}\bigg{)}d\mu + p\int_M\vert u\vert^{p-1}(1+ \vert d\phi\vert^2)^{p/2 - 1}
\vert d\phi\vert\vert\nabla^{lc}u\vert\, d\mu.
\end{align*}

We then use the inequality
$(1 + t^2)^{p/2 - 1} \leq
\frac{1 - (\sqrt{n} + |p-2|)\vep}{2(\sqrt{n} + |p-2|)C_{\vep}}
(1 + t^2)^{p/2 - 1}t^2 + c$ for some $c > 0$ and H\"older's inequality to obtain
\begin{align*}
&\bigg{(}
\frac{1 - (\sqrt{n} + |p-2|)\vep}{2}
\bigg{)}
\int_M\bigg{(}
\vert u\vert^p(1+ \vert d\phi\vert^2)^{p/2 - 1}\bigg{)}
\vert d\phi\vert^2\,d\mu \\
&\leq p\bigg{(}\int_M\vert u\vert^{p}(1+ \vert d\phi\vert^2)^{p'(p/2 - 1)}
\vert d\phi\vert^{p'}d\mu\bigg{)}^{\frac{1}{p'}}
\bigg{(}\int_M \vert\nabla^{lc} u\vert^pd\mu\bigg{)}^{\frac{1}{p}} +
C\int_M\vert u\vert^p d\mu,
\end{align*}
where $p'$ is the conjugate exponent to $p$ and $C > 0$ is a constant.

In order to estimate the right hand side of the above inequality, we proceed as follows. First note that there exist constants $c_1$, $c_2 > 0$ such that
$(1 + t^2)^{p'(\frac{p}{2}-1)}t^{p'} \leq c_1(1 + t^2)^{p/2-1}t^2 + c_2$.
Using this inequality, we obtain
\begin{align*}
&\bigg{(}
\frac{1 - (\sqrt{n} + |p-2|)\vep}{2}
\bigg{)}
\int_M\bigg{(}
\vert u\vert^p(1+ \vert d\phi\vert^2)^{p/2 - 1}\bigg{)}
\vert d\phi\vert^2d\mu \\
&\leq
p\bigg{(}\int_M\bigg{(}c_1\vert u\vert^{p}(1+ \vert d\phi\vert^2)^{p/2 - 1}
\vert d\phi\vert^{2} + c_2\vert u\vert^p\bigg{)}d\mu\bigg{)}^{\frac{1}{p'}}
\bigg{(}\int_M \vert\nabla^{lc}u\vert^pd\mu\bigg{)}^{\frac{1}{p}} +C\int_M\vert u\vert^pd\mu.
\end{align*}
Using a weighted Young's inequality (see (n1) in section~\ref{SS:bin}), we have
\begin{align*}
&\bigg{(}
\frac{1 - (\sqrt{n} + |p-2|)\vep}{2}
\bigg{)}
\int_M\bigg{(}
\vert u\vert^p(1+ \vert d\phi\vert^2)^{p/2 - 1}\bigg{)}
\vert d\phi\vert^2d\mu \\
&\leq \frac{c_1\vep' p}{p'}
\int_M\bigg{(}\vert u\vert^{p}(1+ \vert d\phi\vert^2)^{p/2 - 1}
\vert d\phi\vert^{2}\bigg{)}d\mu +
\frac{c_2\vep' p}{p'}\int_M\vert u\vert^p\,d\mu +
\frac{1}{\vep'}\int_M\vert\nabla^{lc}u\vert^p d\mu
+
C\int_M\vert u\vert^pd\mu.
\end{align*}
Choosing $\vep$, $\vep' > 0$ sufficiently small so that $\bigg{(}
\frac{1 - (\sqrt{n} + |p-2|)\vep}{2} - \frac{c_1\vep' p}{p'}\bigg{)} > 0$
we get
\begin{align*}
&\bigg{(}
\frac{1 - (\sqrt{n} + |p-2|)\vep}{2} - \frac{c_1\vep' p}{p'}\bigg{)}
\int_M\bigg{(}
\vert u\vert^p(1+ \vert d\phi\vert^2)^{p/2 - 1}\bigg{)}
\vert d\phi\vert^2d\mu \\
&\leq \bigg{(} \frac{c_2\vep' p}{p'} + C \bigg{)}\int_M\vert u\vert^pd\mu +
\frac{1}{\vep'}\int_M\vert\nabla^{lc}u\vert^pd\mu.
\end{align*}

Remembering~(\ref{E:rem-a-b-1}), the last estimate proves that the map $u \mapsto \vert d\phi \vert u$ is a bounded operator from
$W^{1,p}_{\mu}(\Lambda^1T^*M)$ to $L^p_{\mu}(\Lambda^1T^*M)$.

\end{proof}

\begin{lemma}\label{main-2_bound_2} Assume that $(M,g)$ admits a sequence of weak Hessian cut-off functions.
Let $\phi$ be as in~(\ref{E:meas-mu}). Then the map $w \mapsto \vert d\phi\vert^2w$ is bounded
from $W^{2,p}_{\mu}(M)$ to $L^p_{\mu}(M)$.
\end{lemma}
\begin{proof}
Given $w \in \mcomp$, let us consider the one-form $v:=wd\phi \in W^{1,p}_{\mu}(\Lambda^1T^*M)$.
Keeping in mind that $|w|d\phi|^2|=|d\phi||v|$, we have
\begin{align}\label{E:sob-2-est-space}
&\vert\vert w\vert d\phi\vert^2\vert\vert_{L^p_{\mu}(M)}=\vert\vert |d\phi|v\vert\vert_{L^p_{\mu}(\Lambda^1T^*M)}\nonumber\\
&\leq \widetilde {C} \vert\vert v\vert\vert_{W^{1,p}_{\mu}(\Lambda^1T^*M)}\leq \widetilde {C} (\|v\|_{L^p_{\mu}(\Lambda^1T^*M)}+\|\nabla^{lc}v\|_{L^p_{\mu}(T^*M\otimes T^*M)})\nonumber\\
&=\widetilde {C}(\|w|d\phi|\|_{L^p_{\mu}(M)}+\|dw\otimes d\phi+w\hess \phi\|_{L^p_{\mu}(T^*M\otimes T^*M)})\nonumber\\
&\leq \widetilde {C}(\|w|d\phi|\|_{{L^p_{\mu}(M)}}+\|dw\otimes d\phi\|_{L^p_{\mu}(T^*M\otimes T^*M)}+
\|w\hess\phi\|_{L^p_{\mu}(T^*M\otimes T^*M)}),
\end{align}
where the first inequality follows by part (i) of lemma \ref{main-2_bound_1}, the second inequality follows by the definition of the Sobolev norm for $W^{1,p}_{\mu}(\Lambda^1T^*M)$, and the second equality follows from the formulas $|v|=|w|d\phi||$ and $\nabla^{lc}(wd\phi)=dw\otimes d\phi+w\hess \phi$. (Here, $\widetilde {C}$ is a constant possibly changing its value as we go from one estimate to another.)

Using part (ii) of lemma~\ref{main-2_bound_1}, we have
\begin{equation*}
\|w|d\phi|\|_{{L^p_{\mu}(M)}}\leq  C\|w\|_{W^{1,p}_{\mu}(M)}\leq  C\|w\|_{W^{2,p}_{\mu}(M)},
\end{equation*}
where $C$ is a constant.

Remembering that $|dw\otimes d\phi|\leq |dw||d\phi|=|dw|d\phi||$, we have
\begin{align*}
&\|dw\otimes d\phi\|_{L^p_{\mu}(T^*M\otimes T^*M)}\leq \|dw|d\phi|\|_{L^p_{\mu}(\Lambda^1T^*M)}
\leq C\|dw\|_{W^{1,p}_{\mu}(\Lambda^1T^*M)}\nonumber\\
&\leq C (\|dw\|_{L^p_{\mu}(\Lambda^1T^*M)}+\|\hess w\|_{L^p_{\mu}(T^*M\otimes T^*M)})\leq C\|w\|_{W^{2,p}_{\mu}(M)},
\end{align*}
where in the second inequality we used part (i) of lemma~\ref{main-2_bound_1} with $u=dw$ and in the third inequality we used the definition of the Sobolev norm for $W^{1,p}_{\mu}(\Lambda^1T^*M)$. (Here, $C$ is a constant possibly changing its value as we go from one estimate to another.)

Lastly, we use the condition (A2) in section~\ref{assump_V_X_etc} to get
\begin{align*}
&\|w\hess\phi\|_{L^p_{\mu}(T^*M\otimes T^*M)}\leq \vep\vert\vert w\vert d\phi\vert^2\vert\vert_{L^p_{\mu}(M)} +
C_{\vep}\vert\vert w\vert\vert_{L^p_{\mu}(M)}\\
&\leq \vep\vert\vert w\vert d\phi\vert^2\vert\vert_{L^p_{\mu}(M)} +
C_{\vep}\vert\vert w\vert\vert_{W^{2,p}_{\mu}(M)},
\end{align*}
for all $\vep>0$.

Combining the last three estimates with~(\ref{E:sob-2-est-space}) and making $\vep>0$ sufficiently small, we obtain (after rearranging)

\begin{equation*}
\|w|d\phi|^2\|_{{L^p_{\mu}(M)}}\leq  \widetilde{C}\|w\|_{W^{2,p}_{\mu}(M)},
\end{equation*}
for all $w\in\mcomp$, where $\widetilde{C}$ is a constant.

The boundedness of the map $w \mapsto \vert d\phi\vert^2w$ (from $W^{2,p}_{\mu}(M)$ to $L^p_{\mu}(M)$) follows since $\mcomp$ is dense in $W^{2,p}_{\mu}(M)$, as indicated in part (ii) of lemma~\ref{main-2_density}.
\end{proof}

\subsection{Proof of part (i) of Theorem \ref{T:main-2}}

By part (i) of theorem~\ref{T:main-1} (specialized to  the expression $P^d$ in~(\ref{E:def-H-d})) it suffices to prove that $\hvmaxp = H^{d}$.  Looking at the definitions of $\hvmaxp$ and $H^{d}$ in section~\ref{SS:min-max-rel-1} (specialized to the expression $P^d$ in~(\ref{E:def-H-d})) and section~\ref{SS:operator-H} respectively, we note that we can prove the equality of $\hvmaxp$ and $H^d$ by showing that $\dom(\hvmaxp) = \dom(H^d)$.

We start by observing that the definitions of $\hvmaxp$ and $H^d$ grant the inclusion $\dom(H^d) \subset \dom(\hvmaxp)$.
To see that $\dom(\hvmaxp) \subset \dom(H^d)$, start by picking an arbitrary
$u \in \dom(\hvmaxp)$. By part (i) of theorem~\ref{T:main-1}, we know that
$C_c^{\infty}(M)$ is a core for $\hvmaxp$. Thus, there exists a sequence
$(u_k) \in C_c^{\infty}(M)$ such that
\begin{equation}\label{E:conv-rel-8}
u_k \rightarrow u \text{ in } L^p_{\mu}(M) \text{ and }
P^du_k \rightarrow \hvmaxp u \text{ in } L^p_{\mu}(M).
\end{equation}
Using these convergence relations and the coercive estimate~(\ref{more_coer_est_1_1}), we infer that
$(d\phi)^{\#}u \in L^p_{\mu}(M)$ and $Xu \in L^p_{\mu}(M)$, where for the second inclusion we also used the condition (A4) from section~\ref{assump_V_X_etc} (with $V$ in the place of $h$). Similarly, using the estimate~(\ref{E:coercive-1}) (with $V$ in the place of $h$) and~(\ref{E:conv-rel-8}), we obtain $Vu \in L^{p}_{\mu}(M)$. As $P^du\in L^p_{\mu}(M)$ by the definition of $\hvmaxp$, the three inclusions from the last two sentences lead to $\Delta u\in  L^p_{\mu}(M)$. Therefore, $u\in\dom (H^d)$, and this proves part
(i) of theorem \ref{T:main-2}. $\hfill\square$

\subsection{Proof of part (ii) of Theorem \ref{T:main-2}}

To prove part (ii) of theorem \ref{T:main-2} we need to show that
\begin{equation*}
\dom(H^d) = \mathcal{D} := \{u \in W^{2,p}_{\mu}(M) : Vu \in L^p_{\mu}(M)\},
\end{equation*}
where  $W^{2,p}_{\mu}(M)$ is as in~(\ref{E:sob-mu}).

Let $u$ be an arbitrary element of $\dom(H^d)=\dom(\hvmaxp)$, where the equality comes from the proof of part (i) of this theorem. Using the estimate~(\ref{more_coer_est_1_1}) and~(\ref{E:conv-rel-8}) we obtain $du\in L^p_{\mu}$ and $Vu\in L^p_{\mu}$. Similarly, from~(\ref{more_coer_est_1_2}) and~(\ref{E:conv-rel-8}) we get $\hess u\in L^p_{\mu}$. Therefore, $\dom(H^d) \subset \mathcal{D}$.

Note that if $u\in \mathcal{D}$, then, in particular, $\hess u\in L^p_{\mu}$. Thus, looking at~(\ref{E:lap-hess}) and remembering that $\mcomp$ is dense in $W_{\mu}^{2,p}(M)$ (see lemma~\ref{main-2_density}(ii)), it follows that $\Delta u\in L^p_{\mu}$. Hence, to prove the inclusion $\mathcal{D} \subset \dom(H^d)$, it is enough to show that if $u \in \mathcal{D}$, then $(d\phi)^{\#}u\in L^p_{\mu}(M)$ and $X(u) \in L^p_{\mu}(M)$.

By part (ii) of lemma~\ref{main-2_density} and since $V\in C^1(M)$, it follows that $C_c^{\infty}(M)$ is dense in $\mathcal{D}$ under the graph norm, $\vert\vert u\vert\vert_{\mathcal{D}} :=
\vert\vert u\vert\vert_{W_{\mu}^{2,p}(M)} + \vert\vert Vu\vert\vert_{p,\mu}$. Thus, it suffices to show that for every $u \in C_c^{\infty}(M)$ we have
\begin{equation}\label{main_thm_2_bound}
\vert\vert (d\phi)^{\#}(u)\vert\vert_{p,\mu} + \vert\vert X(u)\vert\vert_{p,\mu}
\leq C\vert\vert u\vert\vert_{\mathcal{D}},
\end{equation}
for some constant $C > 0$.

We begin our proof of~(\ref{main_thm_2_bound}) by recalling the definition of $Q$ in~(\ref{E:def-Q}), with $\vep_{0}$ as in lemma~\ref{L:lemma-u-hmax}.

Referring to~(\ref{main_prop_eq_2}), for all $0<\tau<\vep_2$ (with $\vep_2$ as in lemma~\ref{more_coer_est_3}) there exists a constant $\widetilde{C}_{\tau}$ such that
\begin{align*}
\vert\vert Q^{1/2}\vert d(e^{-\phi/p}u)\vert\vert_{p} &\leq
\tau\vert\vert\Delta(e^{-\phi/p}u)\vert\vert_p + \widetilde{C}_{\tau}
\vert\vert Qu\vert\vert_{p, \mu} \\
&= \tau\bigg{\vert}\bigg{\vert}
\Delta u + \frac{2}{p}\langle d\phi, du\rangle - \frac{1}{p^2}\vert d\phi\vert^2u
-\frac{1}{p}u\Delta\phi\bigg{\vert}\bigg{\vert}_{p,\mu} +
\widetilde{C}_{\tau}
\vert\vert Qu\vert\vert_{p, \mu},
\end{align*}
for all $u\in\mcomp$, where in the equality we used the rules (p4) and (c2).

Estimating the norm $\|\cdot\|_{p,\mu}$ containing four terms, we get
\begin{equation*}
\vert\vert Q^{1/2}\vert d(e^{-\phi/p}u)\vert\vert_{p}
\leq
\tau\bigg{(}
\sqrt{n}\vert\vert u\vert\vert_{W_{\mu}^{2,p}(M)} +
\frac{2}{p}\vert\vert Q^{1/2}\vert du\vert\vert\vert_{p,\mu} +
\bigg{(}\frac{1}{p^2} + \frac{\sqrt{n}\vep_0}{p} \bigg{)}\vert\vert Qu\vert\vert_{p,\mu}
\bigg{)} +
\widetilde{C}_{\tau}\vert\vert Qu\vert\vert_{p, \mu},
\end{equation*}
for all $u\in\mcomp$, where we used~(\ref{E:lap-hess}) for the first term, we referred to~(\ref{E:d-phi-q}) for the second term, and we estimated the third and the fourth term in the same way as in~(\ref{main_prop_eq_5}).

Combining the preceding estimate with~(\ref{main_prop_eq_1}) we obtain
\begin{align*}
\vert\vert Q^{1/2}\vert du\vert\vert\vert_{p, \mu} &\leq
\vert\vert Q^{1/2}\vert d(e^{-\phi/p}u)\vert\vert\vert_{p} +
\frac{1}{p}\vert\vert Qu\vert\vert_{p,\mu} \\
&\leq
\tau\sqrt{n}\vert\vert u\vert\vert_{W_{\mu}^{2,p}(M)} +
\frac{2\tau}{p}\vert\vert Q^{1/2}\vert du\vert\vert\vert_{p,\mu} +
\bigg{(}
\frac{\tau}{p^2} + \frac{\sqrt{n}\vep_0\tau}{p} + \widetilde{C}_{\tau}
\bigg{)}\vert\vert Qu\vert\vert_{p,\mu}.
\end{align*}
Choosing $\tau$ small enough so that $0<\tau<\vep_2$ and $1-\frac{2\tau}{p}>0$ and rearranging the previous estimate, we get
\begin{equation}\label{E:almost-done}
\vert\vert Q^{1/2}\vert du\vert\vert\vert_{p, \mu} \leq
J_{\tau}\vert\vert u\vert\vert_{W_{\mu}^{2,p}(M)}  + G_{\tau}\vert\vert Qu\vert\vert_{p,\mu},
\end{equation}
for all $u\in\mcomp$, where
\begin{equation}\label{E:j-g-tau}
J_{\tau}:=\left(1-\frac{2\tau}{p}\right)^{-1}\tau\sqrt{n},\qquad G_{\tau}:=\left(\frac{\tau}{p^2} + \frac{\sqrt{n}\vep_0\tau}{p} + \widetilde{C}_{\tau}\right)\left(1-\frac{2\tau}{p}\right)^{-1}.
\end{equation}

Referring to the condition (A4) in
section \ref{assump_V_X_etc}, we see that $\vert X\vert \leq \kappa Q^{1/2}$. Putting $\kappa_1:=\max\{1,\kappa\}$ and using $|(d\phi)^{\#}u|\leq |d\phi||du|$, the property~(\ref{E:d-phi-q}), and the estimate~(\ref{E:almost-done}) we get
\begin{align}\label{E:almost-done-1}
&\vert\vert (d\phi)^{\#}u\vert\vert_{p,\mu} +
\frac{1}{\kappa_1}\vert\vert X(u)\vert\vert_{p,\mu} \leq
2\vert\vert Q^{1/2}\vert du\vert\vert\vert_{p,\mu} \leq
2J_{\tau}\vert\vert u\vert\vert_{W_{\mu}^{2,p}(M)} + 2 G_{\tau}\vert\vert Qu\vert\vert_{p,\mu} \nonumber\\
&\leq 2J_{\tau}\vert\vert u\vert\vert_{W_{\mu}^{2,p}(M)} +  2 G_{\tau}(\vert\vert \vert d\phi\vert^2u\vert\vert_{p,\mu} +
\vert\vert Vu\vert\vert_{p,\mu}) \nonumber\\
&+2 G_{\tau}\bigg{(}\frac{C_{\vep_0}}{\vep_0} + \beta_2 + 1\bigg{)}\vert\vert u\vert\vert_{p,\mu},
\end{align}
for all $u\in\mcomp$, where in the last inequality we used the definition~(\ref{E:def-Q}).

The term $\vert\vert \vert d\phi\vert^2u\vert\vert_{p,\mu}$ is handled with the help of lemma~\ref{main-2_bound_2}, which says that
\begin{equation}\label{main_thm_2_bound_1}
\vert\vert \vert d\phi\vert^2u\vert\vert_{p,\mu} \leq
\widetilde{K}\vert\vert u\vert\vert_{W_{\mu}^{2,p}(M)},
\end{equation}
where $\widetilde{K}> 0$ is a constant.

Combining \eqref{main_thm_2_bound_1} and~(\ref{E:almost-done-1}) and keeping in mind the definition of $\kappa_1$, gives
\begin{equation*}
\vert\vert (d\phi)^{\#}u\vert\vert_{p,\mu} +
\vert\vert X(u)\vert\vert_{p,\mu} \leq \kappa_1\bigg{[}
(2J_{\tau}+2G_{\tau})\bigg{(}\frac{C_{{\vep}_0}}{\vep_0} + \beta_2 + 1 + \widetilde{K}
\bigg{)}
\bigg{]}
\vert\vert u\vert\vert_{\mathcal{D}},
\end{equation*}
for all $u\in\mcomp$, where $J_{\tau}$ and $G_{\tau}$ are as in~(\ref{E:j-g-tau}).

This establishes \eqref{main_thm_2_bound} and thus proves part (ii) of
theorem \ref{T:main-2}. $\hfill\square$

\section{Proof of Theorem~\ref{T:main-3}}\label{SS:FK-proof-1}
In our discussion we closely follow the proof of theorem 1.3 in~\cite{Gu-10} for $\nabla^{\dagger}\nabla +V$ in $L^2(\vbn)$, with the following modifications: (i) we need to account for the presence of $Z$ in~(\ref{E:stoch-diff-1}) and (ii) we need to replace the references to $L^2$-type statements about abstract $C_0$-semigroups (and their generators) by the corresponding $L^p$-type statements with $1<p<\infty$.

As in in~\cite{Gu-10}, we first tackle the simpler case $0\leq V\in C(\End \vbn)\cap L^{\infty}(\End \vbn)$.
In what follows, we let $L^{0}_{\mu}(\vbn)$ stand for the Borel measurable sections. With the notations of section~\ref{SS:FK-OU}, we define the operators $Q_{t}\colon L^{p}_{\mu}(\vbn)\to  L^{0}_{\mu}(\vbn)$, $t\geq 0$, as follows:
\[
(Q_{t}f)(x):=\mathbb{E}\left[\mathscr{V}^{x}_{t}\slash\slash^{t,-1}_{x}f(Y_{t}(x))\right].
\]
We first verify that for all $t\geq 0$, the operators $Q_{t}$ are bounded (in the sense $L^{p}_{\mu}(\vbn)\to L^{p}_{\mu}(\vbn)$). Although this was done in the proof of proposition 5.2 in~\cite{Mil-Sar-19} for $\nabla^{\dagger}\nabla+ V$ in $L^p(\vbn)$, we include the full argument here, which is based on the use of H\"{o}lder's inequality (with $q$ and $p$ related through $p^{-1}+q^{-1}=1$). Denoting by $\rho_t(x,y)$ the density of the diffusion $Y_t(x)$ with respect to the measure $d\mu$ in~(\ref{E:meas-mu}), we have
\begin{align*}
\|Q_tf\|_{p,\mu}^p &\leq e^{pt\|V\|_{\infty}}\int_{M}\mathbb{E}[|f(Y_t(x))|_{Y_t(x)}]^p\,d\mu(x) \\
           &= e^{pt\|V\|_{\infty}}\int_M\bigg{(} \int_M|f(y)|_y\rho_t(x,y)\,d\mu(y)\bigg{)}^p \,d\mu(x) \\
           &\leq e^{pt\|V\|_{\infty}}\int_M\int_M|f(y)|^p_y\rho_t(x,y)\,d\mu(y)
           \bigg{(}\int_M\rho_t(x,z)\,d\mu(z) \bigg{)}^{p/q}\,d\mu(x) \\
           &= e^{pt\|V\|_{\infty}}\|f\|_{p,\mu}^p,
\end{align*}
where in the first inequality we used the property $|\mathscr{V}^{x}_{t}|\leq e^{t\|V\|_{\infty}}$ (see the estimate for the expression~(15) in~\cite{Gu-10}) and in the last equality we used the assumption (SC), that is, the property $\int_{M}\rho_t(x,z)\,d\mu(z)=1$ for all $t>0$ and all $x\in M$.

The next step is to justify the following formula for all $\Psi\in \vcomp$:
\begin{align}\label{E:FKI-1-a}
&d\left(\slash\slash^{t,-1}_{x}\Psi(Y_{t}(x))\right)=\slash\slash^{t,-1}_{x}\sum_{j=1}^{l}(\nabla_{A_{j}}\Psi)(Y_{t}(x))d W_{t}^{j}\nonumber\\
&-\slash\slash^{t,-1}_{x}(\nabla^{\dagger}{\nabla}\Psi-\nabla_Z\Psi)(Y_{t}(x))dt,
\end{align}
where we used the same notations as in~(\ref{E:diff-start-0}) and~(\ref{E:Z-X-PHI}), with $d$ being the It\^o differential.

We point out that in proposition 1.2 of~\cite{Thalmaier-17} the formula~(\ref{E:FKI-1-a}) was established for smooth sections of a vector bundle $\vbn$ over $M$, viewed as an associated bundle of the orthonormal frame bundle $\pi\colon O(M)\to M$, with structure group $O(n)$ and with a covariant derivative $\nabla$ on $\vbn$ induced from the Levi--Civita connection on $TM$. The starting point of the approach used in~\cite{Thalmaier-17} is the diffusion $U_t(u)$ obtained by solving an intrinsic analogue of the equation~(\ref{E:stoch-diff-1}) on $O(M)$.

Since we are using an extrinsic construction as in~\cite{Gu-10}, in which one starts from~(\ref{E:diff-start-0}), we will include some explanations on how to get~(\ref{E:FKI-1-a}) using this approach. As in proposition 2.5 of~\cite{Gu-10}, we apply the equation~(\ref{E:stoch-diff-1})
to the function
\begin{equation}\label{E:F-PSI-1}
F_{\Psi}\colon \mathscr{P}(\vbn)\to \CC^{m},\qquad F_{\Psi}(u):=u^{-1}\Psi(\pi (u)),
\end{equation}
where $m=\rank\vbn$,  the notation $\mathscr{P}(\vbn)$ is as in section~\ref{SS:FK-OU}, and $\pi\colon \mathscr{P}(\vbn)\to M$ is the projection map.

To transform~(\ref{E:stoch-diff-1}) into (\ref{E:FKI-1-a}) we use the definitions~(\ref{E:F-PSI-1}) and~(\ref{E:spt-1}) and the following two formulas recalled in lemma 2.4 of~\cite{Gu-10}:
\begin{equation*}
Z^*F_{\Psi}=F_{\nabla_{Z}\Psi},\qquad \displaystyle\sum_{j=1}^{l}(A^{*}_j)^2F_{\Psi}=F_{-\nabla^{\dagger}{\nabla}\Psi}.
\end{equation*}
The first formula explains the presence of the item $\nabla_Z\Psi$ in (\ref{E:FKI-1-a}), while the item $\nabla^{\dagger}{\nabla}\Psi$ is a consequence of the second formula and the switch from the Stratonovich differential $\underbar{d}$ to the It\^o differential $d$.

Having established~(\ref{E:FKI-1-a}), we use the It\^o product rule and~(\ref{E:def-script-v}) to infer
\begin{align}\label{E:FKI-1-b}
&d(\mathscr{V}^{x}_{t}\slash\slash^{t,-1}_{x}\Psi(Y_{t}(x)))=-\mathscr{V}^{x}_{t}\slash\slash^{t,-1}_{x}(V\Psi)(Y_t(x))dt\nonumber\\
&+\mathscr{V}^{x}_{t}\slash\slash^{t,-1}_{x}\sum_{j=1}^{l}(\nabla_{A_{j}}\Psi)(Y_{t}(x))d W_{t}^{j}\nonumber\\
&-\mathscr{V}^{x}_{t}\slash\slash^{t,-1}_{x}(\nabla^{\dagger}{\nabla}\Psi+\nabla_{(d\phi)^{\sharp}}\Psi-\nabla_{X}\Psi)(Y_{t}(x))dt,
\end{align}
where we also used~(\ref{E:Z-X-PHI}) to rewrite $\nabla_{Z}\Psi$.

Using the assumption (SC) and the relation $\Psi\in \vcomp$, it can be seen (as in the proof of theorem 1.1 of~\cite{Gu-10}) that the second term on the right hand side of~(\ref{E:FKI-1-b}) is an $\vbn_{x}$-valued continuous martingale. Therefore, applying the expectation $\mathbb{E}[\cdot]$ to both sides of~(\ref{E:FKI-1-b}) and using the definition of $Q_t$ we obtain
\begin{equation*}
(Q_t\Psi)(x) = \Psi(x) - \int_0^tQ_sP^{\nabla}\Psi(x)\,ds,
\end{equation*}
for all $\Psi \in \vcomp$, where $P^{\nabla}$ is as in~(\ref{E:def-H}). This leads to the equation

\begin{equation}\nonumber
\frac{d}{dt}(Q_t\Psi) = -Q_tP^{\nabla}\Psi,\qquad Q_0\Psi = \Psi,
\end{equation}
for all $\Psi \in \vcomp$.

From hereon the proof enters the realm of operator theory, where the crucial role is played by the fact that, under the assumptions (F1)--(F3) and the geodesic completeness of $M$, our operator $-\hvmaxp$ generates a (quasi-contractive) $C_0$-semigroup and $\hvmaxp$ has $\vcomp$ as a core.

As $-\hvmaxp$ is the generator for the $C_0$-semigroup $e^{-t\hvmaxp}$, we can use the abstract lemma II.1.3(ii) of \cite{engel-nagel}, to infer
\begin{equation}\nonumber
\frac{d}{dt}\left(e^{-t\hvmaxp}\Psi\right) = -e^{-t\hvmaxp}P^{\nabla}\Psi, \qquad
 \left(e^{-t\hvmaxp}|_{t=0}\right)\Psi = \Psi,
\end{equation}
for all $\Psi \in \vcomp$, where $P^{\nabla}$ is as in~(\ref{E:def-H}).

Comparing the last two differential equations, we get $Q_t\Psi = e^{-t\hvmaxp}\Psi$ for all $\Psi \in \vcomp$, which leads to the equality
$Q_tf = e^{-t\hvmaxp}f$, for all $f \in L^p_{\mu}(\vbn)$. This completes the proof for the case $0\leq V\in C(\End \vbn)\cap L^{\infty}(\End \vbn)$.

For the rest of the proof we use the approximation procedure from sections 3 and 4 in~\cite{Gu-10}. To illustrate the flavor of the argument and to point out a few small changes (due to the presence of the drift term~(\ref{E:stoch-diff-1}) and the $L^p$-environment), we include the explanations for the case $0\leq V \in L^{\infty}(\End \vbn)$.

To distinguish between the expressions $P^{\nabla}$ in~(\ref{E:def-H}) with different potentials $V$, for the reminder of the proof we will use
the notation $P^{\nabla, V}$ instead of $P^{\nabla}$.

Starting with $0\leq V \in L^{\infty}(\End \vbn)$, we can use lemma 3.1 of \cite{Gu-10} to produce a sequence
$0\leq V_k \in  C(\End \vbn)\cap L^{\infty}(\End \vbn)$ satisfying the following two properties:
\begin{equation}\label{E:approx-norm-infty}
|V_k(x)|\leq \|V\|_{\infty}, \quad \textrm{for all }x\in M\textrm{ and }k\in\ZZ_{+},
\end{equation}
and
\begin{equation*}
\|P^{\nabla, V_k}\Psi -P^{\nabla, V}\Psi\|_{p, \mu} \rightarrow 0,
\end{equation*}
for all $\Psi\in\vcomp$, as $k \rightarrow \infty$.  (In~(\ref{E:approx-norm-infty}), the symbol $|\cdot|$  is the norm of a linear operator $\vbn_x\to\vbn_x$.)

Denoting by $\hvmaxp^{(k)}$ the maximal operator in $L_{\mu}^p(\vbn)$ corresponding to $P^{\nabla, V_k}$, we see (by theorem~\ref{T:main-1}) that  $-\hvmaxp^{(k)}$ generates a $C_0$-semigroup in $L_{\mu}^p(\vbn)$, which we denote by $e^{-t\hvmaxp^{(k)}}$, and, furthermore, $\vcomp$ is a common core for $\hvmaxp^{(k)}$ and $\hvmaxp$. This allows us to use the abstract theorem III.4.8 of \cite{engel-nagel}, known as the Kato--Trotter theorem, to infer
\begin{equation*}
e^{-t\hvmaxp^{(k)}}f \rightarrow e^{-t\hvmaxp}f, \quad \textrm{in }L_{\mu}^p(\vbn),
\end{equation*}
for all $f \in L_{\mu}^p(\vbn)$, $1 < p < \infty$.

As $0\leq V_k \in  C(\End \vbn)\cap L^{\infty}(\End \vbn)$, we may apply~(\ref{E:F-K}) to $(e^{-t\hvmaxp^{(k)}}f)(x)$, which allows us to  rewrite the preceding convergence relation as
\begin{equation*}
\mathbb{E}\left[\mathscr{V}^{x}_{k,t}\slash\slash^{t,-1}_{x}f(Y_{t}(x))\right] \rightarrow (e^{-t\hvmaxp}f)(x),
\end{equation*}
in $L_{\mu}^p(\vbn)$, where $\mathscr{V}^{x}_{k,t}$ satisfies~(\ref{E:def-script-v}) with $V=V_k$.

It remains to show that
\begin{equation*}
\mathbb{E}\left[\mathscr{V}^{x}_{k,t}\slash\slash^{t,-1}_{x}f(Y_{t}(x)) \right] \rightarrow \mathbb{E}\left[\mathscr{V}^{x}_{t}\slash\slash^{t,-1}_{x}f(Y_{t}(x))\right],
\end{equation*}
as $k\to\infty$.

The last convergence relation can be inferred by using the dominated convergence theorem together with the following observations:
\begin{equation*}
\mathscr{V}^{x}_{k,t}\slash\slash^{t,-1}_{x}f(Y_{t}(x))\rightarrow\mathscr{V}^{x}_{t}\slash\slash^{t,-1}_{x}f(Y_{t}(x)), \quad\mathbb{P}-a.s.,
\end{equation*}
which follows from lemma 3.2 in~\cite{Gu-10} (with the the diffusion $Y_t$ in place of the usual Brownian motion on $M$);
\begin{equation*}
|\mathscr{V}^{x}_{k,t}\slash\slash^{t,-1}_{x}f(Y_{t}(x))|_{\vbn_x}\leq e^{t\|V\|_{\infty}}|f(Y_t(x))|_{\vbn_{Y_t(x)}},
\end{equation*}
which is a consequence of~(\ref{E:approx-norm-infty}) and the definition of $\mathscr{V}^{x}_{k,t}$ (see the estimate of the expression~(15) in~\cite{Gu-10});
\begin{equation}\label{E:obs-3}
\mathbb{E}\left[|f(Y_t(x))|_{\vbn_{Y_t(x)}}\right]=(e^{-t\hvmaxp^{d,0}})|f(x)|<\infty,
\end{equation}
where $\hvmaxp^{d,0}$ is the maximal operator in $L_{\mu}^p(M)$ corresponding to $P^{d}$ in~(\ref{E:def-H-d}) with $V=0$.

The observation~(\ref{E:obs-3}) is a consequence of the following properties: (i) under our hypotheses, $-\hvmaxp^{d,0}$ generates a $C_0$-semigroup in $L_{\mu}^p(M)$;  (ii) the function $x\mapsto |f(x)|$ belongs to $L_{\mu}^p(M)$. Thus, the formula (\ref{E:F-K}) is applicable for $(e^{-\hvmaxp^{d,0}})|f(x)|$. This concludes the proof of the theorem for the case $0\leq V \in L^{\infty}(\End \vbn)$.

In the general case $0\leq V \in \lloc^{\infty}(\End\vbn)$, we can just follow the approximation procedure used in the proof of theorem 1.3 in \cite{Gu-10}, keeping in mind that we need to replace the usual Brownian motion on $M$ by the diffusion $Y_t(x)$, the space $L^2(\vbn)$ by $L_{\mu}^p(\vbn)$, the expression $\nabla^{\dagger}\nabla +V$ by $P^{\nabla}$, and $d^{\dagger}d$ by $P^{d}$ with $V=0$. $\hfill\square$

\begin{appendices}
\section{Harmonic Coordinates}\label{appd_1}

In this appendix, we collect the basic facts about harmonic coordinates on Riemannian manifolds. Most of this material is taken from \cite{hebey} and Appendix B in
\cite{GP-2015}.

Fix a $n$-dimensional Riemannian manifold $(M, g)$ without boundary, and let $\nabla$ denote the associated Levi-Civita connection.

\begin{definition}
Let $(U, \psi)$ be a coordinate chart of $M$ with associated coordinates $x_i$. We say that the coordinates define a harmonic coordinate chart if $\Delta x_i = 0$ for all
$1 \leq i \leq n$, where $\Delta$ is the scalar Laplacian on $M$.
\end{definition}

Given a point $x \in M$, there exists a harmonic coordinate chart $(U, \psi)$
about $x \in M$. This follows from the classical fact that there always exists a smooth solution of $\Delta u = 0$ with $u(x)$ and $\partial_iu(x)$ prescribed. The solutions
$y_i$ of
\begin{equation*}
 \begin{cases}
      \Delta y_i = 0 \\
      y_i(x) = 0 \\
      \partial_jy_i(x) = \delta_{ij}
    \end{cases}
\end{equation*}
are then the desired harmonic coordinates. Since composing with linear transformations does not affect the fact that coordinates are harmonic, one can always choose the harmonic coordinate system about $x$ so that in these coordinates
$g_{ij}(x) = \delta_{ij}$ for any $i$, $j$.

\begin{definition}\label{D:def-a-2}
Let $x \in M$, $Q \in (1, \infty)$, $k \in \mathbb{N}_{\geq 0}$,
$\alpha \in (0,1)$. The $C^{k,\alpha}$-harmonic radius of $M$ with accuracy $Q$ at $x$ is defined to be the largest real number $r_{Q, k, \alpha}(x)$ with the following property: The geodesic ball $B_{r_{Q, k, \alpha}(x)}(x)$, centred at $x$, admits
a centred harmonic coordinate chart
\begin{equation*}
\psi : B_{r_{Q, k, \alpha}(x)}(x) \rightarrow \R^n
\end{equation*}
such that
\begin{equation*}
Q^{-1}(\delta_{ij}) \leq (g_{ij}) \leq Q(\delta_{ij})
\end{equation*}
in $B_{r_{Q, k, \alpha}(x)}(x)$ as symmetric bilinear forms, and for all
$1 \leq i, j \leq n$
\begin{align*}
&\sum_{\beta \in \mathbb{N}^n, 1\leq \vert\beta\vert\leq k}
r_{Q, k, \alpha}(x)^{\vert\beta\vert} \sup_{x' \in B_{r_{Q, k, \alpha}(x)}(x)}
\vert \partial_{\beta}g_{ij}(x')\vert \\
&+
\sum_{\beta \in \mathbb{N}^n, \vert\beta\vert = k} r_{Q, k, \alpha}(x)^{k + \alpha}
\sup_{x', x'' \in B_{r_{Q, k, \alpha}(x)}(x), x' \neq x''}
\frac{\vert \partial_{\beta}g_{ij}(x') - \partial_{\beta}g_{ij}(x'')\vert}
{d_g(x', x'')^{\alpha}} \\
&\leq Q - 1,
\end{align*}
where $d_g$ denotes the distance associated to the Riemannian metric $g$.
\end{definition}

A coordinate system as defined in the above definition is known as a
$C^{k,\alpha}$-harmonic coordinate system with accuracy $Q$ on
$B_{r_{Q, k, \alpha}(x)}(x)$.

The main result about the harmonic radius is that control on the first $k$ derivatives of the Ricci curvature, together with control on the injectivity radius implies control
on $r_{Q, k+1, \alpha}(x)$.

For any open set $\Omega \subseteq M$ and any $\vep > 0$ let
\begin{equation*}
\Omega(\vep) := \{ x : x \in M, d_g(x, \Omega) < \vep \}
\end{equation*}
be the $\vep$-neighbourhood of $\Omega$.

We now recall theorem 1.2 from \cite{hebey}:

\begin{theorem}
Let $Q \in (1,\infty)$, $\alpha \in (0,1)$. Fix an open subset $\Omega \subset M$ and numbers $k \in \mathbb{N}_{\geq 0}$, $\vep > 0$, $r > 0$, $c_0, \ldots, c_k > 0$
with
\begin{equation*}
\vert \nabla^j\ric_{M}(x)\vert_x \leq c_j, r_{\textrm{inj}}(x) \geq r \text{ for all }
x \in \Omega(\vep), j \in \{0,\ldots,k\}.
\end{equation*}
Then there is a constant $C = C(n, Q, k, \alpha, \vep, r, c_1,\ldots,c_k) > 0$,
such that for all $x \in \Omega$ one has $r_{Q, k+1, \alpha}(x) \geq C$.
\end{theorem}

This result implies $r_{Q,j, \alpha}(x) > 0$ for all $x \in M$. One calls the number
\begin{equation}\label{notation-a-1}
r_{Q,j,\alpha}(M) := \inf_{x \in M}r_{Q,j, \alpha}(x)
\end{equation}
the $C^{k,\alpha}$-harmonic radius for the accuracy $Q$.
\end{appendices}

\section*{Acknowledgements}
HS wishes to thank Lashi Bandara and Alex Blumenthal for discussions on covering theorems on Riemannian manifolds. HS acknowledges support from the Australian Research Council via grant FL170100020. OM would like to thank Batu G\"uneysu for conversations on Feynman--Kac formula and stochastic completeness.

\end{document}